\documentclass[12pt,reqno]{amsart}

\usepackage{adjustbox}
\usepackage{array}% for \extrarowheight. http://ctan.org/pkg/array
\usepackage{multirow} % for multirows in tables
\usepackage{float} % for using H with figure placement
\usepackage{extpfeil} % for xmapsto
\usepackage{graphicx} % for rotating text
\usepackage{caption} % to have figures in minipages
\usepackage{setspace}
\usepackage{xcolor} % define some more colors (AO, 210402)
\usepackage{amsmath,amscd,amssymb,amsfonts}
\usepackage{amsthm}
\usepackage{color,enumerate,accents,csquotes,placeins,mathtools}
\usepackage{multicol}
\usepackage{array}

\usepackage{marginnote}

\usepackage{tikz-cd}
\usepackage{enumitem}
\usepackage{mathtools}
\usepackage{upgreek} % for Ramanujan tau
\mathtoolsset{showonlyrefs} % Only show equation numbers for referenced equations
\usepackage[margin=2.7cm]{geometry}
\usepackage[colorlinks=false,hidelinks,hyperindex,pagebackref=false, citecolor=cyan,linkcolor=cyan]{hyperref} % load this last 

\newenvironment{highlight}{\color{magenta}}{\ignorespacesafterend}
\newtheorem{theorem}{Theorem}[section]
\newtheorem*{theorem*}{Theorem}
\newtheorem{lemma}[theorem]{Lemma}
\newtheorem*{lemma*}{Lemma}
\newtheorem{proposition}[theorem]{Proposition}
\newtheorem*{proposition*}{Proposition}
\newtheorem{corollary}[theorem]{Corollary}
\newtheorem*{corollary*}{Corollary}

\newtheorem*{prop*}{Proposition}

\theoremstyle{definition} % No italics in this environment
\newtheorem{definition}[theorem]{Definition}

\newtheorem*{warning*}{Warning}

\newtheorem*{example*}{Example}
\newtheorem{example}{Example}

\newtheorem*{conjecture*}{Conjecture}

\newtheorem*{question*}{Question}
\newtheorem*{questions*}{Questions}

\theoremstyle{remark} % No italics in this environment
\newtheorem{remark}{Remark}
\newtheorem*{remark*}{Remark}

\newcommand{\NN}{\mathbb Z_{\geq 0}}
\newcommand{\NNpos}{\mathbb Z_{>0}}
\newcommand{\ZZ}{\mathbb Z}
\newcommand{\QQ}{\mathbb Q}
\newcommand{\RR}{\mathbb R}
\newcommand{\CC}{\mathbb C}
\renewcommand{\AA}{\mathbb A} % \AA is taken by something.

\newcommand{\PP}{\mathbb P}
\newcommand{\FF}{\mathbb F}

\newcommand{\Hom}{\operatorname{Hom}} % hom-set
\newcommand{\tik}{\[\begin{tikzcd}}
\newcommand{\zcd}{\end{tikzcd}\]}

\newcolumntype{L}[1]{>{\raggedright\let\newline\\\arraybackslash\hspace{0pt}}m{#1}}
\newcolumntype{C}[1]{>{\centering\let\newline\\\arraybackslash\hspace{0pt}}m{#1}}
\newcolumntype{R}[1]{>{\raggedleft\let\newline\\\arraybackslash\hspace{0pt}}m{#1}}

\newcommand{\defn}[1]{\textsf{#1}}

\newcommand{\Mod}[1]{\ (\mathrm{mod}\ #1)}

\definecolor{cl1}{HTML}{ee82ee} 
\definecolor{cl2}{HTML}{4b0082} 
\definecolor{cl3}{HTML}{0000ff} 
\definecolor{cl4}{HTML}{008000} 
\definecolor{cl5}{HTML}{ffff00} 
\definecolor{cl6}{HTML}{ffa500} 
\definecolor{cl7}{HTML}{ff0000} 
\definecolor{cl8}{HTML}{ff00a2} 
\definecolor{mydarkerblue}{HTML}{3B87C4} 

\newcommand{\vertrowspace}{.6em}
\renewcommand{\arraystretch}{1.15}

\newcommand{\Irr}{\operatorname{Irr}}

\newcommand{\PT}[1]{T_{#1}} % poset of types
\newcommand{\IT}[1]{I_{#1}} % incidence algebra on poset of types
\newcommand{\PS}{\mathsf{P\Lambda}} % ring of polysymmetric functions
\newcommand{\wt}{\mathrm{deg}} % weight
\newcommand{\arr}{a}
\newcommand{\arrsqf}{e}

\newcommand{\len}[1]{\mathrm{len}{#1}} % length
\newcommand{\IPG}{\Sigma} % infinite polysymmetric group
\newcommand{\Rcc}[1]{[\![#1]\!]} % class in ring of char. cycles
\newcommand{\ZZgeq}{\ZZ_{ \geq 0}} % non-negative integers 
\newcommand{\Conf}{\mathrm{Conf}} % configuration space (asvin)
\newcommand{\GrLR}{\mathsf{GrLaRing}} % category of graded $\lambda$-rings
\newcommand{\SymProd}[2]{S^{#2}{#1}} % symmetric product
\newcommand{\Var}[1]{\mathrm{Var}_{#1}} % category of varieties over #1
\newcommand{\KV}[1]{K_0(\Var{#1})} % Groth. ring of varieties over #1
\newcommand{\RC}[1]{C(\Var{#1})} % ring of characteristic cycles 
\def\multiset#1#2{\ensuremath{\left(\kern-.3em\left(\genfrac{}{}{0pt}{}{#1}{#2}\right)\kern-.3em\right)}}

\newcommand{\Ep}[1]{{E}_{#1}} % E-polynomial / Serre polynomial

\newcommand{\Epf}[2]{E_{#2}^{\vec{#1}}} % n-powerfree monomials 
\newcommand{\Uc}{\mathcal U}
\newcommand{\Sc}{\mathcal X}

\DeclareMathOperator{\Exp}{Exp}
\newcommand{\mm}{\xi} % general motivic measure

\newcommand{\papertitle}{\Large{Polysymmetric functions and motivic measures of\\[0.4em] configuration spaces}}
\newcommand{\shortpapertitle}{{POLYSYMMETRIC FUNCTIONS AND CONFIGURATION SPACES}}

\title[\shortpapertitle]{\papertitle}

\author{Asvin G}
\address{
\parbox{0.5\linewidth}{
    Department of Mathematics\\
    University of Wisconsin--Madison\\
    }
}
\email{gasvinseeker94@gmail.com}

\author{Andrew O'Desky}
\address{
\parbox{0.5\linewidth}{
    Department of Mathematics\\
    Princeton University\\
    }
}
\email{andy.odesky@gmail.com}

\thanks{A.O. was partially supported by the NSF under Award~No.~DMS-2103361.}

\date{\today}

\begin{document}

\begin{abstract} %{*
We introduce a generalization of symmetric functions 
    and apply the resulting theory 
    to compute the class 
    in the Grothendieck ring of varieties 
    of the space of geometrically irreducible hypersurfaces 
    of a fixed degree in projective space. 
\end{abstract} %*}


\maketitle

\section{Introduction} %{*
Over the last few decades, there have been many fruitful connections 
    between combinatorics and algebraic geometry. 
Some recent highlights of applying 
    algebraic geometry to combinatorics are 
    \cite{adiprasito2018hodge}, \cite{braden2020singular} 
    and \cite{elias2020hodge}. 
A different strand of research, applying combinatorics to algebraic geometry, 
    begins with Grothendieck and Atiyah's definition of a \emph{lambda ring}, 
    axiomatizing the symmetric and alternating tensor power operations on 
    the Grothendieck ring of vector bundles. 
%A speculative hypothesis providing some philosophical motivation is the notion of a "Field with one element" over which geometry should reduce to combinatorics which perhaps explains why various algebro-geometric notions such as hodge theories have purely combinatorial incarnations. In the other direction, there has been a proposal \cite{borger} to consider a Lambda ring structure on a ring as "
%\emph{descent data} to $\mathbb F_1$". From this perspective, the ring of symmetric functions $\PSg$ plays a crucial role as the universal initial free Lambda ring in one variable. The recent stunning work of Scholze and collaborators can be seen as a p-adically completed version of this philosophy - a delta ring precisely encodes the data associated to the Adams operation corresponding to the prime $p$.
In this paper, we encounter 
    a new application of combinatorics to algebraic geometry 
    extending this strand of research. 

Let $\Lambda$ denote the ring of symmetric functions, 
and let $\Lambda_{(d)}$ 
be a copy of $\Lambda$ whose grading 
has been multiplicatively scaled by $d$. 
The aim of this paper is to introduce and give some applications 
    of the graded $\lambda$-ring 
    with one free generator in each positive degree, 
\begin{equation} %{*
\PS = 
\Lambda_{(1)}
\otimes \Lambda_{(2)}
\otimes \Lambda_{(3)} \otimes \cdots, 
\end{equation} %*}
    which we call the ring of \defn{polysymmetric functions}. 
This ring is closely connected with the combinatorics 
    of splitting types and extends much of 
    the theory of symmetric functions. 

\subsection{Arrangement numbers} %{*

The ring $\PS$ has two distinguished bases which serve 
    a foundational role in what follows. 
Let $x_{\ast \ast}= \{x_{dj}\}_{d,j = 1}^\infty$ 
be a set of indeterminates where $x_{dj}$ has degree $d$ 
and identify $\Lambda_{(d)}$ with 
    symmetric functions in the variables 
    $x_{d \ast}$. 
We say the monomial $x_{d_1j_1}^{m_1}\cdots x_{d_rj_r}^{m_r}$ 
with $m_1,\ldots,m_r$ positive and $r$ minimal has 
    type $\tau=\vec{d}^{\vec{m}}=
    (d_1^{m_1} d_2^{m_2} \cdots d_r^{m_r})$ 
    (regarded as an unordered multiset of tuples $(d_k,m_k)$). 
The \defn{monomial polysymmetric function $M_\tau$ of type $\tau$} 
is the sum %in $\ZZ[\![x_{\ast\ast}]\!]$ 
over all monomials of type $\tau$. 
Next we set 
\begin{equation} %{*
H_d = %H_{d^1} := 
\sum_{\substack{\text{monomials $f$}\\\text{of degree $d$}} } 
   f .
\end{equation} %*}
The \defn{complete homogeneous polysymmetric function $H_\tau$ of type $\tau$} is 
\begin{equation} %{*
H_{\tau} = 
\prod_{d^m \in \tau} H_{d}(x_{\ast\ast}^m).
\end{equation} %*}
We will see that 
\begin{equation} %{*
    \PS = \bigoplus_\tau \ZZ M_\tau,  \qquad
    \PS  \otimes \QQ=  \bigoplus_\tau \QQ H_\tau 
    %= \ZZ[\![ \{H_{d^m}\}_{d, m = 1}^\infty]\!].
\end{equation} %*}
and construct a partial order on types such that 
\begin{equation}%\label{eqn:defnc} %{*
    H_\lambda = \sum_{\tau } \arr_{\tau\lambda} M_\tau
    = \sum_{\tau \leq \lambda } \arr_{\tau\lambda} M_\tau
    \qquad (\arr_{\tau\lambda} \in \ZZ_{ \geq 0}).
\end{equation} %*}
The integer $\arr_{\tau\lambda}$ for splitting types 
    $\tau=\vec{d}^{\vec{m}}$ and $\lambda=\vec{c}^{\vec{n}}$ 
    has a simple combinatorial interpretation 
    as the number of $\ZZgeq$-matrices $A$ satisfying 
\begin{equation}%\label{eqn:arrdefnmatrices}
A\vec{n} = \vec{m} \quad\text{and}\quad A^T \vec{d} = \vec{c}. 
\end{equation}
We prove the following identity satisfied by arrangement numbers
\begin{equation} %{*
    a_{\tau \lambda} = a_{\lambda^t\tau^t}
\end{equation} %*}
where $\tau^t=(\vec{d}^{\vec{m}})^t= \vec{m}^{\vec{d}}$. 
We show in \S\ref{sec:polysymmetricconstructions} 
    that this involution on types 
    clarifies some classical identities 
    for transition matrices between bases of symmetric functions.

%*}

\subsection{Coefficients of the plethystic logarithm}\label{sec:logarithm} %{*

The $H$ and $M$ bases of $\PS$ arise naturally as 
    the coefficients of the plethystic exponential and logarithm. 
Following \cite{mozgovoy2019}, 
a \emph{plethystic exponential} on a commutative ring $R$ 
is a group isomorphism 
\begin{equation} %{*
    \Exp \colon (tR[\![t]\!],+) \to (1+tR[\![t]\!],\times)
\end{equation} %*}
satisfying the following properties 
\begin{enumerate} %{*
    \item $\Exp(t) = (1-t)^{-1}$, 
    \item $\Exp(at) = 1+at + O(t^2)$, 
    \item $\Exp(f(t^n)) = \Exp(f(t))|_{t \mapsto t^n}$, 
    \item for every $k \geq 0$ there exists $n \geq 0$ 
        such that the $k$-jet of $\Exp(f)$ is determined 
        by the $n$-jet of $f$. 
\end{enumerate} %*}
%There is a bijective correspondence between 
%    plethystic exponentials 
%    and $\lambda$-ring structures. 

In order to study plethystic exponentials in 
    a universal setting 
    one may proceed as follows. 
To compute $\Exp(ut)$ for some element $u$ of 
    a $\lambda$-ring $R$, 
    it suffices to compute it once and for all 
    by taking $u = h_1$ in the $\lambda$-ring $\Lambda$. 
The ring $\Lambda$ is the free $\lambda$-ring on 
    one generator $h_1$, 
    meaning that for any element $u$ of 
    any special $\lambda$-ring $R$ 
    there is a unique $\lambda$-ring homomorphism 
    $$I_u \colon \Lambda \to R$$ 
    sending $h_1$ to $u$. 
The universal expression for $\Exp(ut)$ is given by 
\begin{equation}\label{eqn:FundamentalIdentityPlethysm} %{*
    \Exp(h_1t)=
    1+h_1t+h_2t^2+h_3t^3+\cdots
    %\sum_{d=0}^\infty h_d t^d.
\end{equation} %*}
where $h_1,h_2,\ldots$ are 
the complete homogeneous symmetric functions, 
    in the sense that $\Exp(ut)$ 
    can be evaluated by applying $I_u$ 
    to the right-hand side. 
It is natural to ask for an analogous identity 
    for a \emph{general} element 
    $u_1t+u_2t^2+\cdots\in tR[\![t]\!]$ 
    in place of $ut$. 
    %(e.g.~see \eqref{eqn:introExp1} below). 
In order to be universal, 
    each $u_d$ 
    should generate an independent copy of $\Lambda$ 
    in some larger universal $\lambda$-ring. 
The minimal such $\lambda$-ring is $\PS$ 
    which is free on the generators 
    $M_d=h_1 \in 
    \Lambda_{(d)} \subset \PS$, 
    and we may regard the expansion of 
\begin{equation} %{*
    \Exp(M_1 t + M_2 t^2+ 
    M_3 t^3+\cdots) \in \PS[\![t]\!]
\end{equation} %*}
as the \emph{universal plethystic exponential} 
    in the same sense that 
    $$\exp(t) = 1+t+\tfrac12 t^2+\cdots \in \QQ[\![t]\!]$$ 
    is the universal exponential function. 
We will show that the following identity holds:
\begin{equation} %{*
    \mathrm{Exp}(M_1t+M_2t^2+M_3t^3+\cdots)=
    1+H_1t+H_2t^2+H_3t^3+\cdots.
\end{equation} %*}
In other words, 
    the $d$th coefficient of the universal plethystic exponential 
    is the $d$th complete homogeneous polysymmetric function $H_d$, 
    and the $d$th coefficient of the universal plethystic logarithm 
    is the $d$th monomial polysymmetric function 
    $M_d$. 
    
We use this to give a combinatorial interpretation 
    of a formula of Getzler--Kapranov. 
%Let $\Exp$ be a plethystic exponential,  
%$\psi_k$ the $k$th Adams operation,  
%and assume the corresponding $\lambda$-ring is special. 
When $R$ is a special $\lambda$-ring, 
Getzler and Kapranov prove a formula 
    \cite[Prop.~8.6]{modularoperads} 
    for the plethystic logarithm 
\begin{align} %{*
    \mathrm{Log}=\Exp^{-1}\colon (1+tR[\![t]\!],\times) &\to (tR[\![t]\!],+)\\
    1+x_1t+x_2t^2+\cdots&\mapsto u_1t+u_2t^2+\cdots.
\end{align} %*}
For a partition $\lambda = (1^{n_1}\cdots m^{n_m}) \vdash m$ 
set $x^\lambda=x_1^{n_1} \cdots x_m^{n_m}$ 
and $\ell = n_1+\cdots+n_m$. 
Their formula shows that the $d$th coefficient of 
$\mathrm{Log}$ is 
\begin{equation}\label{eqn:fdck} %{*
    u_d=
\sum_{d=km}
\sum_{\lambda \vdash m}
    (-1)^{\ell-1}\frac{\mu(k)}{k\ell}
    \binom{\ell}{n_1,\ldots,n_m}
\psi_{k}
x^\lambda
\end{equation} %*}
where $\psi_k$ is the $k$th Adams operation. 
We prove that 
    these coefficients are transition matrix coefficients 
    for the change of basis $M \to H$ in $\PS$. 
Let us call a type $\tau$ \defn{$k$-pure} 
if $b^c \in \tau \implies c = k$, 
and \defn{mixed} otherwise. 
There is a bijection between $k$-pure types of degree $d$ 
    and partitions of degree $m=d/k$ given by 
$\tau = \lambda^k= (b_1^k \cdots b_r^k) \leftrightarrow 
    \lambda : b_1\geq \cdots \geq b_r$. 

\begin{theorem}\label{thm:FormulaForCoefficients}
Let $\arr^{-1}$ be the inverse of 
    the arrangement numbers $a$ in the incidence algebra 
    on types of degree $d$. 
Then 
%In particular, 
%\begin{equation} %{*
%    M_d 
%    = \sum_{\tau} \arr^{-1}_{\tau d} H_\tau.
%\end{equation} %*}
\begin{enumerate} %{*
\item $\arr^{-1}_{\tau d}$ vanishes if $\tau$ is mixed. 
\item If $\tau = \lambda^k$ is $k$-pure, then 
\begin{equation}\label{eqn:invarrformula} %{*
\arr^{-1}_{\tau d} = 
    (-1)^{\ell-1}\frac{\mu(k)}{k\ell}
    \binom{\ell}{n_1,\ldots,n_m}
\end{equation} %*}
where $\lambda = (1^{n_1}\cdots m^{n_m}) \vdash m=d/k$. 
\end{enumerate} %*}
\end{theorem}

In the context of configuration spaces, 
    $H_\tau$ (resp.~$M_\tau$) 
    corresponds to the closed $\tau$ stratum 
    (resp.~open $\tau$ stratum). 
The identity \eqref{eqn:fdck} is limited to 
$H_d$ and $M_d$, corresponding to 
the \emph{maximal} type $\tau=d=d^1$ of degree $d$. 
This theorem incorporates \eqref{eqn:fdck} 
    into the framework of inclusion-exclusion 
    with respect to the poset of types, 
making it possible to  
compute motivic measures 
of {general} strata. 

\begin{remark}[A construction of Specht] %{*
Let $G$ be a finite group and let $R(G_n)$ be 
    the representation ring of $G_n =G \wr \Sigma_n$ 
    where $\Sigma_n$ is the symmetric group on $n$ letters. 
%There are natural inclusions 
%$G\wr S_m\times G\wr S_n \to G \wr S_{m+n}$. 
%If $u \in R(G_m)$ and $v \in R(G_n)$ then 
%    the Cartesian product $u \times v$ is an element of 
%    $R(G_m \times G_n)$. 
Induction along 
$G_m\times G_n \subset G_{m+n}$ 
    makes $R(G) = \oplus_{n} R(G_n)$ into a commutative ring 
    where $uv = \mathrm{ind}_{G_m\times G_n}^{G_{m+n}}(u \times v)$, 
    and Specht constructed a ring isomorphism 
    \cite[Ch.~I, Appendix~B]{MR3443860} 
\begin{equation} %{*
    R(G) \xrightarrow{\sim} 
    \bigotimes_{c \in Conj(G)} \Lambda.
\end{equation} %*}
For polysymmetric functions, there is an isomorphism 
    $R(\mathrm{SL}_2) \xrightarrow{\sim}\PS$. 
\end{remark} %*}

    %*}

\subsection{Cohomology of the space of irreducible hypersurfaces} %{*

A motivating example for the geometric applications of 
the ring $\PS$ is the following problem.
Let 
    $\Irr_{d}=\Irr_{d,n}$ 
    denote the variety 
    of geometrically irreducible hypersurfaces 
    of degree $d$ in projective $n$-space. 
Recent work of Chen~\cite{chen} 
    shows that the cohomology of $\Irr_{d}$ 
    stabilizes as $d$ or $n$ 
    goes to infinity. 
%Set $N_{d,n}= \dim \Irr_{d,n}$. 
Chen has shown that for any cohomological degree 
$k$ in the stable range 
$2(\dim \Irr_{d,n}-\dim \Irr_{d,n-1}+n)
    \leq k \leq 
    2\dim \Irr_{d,n}$, 
    the $k$th compact cohomology group is 
%=\binom{d+n}{n}-1$, 
\begin{equation} %{*
    %2(N_{d,n}-N_{d,n-1}+n)
    %\leq k \leq 
    %2N_{d,n}
    %\implies 
    H^{k}_c(\Irr_{d}(\CC),\QQ) =
    \begin{cases}
        \QQ&\text{$k$ even,}\\
        0&\text{$k$ odd.}
    \end{cases}
\end{equation} %*}
%In the stable range, the cohomology is free and rank one 
%    in even degrees and zero in odd degrees. 
Unfortunately this only accounts for a vanishing proportion 
    of cohomological degrees 
    as $d \to \infty$. 

In a different direction, 
    the compact Euler characteristic of 
    the related space $\Irr_{d}^\RR$ of 
    $\RR$-irreducible degree $d$ real hypersurfaces 
    was computed by Hyde~\cite{hyde2020euler}. 
Hyde showed that  
\begin{equation} %{*
    \chi(\Irr_{d}^\RR) = 
    \begin{cases}
        b_k& \text{if $d=2^k$,}\\
        0&\text{otherwise.}
    \end{cases}
\end{equation} %*}
where $b_k \in \{0,1,-1\}$ is defined by the following rule:  
any positive integer $n$ may be uniquely expressed as 
    an alternating sum of descending powers of $2$, 
$$
n = 2^{k_{2m}}-2^{k_{2m-1}}+2^{k_{2m-2}}-\cdots+2^{k_2}-2^{k_1},
$$
for integers $k_{2m} > \cdots > k_1  \geq 0$ 
%Set $n=\sum_{k\geq 0}b_k2^k$ with $b_k \in \{0,-1,1\}$. 
and $b_k$ is the coefficient of $2^k$ 
in this expression 
if $k \in \{k_1,\ldots,k_{2m}\}$ 
or $b_k = 0$ otherwise. 
This formula suggests the cohomology of $\Irr_{d}$ 
    does not have a simple description. 

Following Hyde and Chen, 
    our strategy is to use the stratification 
    by splitting types 
    on the projective space of 
    all hypersurfaces. 
%The basic idea of this strategy is to invert 
%    the relations 
%    arising from this stratification 
%    to obtain information about the open strata 
%    from the closed strata. 
Where Hyde used generating function methods 
    and Chen used the spectral sequence 
    associated to this stratification, 
    we will use plethysm and polysymmetric functions 
    to compute the motivic class 
    of $\Irr_{d}$ in 
    the Grothendieck ring of varieties $\KV{}$ 
    over any base field. 
Our formula determines 
    the contribution from unstable cohomological degrees 
    to the number of points of $\Irr_{d}$ over finite fields. 
For the stratum of hypersurfaces with an arbitrary splitting type 
    we give a weaker result, namely a formula for its class 
    in a certain quotient ring of $\KV{}$. 

\begin{theorem}\label{thm:GeomIrred}
The class of $\Irr_{d}$ 
    in the Grothendieck ring of varieties $\KV{}$ is 
    %\note{Fine to have denominators in the formula?}
\begin{equation}\label{eqn:irrdnmotivicformula} %{*
    [\Irr_{d}] = 
\sum_{d=km}
\sum_{\lambda \vdash m}
    (-1)^{\ell-1}\frac{\mu(k)}{k\ell}
    \binom{\ell}{n_1,\ldots,n_m}
\prod_{j=1}^m
\Big( 
    1 + [\AA^1]^{k} + \cdots + 
[\AA^1]^{k\left(\binom{n+j}{j}-1\right)}
\Big)^{n_j} .
%\psi_m [Y_d]. 
\end{equation} %*}
Let $\RC{}$ be the largest $\lambda$-quotient of $\KV{}$ 
    that is a binomial ring. 
Let $\Irr_{\tau}$ be the space of 
    hypersurfaces in projective $n$-space 
    with geometric splitting type $\tau$ 
    and let $\Rcc{\Irr_{\tau}}$ be its class in $\RC{k}$. 
The binomial class $\Rcc{\Irr_{\tau}}$ is zero 
    unless $d^m \in \tau$ implies $d= 1$ in which case 
\begin{equation} %{*
    \Rcc{\Irr_{\tau}} = 
        \binom{\Rcc{\AA^1}(n-1) + 1}{\tau[1^1],\ldots,\tau[1^d]}
\end{equation} %*}
    where $\tau[1^j]$ is 
    the number of occurrences of $1^j$ in $\tau$. 
\end{theorem} 

\begin{figure}[H] %{*
\centering
\begin{equation} %{*
\begin{array}{|r|l|} %{* 
    \hline
    (d,n)&[\Irr_{d}]\text{ as a polynomial in $q = [\AA^1]$}\\
    \hline
    (4,2)&\mathbf{q^{14} + q^{13} + q^{12}} - 2 q^{10} - 2 q^{9} - q^{8} + q^{7} + q^{6}\\
    \hline
    (5,2)&\mathbf{q^{20} + q^{19} + q^{18} + q^{17}} - q^{15} - 3 q^{14} - 3 q^{13} - q^{12} + 2 q^{11} + 3 q^{10} + q^{9} - q^{8} - q^{7}\\
    \hline
    (3,3)&\mathbf{q^{19} + q^{18} + q^{17} + q^{16} + q^{15} + q^{14} + q^{13}} - q^{11} - 2 q^{10} - 3 q^{9} - 2 q^{8} - q^{7} + q^{5} + q^{4}\\
    \hline
%    (2, 4)& \mathbf{q^{14} + q^{13} + q^{12} + q^{11} + q^{10} + q^{9}} - q^{6} - q^{5} - 2 q^{4} - q^{3} - q^{2}\\
%    \hline
\end{array} %*}
\end{equation} %*}
    \caption{Some cases of the theorem. Powers of $q$ in Chen's stable range are in bold.} 
\end{figure} %*}

\begin{remark}
Along the way 
    we prove a general formula for the binomial class     
    of a generalized configuration space 
    (Theorem~\ref{thm:charclassgenconfigsp}) 
    with the help of a useful recurrence 
    (Proposition~\ref{prop:generalizedconfspacerecurrence}) 
    satisfied by their motivic classes. 
\end{remark}

\begin{remark}
By using the $\lambda$-ring $\ZZ[q]$ 
    in place of $\KV{}$, 
    our method also computes the number $M_{d,n}(q)$ 
    of \emph{arithmetically} irreducible 
    degree $d$ hypersurfaces in projective $n$-space over 
    the finite field $\FF_q$, 
    a problem studied by several authors 
\cite{carlitz63}, 
\cite{MR0172872}, 
\cite{MR2516427}, 
\cite{MR3004008}, 
\cite{bodin}, 
\cite{MR2594509}: 
\begin{equation} %{*
    %    M_{d,n}^{\mathrm{a}}(q) = 
    M_{d,n}(q)=
\sum_{d=km}
\sum_{\lambda \vdash m}
    (-1)^{\ell-1}\frac{\mu(k)}{k\ell}
    \binom{\ell}{n_1,\ldots,n_m}
\prod_{j=1}^m
\Big( 
    1 + q + \cdots + 
    q^{\left(\binom{n+j}{j}-1\right)}
\Big)^{n_j} .
    \end{equation} %*}  
When $n = 1$ this recovers the classical formula of Gauss 
    $$\frac{1}{d}\sum_{d=km}\mu(k)q^{m}$$ 
    for the number\footnote{Strictly speaking, 
    in the degenerate case $d=1$ 
    our formula counts the point at infinity 
    while Gauss's does not.} 
    of monic irreducible 
    degree $d$ polynomials in $\FF_q[t]$. 
\end{remark}

    %*}

\subsection*{Acknowledgements} %{* 

\vspace{.5em}

The authors would like to thank 
Suki Dasher, 
Jordan Ellenberg, 
Darij Grinberg, 
Trevor Hyde, 
Lauren\c{t}iu Maxim, 
Victor Reiner, 
Harry Richman for helpful discussions and 
comments on earlier versions of this paper. 
We would also like to thank Steven Sam 
for pointing out to us 
the connection with Specht's dissertation, 
as well as an anonymous referee 
for bringing \cite[Prop.~8.6]{modularoperads} to our attention.

\section{Polysymmetric functions}\label{sec:1} %{*

%{* Definition

Let $x_{\ast \ast}= \{x_{ij}\}_{i,j = 1}^\infty$ 
be a set of indeterminates where $x_{ij}$ has degree $i$. 
Let $\ZZ[\![x_{\ast \ast}]\!]$ denote 
the ring of all formal infinite sums of monomials in 
$x_{\ast\ast}$ with integer coefficients and bounded degree. 
Let $\IPG$ denote the subgroup of permutations of 
$\NNpos^2$ 
fixing the first coordinate.
This group acts on 
$\ZZ[\![x_{\ast \ast}]\!]$ by 
$(\sigma \cdot f)(x_{ij}) = f(x_{\sigma(ij)})$. 
%A $\IPG$-invariant element of $\ZZ[\![x_{\ast \ast}]\!]$ 
%is called a {polysymmetric function}. 
The \defn{ring of polysymmetric functions} is 
%$\PS \subset \ZZ[\![\{x_{i,j}\}]\!]$ 
the subring 
$$\PS_\ZZ = \ZZ[\![x_{\ast \ast}]\!]^{\IPG}$$ 
of $\IPG$-invariants in $\ZZ[\![x_{\ast \ast}]\!]$. 
%\begin{equation} %{*
%    \PS_\ZZ := \bigoplus_{d = 0}^\infty \PS_{\ZZ,d} \subset \ZZ[\![\{x_{i,j}\}]\!].
%\end{equation} %*}
We set $\PS_R = \PS_{\ZZ} \otimes_\ZZ R$ for a ring $R$ 
    and $\PS = \PS_{\QQ}$ 
    (note that in the introduction we wrote $\PS$ for $\PS_\ZZ$). 
%Let $\PS_{\ZZ,d}$ denote the homogeneous degree-$d$ elements of $\PS_\ZZ$.  

\begin{definition}\label{definition:types}
%Let $d$ be a non-negative integer. 
A \defn{type $\tau$ of degree $d$} (written $\tau \vDash d$) 
is a finite multiset 
$$\tau=\{(p_1,m_1),\ldots,(p_r,m_r)\}$$ 
of ordered pairs of positive integers 
satisfying $p_1m_1+\cdots+p_rm_r = d$. 
We also employ exponential notation, 
writing $\tau=\vec{p}^{\vec{m}}$ or 
$(p_1^{m_1} p_2^{m_2} \cdots p_r^{m_r})$. 
The $p_j$'s and $m_j$'s are called 
    the \defn{degrees} and \defn{multiplicities} of $\tau$. 
We let $\tau[p^m]$ denote 
the number of occurrences of $p^m$ in $\tau$. 
\end{definition}

\begin{remark}
As a general rule, 
types serve the same role for polysymmetric functions as 
partitions do for symmetric functions. 
For instance, 
the Hilbert series of $\PS$ is the generating function for types. 
It satisfies 
$\sum_{d = 0}^\infty 
\mathrm{dim}\,\PS_{d}\,t^d 
    =\sum_{d=0}^\infty
    \#\{\tau \vdash d\} t^d
=
\prod_{k = 1}^\infty 
(1-t^k)^{-\sigma_0(k)}$ 
    %= 1 + t + 3t^2 + 5t^3 + 11t^4 + 17t^5 + \cdots 
where $\sigma_0(k)$ is the number of divisors of $k$. 
\end{remark}

%\begin{proof}
%The fact that the rank of $\PS_{\ZZ,d}$ 
%equals the number of types of degree $d$ is 
%proven below. 
%There is a natural bijection between types of degree $d$ 
%and ordered tuples of partitions $(\lambda_1,\ldots,\lambda_r)$ 
%satisfying 
%$|\lambda_1| + 2 |\lambda_2| + \cdots + r |\lambda_r| = d$. 
%The generating function for the latter set is 
%$\prod_{k=1}^\infty 
%\prod_{i=1}^\infty 
%    (1-t^{ik})^{-1}$ 
%and collecting terms obtains 
%$\prod_{k = 1}^\infty 
%(1-t^k)^{-\sigma_0(k)}$. 
%\end{proof}
%
%\begin{remark}
%Tuples of partitions (multipartitions) arise 
%in the representation theory of affine Lie algebras. 
%The difference between types and multipartitions is 
%how their degree is defined: 
%under the bijection between types and multipartitions 
%    in the proof of the proposition, 
%    a multipartition 
%$(\lambda_1,\ldots,\lambda_r)$ corresponds to 
%a type of degree $d$ when 
%$d=|\lambda_1| + 2 |\lambda_2|+ \cdots + r |\lambda_r|$ 
%    whereas its degree 
%    as a multipartition is defined as 
%$|\lambda_1| + |\lambda_2|+ \cdots + |\lambda_r|$. 
%\end{remark}

\begin{definition}\label{defn:monomialpolysymmetricfunction}
The monomial $x_{i_1j_1}^{m_1}\cdots x_{i_rj_r}^{m_r}$ 
with $m_1,\ldots,m_r$ positive and $r$ minimal has 
    \defn{type $(i_1^{m_1} i_2^{m_2} \cdots i_r^{m_r})$}. 
The \defn{monomial polysymmetric function $M_\tau$ of type $\tau$} 
is the sum %in $\ZZ[\![x_{\ast\ast}]\!]$ 
over all monomials of type $\tau$. 
\end{definition}

\begin{remark}\label{rmk:Mtautosymmetricmlambda}
If $\tau_w$ is the multiset of integers $k$ 
    such that $(w,k) \in \tau$, 
then $M_\tau$ equals 
    $m_{\tau_1}(x_{1\ast})\cdots m_{\tau_d}(x_{d\ast})$ 
    where $m_\lambda$ 
    is the monomial symmetric function of shape $\lambda$. 
\end{remark}

\begin{example}
$$
M_{1^2 1^2 1^3}
=
\sum_{\substack{1 \leq i < j < \infty\\1 \leq k < \infty,\, k \not \in \{i,j\}} }
x_{1i}^2
x_{1j}^2
x_{1k}^3
\quad
\text{and} 
\quad
M_{1^2 1^2 2^3}
=
\sum_{\substack{1 \leq i < j < \infty\\1 \leq k < \infty}}
x_{1i}^2
x_{1j}^2
x_{2k}^3. 
$$
Note that $M_{\tau \cup \tau'} \neq M_{\tau}M_{\tau'}$ (e.g. 
$M_{1 1} \neq M_{1}^2 = 
2M_{1 1} + M_{1^2}$). 
\end{example}

For a graded ring $A = \bigoplus_{d}A_d$ 
and positive integer $n$, 
let $A_{(n)}$ be the graded ring 
supported only in degrees that are multiples of $n$ 
and satisfying $(A_{(n)})_{dn} = A_d$ 
for all $d$. 
There is a natural injection 
$\Lambda_{\ZZ,(n)} \to \PS_\ZZ$ 
induced by mapping the indeterminate $y_j$ to $x_{dj}$. 
These combine into a homomorphism 
\begin{equation}
    \bigotimes_{n=1}^\infty \Lambda_{\ZZ,(n)} 
    \xrightarrow{\sim} \PS_\ZZ
\end{equation}
which is an isomorphism of graded rings. 
In particular, the set $\{M_\tau\}_{\text{$\tau$ type}}$ 
is a linear basis for $\PS_{\ZZ}$ as a free $\ZZ$-module, 
and the rank of $\PS_{\ZZ,d}$ is 
the number of types of degree $d$. 
% direct proof of last statement in paragraph above: 
% 
%Let $f \in \PS_{\ZZ,d}$. 
%Suppose a monomial of type $\tau$  
%    has a nonzero coefficient $c$ in the formal sum expressing $f$. 
%Observe that the action of $\IPG$ on 
%    the set of monomials of type $\tau$ is transitive, 
%so by symmetry every monomial of type $\tau$ 
%    has the same coefficient $c$ in $f$. 
%Then $f-cM_\tau$ has no monomials of type $\tau$. 
%%We continue the process of collecting monomials in $f$ according to their type.  
%Since there are only finitely many types of degree $d$, 
%    we obtain a finite sum 
%    $f = \sum_{\tau \vDash d} c_\tau M_\tau$ for $c_\tau \in \ZZ$. 
%The monomial symmetric functions are clearly linearly independent, 
%and we conclude the set $\{M_\tau\}_{\text{$\tau$ type}}$ 
%    is a linear basis for $\PS_{\ZZ}$. 

\begin{remark}
The tensor product description of $\PS$ implies 
that any linear basis of $\Lambda$ induces a linear basis 
of $\PS$ consisting of pure tensors. 
Among the five polysymmetric bases constructed in this paper --- 
$HEE^+PM$ ---  
    {only} the monomial polysymmetric basis 
    consists of pure tensors 
with respect to one of the 
classical bases of symmetric functions. 
The other polysymmetric bases are non-classical in this sense. 
\end{remark}

%*}

\subsection{Coefficients of the plethystic logarithm}\label{sec:mainformulapf} %{*

Let $R$ be a commutative ring with unity. 

\subsubsection{$\lambda$-rings} %{*

A \defn{$\lambda$-ring structure} for $R$ is a group homomorphism 
$\lambda \colon R \to 1 + tR[\![t]\!]$ 
satisfying $\lambda(r) = 1 + rt + O(t^2)$ for all $r$. 
Let $\lambda_k(r)$ denote the $k$th coefficient of $\lambda(r)$. 
A homomorphism of $\lambda$-rings $\phi\colon R \to S$ 
    is a ring homomorphism satisfying 
$\phi(\lambda_k(r)) = \lambda_k(\phi(r))$ for all $r$. 
A $\lambda$-ring is \defn{special} if %the additive map 
$\lambda \colon R \to 1 + t R[\![t]\!]$ 
is a $\lambda$-ring homomorphism 
when $1 + t R[\![t]\!]$ 
%regarded as the ring of (big) Witt vectors, 
is equipped with its canonical $\lambda$-structure 
(Example~\ref{ex:wittvectors}). 

\begin{remark}
If $\lambda=\lambda_t$ is a $\lambda$-structure for $R$, then 
    $\sigma_t=\lambda_{-t}^{-1}$ is 
    another $\lambda$-structure for $R$, 
    called the associated \defn{$\sigma$-structure} 
    or the \defn{opposite $\lambda$-structure}. 
\end{remark}

For any integral symmetric function 
$f \in \Lambda_\ZZ=\ZZ[e_1,e_2,\ldots]$ 
expressed as a polynomial in elementary symmetric functions, 
let 
\begin{align} %{*
    f \circ- \colon R &\to R\\
    r &\mapsto f \circ r = f(\lambda_1(r),\lambda_2(r),\ldots). 
\end{align} %*}

The next proposition is well-known. 

\begin{proposition}\label{prop:plethysm}
For any $r \in R$, 
the function $- \circ r \colon \Lambda_\ZZ \to R$ 
is a ring homomorphism, 
which commutes with $\lambda$-operations 
if and only if 
$(g \circ f) \circ r = g \circ (f \circ r)$ 
for all $f,g \in \Lambda_\ZZ$. 
If $R$ is special then 
$(g \circ f) \circ r = g \circ (f \circ r)$ 
    for any $f,g\in \Lambda_\ZZ$, $r\in R$. 
\end{proposition}

%\begin{remark}
%The binary operation $\circ$ 
%    is not always associative 
%    and its properties in each argument are different, 
%    e.g. $- \circ r$ is a ring homomorphism while 
%    $f \circ -$ is usually not even additive. 
%\end{remark}

\begin{example}[{\cite[Lemma~4.1]{heinloth}}]\label{prop:symmetricmonoidalcategory}
Let $(\mathcal A,\otimes)$ be a $\QQ$-linear tensor category 
and let $p_{\varepsilon,n}\in \QQ[\Sigma_n]$ 
    ($\varepsilon \in\{\pm 1\}$) 
    denote the element of 
    the group algebra of the symmetric group given by 
$p_{\varepsilon,n} = \frac{1}{n!}\sum_{\sigma \in \Sigma_n} 
    \varepsilon^{\sigma} \sigma$. 
Then 
    $\lambda(a) = 
    \sum_{n=0}^\infty
    p_{-1,n}(a^{\otimes n}) t^n$ 
defines a special $\lambda$-structure on 
    the Grothendieck ring $K_0(\mathcal A,\otimes)$. 
Its opposite $\lambda$-structure has operations 
    $\lambda_n(a) = p_{1,n}(a^{\otimes n})$. 
\end{example}

\begin{example}\label{ex:wittvectors}
Let $\Lambda(R) = 1+tR[\![t]\!]$. 
There is a ring structure on $\Lambda(R)$ 
whose {addition} is ordinary multiplication of power series 
and whose {multiplication} $\ast$ 
is determined by the formula 
$(1-at)^{-1} \ast (1-bt)^{-1} = (1-abt)^{-1}$ for $a,b \in R$ 
and extended to general power series 
by linearity and $t$-adic continuity. 
%The additive identity in $\Lambda(R)$ 
%is $1$ and the multiplicative identity is $(1-t)^{-1}$. 
The ring $\Lambda(R)$ 
    (ring of big Witt vectors of $R$ \cite{MR4015455}) 
    has a canonical $\lambda$-structure 
for which it is a special $\lambda$-ring \cite[\S2]{MR2649360}. 
%If $R$ is a $\QQ$-algebra, 
%then any element $f$ of $\Lambda(R)$ has a formal logarithm 
%$a_1 t + a_2 \frac{t^2}{2} + a_3 \frac{t^3}{3}+\cdots$ 
%satisfying 
%$f = \exp(a_1 t + a_2 \frac{t^2}{2} + a_3 \frac{t^3}{3}+\cdots)$, 
%and we have 
%    $\psi_rf
%    =
%    \exp( a_r t + a_{2r} \frac{t^2}{2} 
%    +a_{3r} \frac{t^3}{3}+ \cdots )$. 
\end{example}

%However, if $R$ is special and $f$ is a power symmetric function 
%then $f \circ -$ is a ring homomorphism, even commuting with $\lambda$-operations 
%(cf.~\S\ref{sec:adams}).) 

%\begin{proof} %{*
%The first assertion is obvious. 
%As $\lambda_k(f) = e_k \circ f$, 
%the function $- \circ r$ commutes with 
%$\lambda$-operations if and only if 
%%$\phi_r(\lambda_k(f)) = \lambda_k(\phi_r(f))$ 
%%    for any $k \geq 2$ and $f \in \Lambda$. 
%%This is the same as requiring that 
%$(e_k \circ f) \circ r  = 
%e_k \circ (f \circ r)$ 
%for all $k$ and $f \in \Lambda_\ZZ$, 
%%In fact, the condition $(e_k \circ f) \circ r  = 
%%    e_k \circ (f \circ r)$ holds for all $k \geq 2$ 
%but this is equivalent to 
%$(g \circ f) \circ r = g \circ (f \circ r)$ 
%for all $f,g \in \Lambda_\ZZ$. 
%%Indeed, 
%%\begin{align} %{*
%%(g \circ f) \circ r = 
%%    g(e_1\circ f,\ldots) \circ r
%%    &= 
%%    g((e_1\circ f) \circ r,\ldots) \\
%%    &= g(e_1\circ (f \circ r),\ldots) \\
%%    &= g \circ (f \circ r).
%%\end{align} %*}
%For the last assertion let $f,g \in \Lambda_\ZZ$ and $r \in R$. 
%If $R$ is special then there is a $\lambda$-homomorphism 
%    $\phi \colon \Lambda_\ZZ \to R$ satisfying $\phi(h_1) = r$ 
%    by the universal property \eqref{eqn:nongradedhomsetidentify} 
%    of $\Lambda_\ZZ$, 
%    and then $\phi(g \circ f) = \phi((g \circ f) \circ h_1) 
%    = (g \circ f) \circ \phi(h_1) 
%    = (g \circ f) \circ r
%    = g \circ \phi(f) = g \circ (f \circ r)$. 
%\end{proof} %*}

%*}

\subsubsection{Adams operations} %{*

The \defn{$n$th Adams operation} $\psi_n\colon R \to R$ is 
the function defined by $\psi_n = p_n \circ -$ 
where $p_n$ is the $n$th power symmetric function. 
%Adams operations may also be defined by the identity 
%$\sum_{n=1}^\infty\psi_n(r)t^n 
%=-t\frac{d}{dt}\log
%\left(\sum_{n=0}^\infty\lambda_n(r)(-1)^nt^n\right)$. 
It is well-known that Adams operations are additive 
(see e.g.~\cite{MR244387} for the case of special $\lambda$-rings). 
If the $\lambda$-structure is special, then 
$\psi_n$ is a $\lambda$-endomorphism and 
$\psi_{mn} = \psi_m \psi_n$ (function composition). 

%\begin{example}
%For any symmetric function $f(x_\ast)$ 
%or polysymmetric function $g(x_{\ast\ast})$, 
%we have $\psi_n f(x_\ast) = f(x_\ast^n)$ 
%and $\psi_n g(x_{\ast\ast}) = g(x_{\ast\ast}^n)$. 
%\end{example}

\begin{lemma}\label{lemma:adamstolambda}
Let $f \colon R \to S$ be a ring homomorphism 
between $\lambda$-rings 
and assume $S$ has no additive torsion. 
If $f$ commutes with Adams operations then 
$f$ commutes with $\lambda$-operations. 
\end{lemma}

\begin{proof}
This follows from Newton's identities 
and an induction argument, see \cite[Corollary~3.16]{MR2649360}. 
The same proof there given for special $\lambda$-rings 
works for general $\lambda$-rings. 
\end{proof}

%*}

\subsubsection{Coefficients of the plethystic logarithm} %{*

Let $R$ be a not necessarily special 
$\lambda$-ring without additive torsion 
and write 
$\mathrm{Log}(1+x_1 t+x_2t^2+\cdots)
=u_1t+u_2t^2+\cdots$. 
Equivalently, $x_1,x_2,\ldots$ and $u_1,u_2,\ldots$ 
are elements in $R$ satisfying 
\begin{equation}\label{eqn:interestingidentity} %{*
1+\sum_{k=1}^\infty x_k t^k
=
    \prod_{d=1}^\infty
    \exp \left(
    \sum_{r=1}^\infty
    \psi_r u_d
    \frac{t^{rd}}{r}
    \right) . 
\end{equation} %*}
The formula for $u_d$ 
    proven in \cite[Prop.~8.6]{modularoperads} 
    assumed that $R$ was special. 
This assumption is too strong in practice 
(we will later take $R=\KV{k}$ 
the Grothendieck ring of varieties 
which is not special) 
    so here we give a simple proof of the same formula 
    under minimal assumptions. 

\begin{theorem}\label{thm:maintheorem}
Suppose that $x_1,\ldots,x_d$ are contained 
   in a $\lambda$-subring of $R$ 
   on which Adams operations separate, 
   i.e. $\psi_{mn} = \psi_m \psi_n$. 
Then 
\begin{equation}%\label{eqn:fdck} %{*
    u_d=
\sum_{d=km}
\sum_{\lambda \vdash m}
    (-1)^{\ell-1}\frac{\mu(k)}{k\ell}
    \binom{\ell}{n_1,\ldots,n_m}
\psi_{k}
x^\lambda.
\end{equation} %*}
\end{theorem}

\begin{remark}
Getzler--Kapranov~\cite{modularoperads} 
    attribute this formula to 
    Cadogan~\cite{cadogan}.  
In the special case when $R = \QQ[t]$ 
    this was rediscovered in 
    \cite[Lemma~6.3]{FLORENTINO2021104008}. 
\end{remark}

%First we need a version of M\"obius inversion 
%with Adams operations. 

\begin{lemma}%[M\"obius inversion with Adams operations]
Suppose a ring $R$ is equipped with additive functions 
    $\psi_k \colon R \to R$. 
    %with $\psi_1 = \id_R$. 
    %$\psi_{ab} = \psi_a \psi_b$ for all $a,b \geq 2$. 
If $(x_k)_{k = 1}^d$ and $(u_k)_{k = 1}^d$ 
    are elements in $R$ 
    such that 
    $\psi_1 u_d = u_d$ and  
    $\psi_{ab}u_c = \psi_a \psi_b u_c$ 
    for any integers $a,b,c$ with $abc = d$,  
    and satisfying 
\begin{equation} %{*
    x_d = \sum_{k|d} \psi_{d/k} u_k , 
\end{equation} %*}
then 
\begin{equation} %{*
    u_d = \sum_{k|d} \mu(d/k) \psi_{d/k} x_k.
\end{equation} %*}
\end{lemma}

\begin{proof}
\begin{align*}
    \sum_{k|d} \mu(d/k)\psi_{d/k} x_k 
    &= \sum_{k|d} \mu(d/k)\psi_{d/k} 
        \sum_{\ell|k}\psi_{k/\ell}u_{\ell}\\
    &= \sum_{\ell|k|d}\mu(d/k)\psi_{d/\ell}u_\ell\\
    &= \sum_{\ell|d}\psi_{d/\ell}u_\ell\sum_{n|d/\ell}\mu(n)
    = u_d.\qedhere
\end{align*}
\end{proof}

Call a subset \defn{separable} 
    if it satisfies the hypothesis of the theorem. 

\begin{lemma}\label{lemma:separableuorx} 
    $\{x_1,\ldots,x_d\}$ is separable 
    if and only if 
    $\{u_1,2u_2,\ldots,du_d\}$ is separable. 
\end{lemma}

\begin{proof}
First we prove an identity. 
Taking the logarithm of \eqref{eqn:interestingidentity} obtains 
\begin{equation}%\label{eqn:interestingidentity} %{*
    \log\left(
    1+
\sum_{n=1}^\infty 
    x_n t^n
    \right)
=
    \sum_{d,k=1}^\infty
    \psi_k u_d
    \frac{t^{kd}}{k}  
    =
    \sum_{m=1}^\infty
    \frac{t^m}{m}
    \sum_{k|m}
    \psi_k u_{m/k}
    \frac{m}{k} . 
\end{equation} %*}
The coefficient of $\frac{t^k}{k}$ on the left-hand side is 
    $M_k(x_1,\ldots,x_k)$ 
    where $M_k(h_1,\ldots,h_k)$ is defined by\footnote{The $m$th Faber polynomial $F_m$ is defined by $\sum_{m =0}^\infty F_m(a_1,\ldots,a_m;t)z^{-m-1} = \frac{\partial}{\partial z}\log(f(z)-t)$ where $f(z) = z + a_1 + a_2 z^{-1} + a_3 z^{-2} + \cdots$. They satisfy $F_m(h_1,\ldots,h_m;0) = -M_m(h_1,\ldots,h_m)$.} 
\begin{equation} %{*
    1+h_1t+h_2t^2+\cdots 
    %=1+M_1(p_1)t+M_2(p_1,p_2)t^2+\cdots 
    = \exp\left( \sum_{k=1}^\infty p_k \frac{t^k}{k} \right) 
    = \exp\left( \sum_{k=1}^\infty M_k(h_1,\ldots,h_k) \frac{t^k}{k} \right) .
\end{equation} %*}
Then 
\begin{equation}\label{eqn:uandxidentity} %{*
    M_d(x_1,\ldots,x_d) = 
    \sum_{k|d}
    \psi_{d/k} u_{k}k .
\end{equation} %*}

With this identity in hand, 
    now suppose $\{x_1,\ldots,x_d\}$ is separable. 
As $x_1 = u_1$, the singleton set $\{u_1\}$ is separable. 
Assume we have shown $\{u_1,\ldots,(d-1)u_{d-1}\}$ is separable. 
By \eqref{eqn:uandxidentity}, 
\begin{equation} %{*
    u_dd = 
    M_d(x_1,\ldots,x_d) - 
    \sum_{k|d,k\neq d}
    \psi_{d/k} u_{k}k .
\end{equation} %*}
The complete homogeneous symmetric functions 
are an integral basis for $\Lambda$ 
    which implies that 
    $M_d(x_1,\ldots,x_d)$ is contained in any ring 
    generated by the $x_1,\ldots,x_d$. 
This shows the right-hand side is contained in 
    a separable $\lambda$-subring of $R$ by hypothesis, 
    which shows that $u_dd$ is as well. 

Conversely, if $\{u_1,2u_2,\ldots,du_d\}$ is separable, 
then it is clear from \eqref{eqn:interestingidentity} 
    that each $x_k$ is contained in a separable 
    $\lambda$-subring of $R$. 
\end{proof}

\begin{proof}[Proof of Theorem~\ref{thm:maintheorem}] %{*
By Lemma~\ref{lemma:separableuorx} 
    the Adams operations separate over $\{u_1,\ldots,du_d\}$, 
so we may apply M\"obius inversion with Adams operations 
to \eqref{eqn:uandxidentity} to obtain 
\begin{multline} %{*
    u_d d 
    =
    \sum_{m|d}
    \mu(d/m)
    \psi_{d/m} 
    M_{m}(x_1,\ldots,x_m) \\
    = 
    \sum_{m|d}
    \mu(d/m)
    \psi_{d/m} 
    \sum_{\substack{n_1+2n_2+\cdots + mn_m = m\\n_1,\ldots,n_m \geq 0}}
    \frac{m(-1)^{1+\sum_k n_k}}{\sum_k n_k}
    \frac{(\sum_k n_k) !}{n_1! \cdots n_m!}
    \prod_{k=1}^m
    x_k^{n_k} . 
\end{multline} %*}
where the expression for $M_m$ comes from \eqref{eqn:pnintermsofh}. 
Note the right-hand side is actually an integral expression 
in powers of $x_1,\ldots,x_d$ and Adams operations 
    in view of the first equality. 
\end{proof} %*}

%*}

%*}

\subsubsection{Binomial rings}\label{sec:binomialrings} %{*

A ring $B$ satisfying any of the equivalent conditions 
    of the following proposition 
    is a \defn{binomial ring}. 

\begin{proposition}[{\cite[\S5]{MR2649360}}]
The following statements are equivalent: 
\begin{enumerate} %{*
    \item $B$ is $\ZZ$-torsion-free and $\binom{b}{k} = \frac{b(b-1)\cdots(b-k+1)}{k!} \in B \subset B\otimes\mathbb Q$ 
        for all $b \in B$ and $k \geq 1$, 
    \item $B$ is $\ZZ$-torsion-free and 
        $a^p \equiv a \Mod {pB}$ for any $a \in B$ and prime $p$, 
    \item $B$ is a special $\lambda$-ring 
        with trivial Adams operations. 
\end{enumerate} %*}
\end{proposition}

The $\lambda$-operations on a binomial ring 
    are given by $\lambda_k(b) = \binom{b}{k}$. 

\begin{remark}\label{rmk:quotientBinomialRing}
Any $\lambda$-ring $R$ has
    a canonical quotient binomial ring $R_B$, 
namely the quotient of $R$ by the smallest ideal 
    closed under $\lambda$-operations 
    containing 
    the additive torsion of $R$ 
    and $\psi_m(r) - r$ 
    for every $r \in R$ and $m \geq 1$. 
%This quotient is automatically special 
%since any ring with trivial Adams operations 
%is special by Wilkerson's theorem. 
\end{remark}

%(cf.~e.g.~\cite[Prop.~5.7]{MR2649360}). 
The following should be well-known 
though we do not know a reference. 
We leave the proof to the reader. 
%but for lack of a reference 
%we give a proof. 
%General plethsyms can be expressed 
%in the monomial symmetric basis 
%with multinomial coefficients. 
%
%\begin{definition}[multinomial coefficients]\label{defn: multinomialcoefficients}
%Let $n_1,\ldots,n_d \geq 0$ be integers with $N = \sum_i n_i$ 
%and define 
%\begin{equation}
%    \binom{x}{n_1,\dots,n_d} \coloneqq \frac{x(x-1)\cdots(x-N+1)}{n_1!\cdots n_d!} . 
%\end{equation}
%\end{definition}

%(For $x$ a positive integer, 
%$\binom{x}{n_1,\dots,n_d}$ equals 
%the number of ways of depositing $x$ objects 
%into $d+1$ bins with the first bin having $n_1$ objects, 
%the second bin having $n_2$ objects and so on 
%with the final bin having $x- \sum_i n_i$ objects.) 

\begin{proposition}\label{prop:plethysmsasbinomialcoefficients}
    Let $\lambda = (1^{n_1}2^{n_2}\cdots d^{n_d})$ 
    be a partition of degree $d$. 
If $B$ is a binomial ring, then for any $x \in B$ we have 
\begin{equation}
    m_{\lambda}\circ x = \binom{x}{n_1,\dots,n_d}.
\end{equation}
\end{proposition}

\subsection{Higher plethysm}\label{sec:gradedplethysm} %{* 

Let $R$ be a $\lambda$-ring. 
Classically plethysm is defined as a function 
    $f \circ - : R \to R$ for $f \in \lambda$. 
Here we extend it to a function 
$F \circ - \colon tR[\![t]\!] \to R$ 
for $F \in \PS$. 
(Classical plethysm is recovered by identifying 
$R \xrightarrow{\sim} tR \subset tR[\![t]\!]$.) 
The idea is that the higher copies 
$\Lambda_{(1)}, \Lambda_{(2)},\ldots \subset \PS$ 
``witness'' the higher order terms of 
$r_1t+r_2t^2+\cdots \in tR[\![t]\!]$. 
Any polysymmetric function 
$F 
\in \PS_\ZZ 
=\bigotimes_d \Lambda_{\ZZ,(d)}$ 
can be expressed as a polynomial 
$$F(f_1, \ldots, f_k)$$ 
where $f_j\in \Lambda_{\ZZ,(j)}$. 
Let $$F \circ - \colon tR[\![t]\!] \to R$$ 
    be the function defined by 
\begin{equation} %{*
F \circ (r_1t+r_2t^2+\cdots) = 
F(f_1 \circ r_1, \ldots, f_k \circ r_k).
\end{equation} %*}
We state some basic facts about higher plethysm, 
omitting the proofs as they are analogous to the ordinary case. 

\begin{proposition}\label{prop:gradedplethysm}
%Let $R=\bigoplus_{d=0}^\infty R_d$ be a graded $\lambda$-ring. 
Let $r = r_1t+r_2t^2+\cdots\in tR[\![t]\!]$. 
\hfill\begin{enumerate}
\item $- \circ r$ 
is a ring homomorphism $\PS_\ZZ \to R$ 
and it commutes with $\lambda$-operations 
if and only if 
$(g \circ f) \circ r= g \circ (f \circ r)$ 
for all $f,g \in \PS_\ZZ$. 
    If $R= \bigoplus_{d} R_d$ is a graded $\lambda$-ring\footnote{A \defn{graded $\lambda$-ring} is a graded ring with $\lambda$-operations satisfying $\lambda_n(R_d) \subset R_{nd}$ 
        for all $n$ and $d$.} 
    and $r_d \in R_d$ for all $d$,  
    then $- \circ r \colon \PS_\ZZ \to R$ is 
    a graded ring homomorphism. 
\item 
There is a natural bijection 
\begin{align}\label{eqn:homsetidentify} %{*
\Hom_{\GrLR}(\PS_\ZZ,R) 
&\xrightarrow{\sim} \prod_{d=1}^\infty R_d \\
\phi &\mapsto (\phi(M_d))_{d}. 
\end{align} %*}
\end{enumerate}
\end{proposition}

%\begin{remark}[Power structures]\label{sec:powerstructure} %{*
%Graded plethysm extends 
%the notion of a power structure \cite{MR2046199}. 
%Recall that \cite{MR2046199} defines a function 
%    $R[\![t]\!] \times R \to R[\![t]\!]
%    \colon (A(t),[X]) \mapsto A(t)^{[X]}$ 
%satisfying the usual properties of exponentiation 
%as well as the identity 
%\begin{equation}
%    (1+t+t^2+\cdots)^{[X]} 
%    = 1 + [X]t + [\SymProd{X}{2}]t^2 + \cdots 
%    = Z_X(t)
%\end{equation}
%when $R = \KV{k}$ 
%\cite[Theorem~1]{MR2046199}. 
%Graded plethysm extends power structures to a function 
%\begin{equation}
%    R[\![t]\!] \times tR[\![t]\!] \to R[\![t]\!]
%    \colon (A(t),u) \mapsto A(t)^{u}
%\end{equation}
%satisfying the usual properties of exponentiation 
%as well as the identity 
%\begin{equation}\label{eqn:gradedpowerstructure}
%    (1+t+t^2+\cdots)^{[\Uc]} 
%    = Z_{\Uc}(t) 
%    =
%    \prod_{d=1}^\infty 
%    \left(\frac{1}{1-t^d}\right)^{[\Uc_d]}
%    = 1 + [\Sc_1]t + [\Sc_2]t^2 + \cdots 
%\end{equation}
%for $\Uc = [\Uc_1]t+[\Uc_2]t^2+\cdots$. 
%(Indeed the same argument as in \cite[Prop.~1]{MR2252764} 
%shows that $A(t)^{[\Uc]}$ is already determined by 
%\eqref{eqn:gradedpowerstructure} 
%and the properties in \cite[\S1]{MR2252764}. 
%The second equality of 
%\eqref{eqn:gradedpowerstructure} 
%is by definition of the motivic zeta function of $\Uc$ 
%while the third equality 
%is Proposition~\ref{prop:zetafngradedspace}.) 
%\end{remark}
%
%%*}

%*}

%*}

\section{Families of polysymmetric functions}\label{sec:bases} %{*

%In \S\ref{sec:1} we introduced 
%the monomial polysymmetric functions. 
%In this section we define 
%the remaining polysymmetric families 
%and prove their generating properties. 
%The monomial polysymmetric functions are 
%pure tensors with respect to 
%the monomial symmetric basis of $\Lambda$, 
%however the four bases below 
%are mixed tensors with respect to 
%any classical basis of $\Lambda$ 
%and in this sense are non-classical. 

\subsection{Definitions and identities}\label{sec:canonicalbases} %{*

The \defn{length} $\len f$ of a monomial $f$ is defined by 
$y^{\len f}f(x_{\ast\ast}) = f(yx_{\ast\ast})$. 
%considered as an element of $\ZZ[x_{\ast\ast}]$. 
Let $d$ be a positive integer. 
\begin{itemize}
\item[$\diamond$] The \defn{$d$th complete homogeneous polysymmetric function} is 
\begin{equation} %{*
H_d = %H_{d^1} := 
\sum_{\substack{\text{monomials $f$}\\\text{of degree $d$}} } 
   f .
%=
%\sum_{
%    \substack{(d_1,i_1)\leq (d_2,i_2) \leq \cdots \leq (d_r,i_r)\\
%              d_1+\cdots+d_r=d
%    }}
%        x_{d_1,i_1}x_{d_2,i_2}\cdots x_{d_r,i_r} .
\end{equation} %*}
\item[$\diamond$] The \defn{$d$th elementary polysymmetric function} is 
\begin{equation} %{*
E_d = %E_{d^1} := 
\sum_{\substack{\text{sq.free monomials}\\\text{$f$ of degree $d$}} } 
   (-1)^{\mathrm{len}\, f}f. 
%E_{(d,m)}= 
%\sum_{
%    \substack{(d_1,i_1)< (d_2,i_2) < \cdots < (d_r,i_r)\\
%              d_1+\cdots+d_r=d
%    }}
%        x_{d_1,i_1}x_{d_2,i_2}\cdots x_{d_r,i_r} .
\end{equation} %*}
\item[$\diamond$] The \defn{$d$th unsigned elementary polysymmetric function} is
\begin{equation} %{*
E^+_{d} = % E^s_{d^1} := 
\sum_{\substack{\text{sq.free monomials}\\\text{$f$ of degree $d$}} } 
   f.
%    =
%\sum_{
%    \substack{(d_1,i_1)< (d_2,i_2) < \cdots < (d_r,i_r)\\
%              d_1+\cdots+d_r=d
%    }}
%        (-1)^rx_{d_1,i_1}x_{d_2,i_2}\cdots x_{d_r,i_r}. 
\end{equation} %*}
\item[$\diamond$] The \defn{$d$th power polysymmetric function} is
\begin{equation} %{*
P_d = 
\sum_{k|d} k 
\sum_{\substack{\text{monomials $f$}\\\text{of type $k^{d/k}$}} } f = 
\sum_{k|d} k \sum_{j=1}^\infty x_{kj}^{d/k}.
\end{equation} %*}
\end{itemize}
%Set $E_{i^j} = E^s_{i^j} = H_{i^j} = 1$ 
%    if either $i$ or $j$ is zero. 
%For a positive integer $m$ we set  
%\begin{equation} %{*
%    %E_{d^m}  := E_{d}(\{x_{\ast\ast}^m\}),\quad
%    %E^s_{d^m}  := E^s_{d}(\{x_{\ast\ast}^m\}),\quad
%    H_{d^m} := \psi_m(H_d) =  H_{d}(\{x_{\ast\ast}^m\}).
%\end{equation} %*}
For a type $\tau$ we define 
\begin{equation} %{*
%E_{\tau} := \prod_{d^m \in \tau} E_{d^m},\qquad
%E^s_{\tau} := \prod_{d^m \in \tau} E^s_{d^m},\qquad
H_{\tau} = 
\prod_{d^m \in \tau} \psi_m( H_{d}) = 
\prod_{d^m \in \tau} H_{d}(x_{\ast\ast}^m)
\end{equation} %*}
where $\psi_m$ is the $m$th Adams operation on $\PS$. 
%(cf.~\S\ref{sec:adams}); 
The polysymmetric functions $E_\tau$, $E^+_\tau$ and $P_\tau$  
are defined analogously.\footnote{$H_{2^2}$ always means 
$\psi_2 H_2 = H_2(x_{\ast\ast}^2)$ and never $H_4$.} 
%The symmetric function $H_{\tau}$ is defined in the same way as $E_{\tau}$. 

\begin{remark}
One can also define 
$H_{d^m}$, $E_{d^m}$, and $E^+_{d^m}$ 
with generating functions: 
$\sum_{d} H_{d^m}t^{md} = 
\prod_{i,j}(1-x_{ij}^mt^{mi})^{-1}$, 
$\sum_{d} E_{d^m}^+t^{md} = 
\prod_{i,j}(1+x_{ij}^mt^{mi})$,
and $\sum_{d} E_{d^m}t^{md} = 
\prod_{i,j}(1-x_{ij}^mt^{mi}).$ 
\end{remark}

These bases behave similarly to their symmetric siblings. 
For instance, the polysymmetric analogue of 
\begin{equation}\label{eqn:fundamentalidentitysymmfn} %{*
\sum_{d=0}^\infty
h_dt^d
=
\left(\sum_{d=0}^\infty
(-1)^d e_dt^d\right)^{-1}
=
\sum_{\lambda \text{ partition}}
    m_\lambda t^{|\lambda|}
=
\exp\left(
\sum_{k=1}^\infty
    p_k \frac{t^k}{k}
\right) 
%=
%\prod_{d=1}^\infty
%    (1-t^d)^{-q_d}.
\end{equation} %*}
is 
\begin{equation}\label{eqn:fundidentitypolysymm} %{*
\sum_{d=0}^\infty
H_dt^d
=
\left(\sum_{d=0}^\infty
E_dt^d\right)^{-1}
=
\left(\sum_{d=0}^\infty
E_d^+(-x_{\ast\ast})t^d\right)^{-1}
=
\sum_{\tau \text{ type}}
    M_\tau t^{|\tau|}
=
\exp\left(
\sum_{k=1}^\infty
    P_k \frac{t^k}{k}
\right). 
%=
%\prod_{d=1}^\infty
%    (1-t^d)^{-M_d}.
\end{equation} %*}

%\begin{remark}
%$E^+_d$ is the analogue to $e_d$, 
%    and $E_d$ to $(-1)^de_d$, 
%but unlike the symmetric function situation 
%there is no simple relation between $E^+_d$ and $E_d$. 
%The concepts of length and degree diverge 
%for monomials in $x_{\ast\ast}$, 
%which is why we cannot `fix' 
%\eqref{eqn:fundidentitypolysymm} to use $E^+$ in place of $E$. 
%\end{remark}

We record a list of further identities for future reference: \\[-.5em]
\begin{align}
&\sum_{d=0}^\infty H_{d}t^{d} = 
\prod_{n=1}^\infty 
    \sigma_{t^n}(M_n), 
    &
&\sum_{d=0}^\infty H_{d}t^{d} = 
\prod_{n=1}^\infty 
    \sum_{k=0}^\infty 
    h_{k,(n)} t^{kn}, \\
&\sum_{d=0}^\infty H_{d}t^{d} = 
\sum_{\tau \text{ type}}
    M_\tau t^{|\tau|}, 
    &
&\sum_{d=0}^\infty E_{d}^+t^{d} = 
\sum_{\tau \text{ sq.free}}
    M_\tau t^{|\tau|},\\
&\sum_{k=0}^dH_{k^2}E_{d-2k}^+ = H_d, 
&
&\sum_{k=0}^dH_{k}E_{d-k} = \delta_{d,0},\\
&P_d %= \sum_{k|d} k \psi_{d/k} h_1 \langle k \rangle
= \sum_{k|d} k \psi_{d/k} M_k, 
&
&M_d = \frac{1}{d} \sum_{k|d} \mu(\tfrac{d}{k}) \psi_{d/k} P_{k}.
\end{align}

%\begin{proof} %{*
%The two identities with H_{k^2} and \delta_{d,0} 
%    follow from two power series identities: 
%    \[\left({\sum_{d=0}^\infty H_{d}t^d}\right)
%    \left({\sum_{d=0}^\infty H_{d^2}t^{2d}}\right)^{-1} = 
%    \sum_{d=0}^\infty E_{d}^+t^{d}\]
%and
%\begin{equation} %{*
%\left(\sum_{d=0}^\infty H_{d}t^d\right)
%    \left(\sum_{d=0}^\infty E_{d}t^d\right) = 1. 
%\end{equation} %*}
%\end{proof} %*}

\begin{remark}\label{rmk:gorsky}
In the language of polysymmetric functions, 
    Gorsky's formula \cite[Theorem~1]{gorsky} 
    expresses $M_d$ in terms of the power polysymmetric basis. 
\end{remark}

%\begin{theorem*}[]
%Suppose that 
%    $(x_k)_{k\geq 1}$ and $(u_d)_{d\geq 1}$ 
%    are elements of a $\lambda$-ring $R$ 
%    satisfying 
%\begin{equation}\label{eqn:gorskyhypothesis} %{*
%    1+x_1t + x_2t^2+\cdots
%    =
%    (1-t)^{-u_1}
%    (1-t^2)^{-u_2}
%    \cdots
%\end{equation} %*}
%    where each expression $(1-t^k)^{-u_k}$ 
%    is evaluated using the associated power structure. 
%Let elements $(y_k)_{k\geq 1} \subset R$ be defined by 
%    $t\frac{d}{dt} \log(1+x_1t + x_2t^2+\cdots)
%    = y_1t+y_2t^2+\cdots$. 
%If $R$ is special, then 
%\begin{equation}\label{eqn:gorskyconclusion} %{*
%    du_d = 
%    \sum_{m|d}
%    \mu(d/m) \psi_{d/m} y_{m}.
%\end{equation} %*}
%\end{theorem*}
%
%In terms of graded plethysm, 
%\eqref{eqn:gorskyhypothesis} 
%is the result of applying $- \circ u$ 
%to the power series $1+H_1t+H_2t^2+\cdots$ in $\PS[\![t]\!]$. 
%Recall the following identity from \S\ref{sec:canonicalbases}: 
%\begin{equation}
%\sum_{k=0}^\infty H_{k}t^{k} = 
%\prod_{d=1}^\infty 
%    \sigma_{t^d}(M_d) = 
%\exp\left(P_1 + P_2\tfrac{t^2}{2}+P_3\tfrac{t^3}{3}+\cdots\right). 
%\end{equation}
%We see that $x_k$, $u_d$, and $y_r$ 
%are the result of applying $- \circ u$ 
%to $H_k$, $M_d$, and $P_r$, respectively. 

%*}

%*}

\subsection{Generating properties} %{*

\begin{theorem}\label{thm: H, E algebraic basis}
Each of the sets $\{E_\tau^+\}_{\tau\text{ type}}$, 
$\{H_\tau\}_{\tau\text{ type}}$ 
is a linear basis of $\PS$. 
Each of the sets $\{E_{d^m}^+\}_{d,m = 1}^\infty$, 
$\{H_{d^m}\}_{d, m = 1}^\infty$ 
is an algebraic basis of $\PS$. 
%$\PS = \QQ[\{H_{(d,m)}\}_{d,m = 1}^\infty] = \QQ[\{E_{(d,m)}\}_{d,m = 1}^\infty].$ 
%= \QQ[\{E_{(d,m)}^s\}_{d,m = 1}^\infty]. 
%In particular, $\PS$ is a free $\QQ$-algebra. 
\end{theorem}

\begin{proof} %{*
The matrix coefficients in the monomial polysymmetric basis 
of the linear mapping of $\PS_{d}$ defined by 
$M_\tau \mapsto H_\tau$ %for all $\tau \vDash d$ 
are given by $\arr$. 
This matrix is invertible as $\arr$ is invertible in 
    the incidence algebra $\IT{d}$ of types over $\QQ$ 
    (see Proposition~\ref{prop:arrangementforgetmerge} below). 
Thus, the set $\{H_\tau\}_{\tau \vDash d}$ is a $\QQ$-basis of $\PS_d$. 
The set $\{H_\tau\}_{\tau \text{ type}}$ is precisely 
the set of monomials in 
$\{H_{d^m}\}_{d,m = 1}^\infty$, 
which shows that $\PS$ is free as a $\QQ$-algebra on the set 
$\{H_{d^m}\}_{d,m = 1}^\infty$. 
The same argument works with $E_\tau^+$ in place of $H_\tau$ 
and $\arrsqf$ in place of $\arr$. 
%    We have the identities in $\ZZ[\![\{x_{i,j}\}]\!][\![t]\!]$, 
%    for any positive integer $m$, 
%    \begin{equation}\label{eqn:recallEsHidentities} %{*
%    \sum_{d=0}^\infty E_{(d,m)}^st^{d} = \prod_{i,j = 1}^\infty(1-x_{i,j}^mt^{i})
%    \quad\text{and}\quad
%    \sum_{d=0}^\infty H_{(d,m)}t^d = \prod_{i,j=1}^\infty(1-x_{i,j}^mt^i)^{-1}
%    \end{equation} %*}
%    Multiplying these together shows that 
%    \begin{equation} %{*
%    \sum_{k = 0}^d
%        E^s_{(k,m)}H_{(d-k,m)} = \delta_{d,0}
%        \quad
%        \text{for all $d \in \{0,1,2,\ldots\}$.}
%    \end{equation} %*}
%    It follows that 
%
%Similar proofs to the above can be given for the $E^s$ and $E$ analogues. However, we take another track and mimic the standard of a similar statement in $\Lambda$.
%
%Since the generating function for the $H_{p,m} : p\geq 1$ and $E^s_{p,m}: p\geq 1$ are inverses of each other, we have the relation for each $m\geq 1$:
%\[\sum_{i=0}^{k}H_{(i,m)}E^s_{(k-i,m)} = \delta_{k,0}.\]
%Thus, the $\{E^s_{(k,m)}\}_{k,m = 1}^\infty \cup \{1\}$ 
%are also an algebraic basis for $\PS$ and the map $w: H_{k,m} \to E^s_{k,m}$ is well-defined.
\end{proof} %*}

The above theorem is useful for constructing maps out of $\PS$. 

\begin{proposition}\label{prop:omegainvolution}
The ring endomorphism $\Omega \colon \PS \to \PS$ 
defined by $\Omega(H_{d^m}) = E_{d^m}$ 
%for all $d,m$ 
is an involution. 
\end{proposition}

\begin{proof} %{*
%$\Omega$ is well-defined by Theorem~\ref{thm: H, E algebraic basis}. 
%$$\Omega(H_{d^m}) = E_{d^m}^s$$
%for all $d$ and $m$. 
Extend $\Omega$ to an endomorphism of $\PS[\![t]\!]$ by acting trivially on $t$. 
Define $\mathcal H\coloneqq \sum_{d=0}^\infty H_{d}t^d$ 
and $\mathcal E \coloneqq \sum_{d=0}^\infty E_{d}t^{d}$. 
We have the factorizations %in $\ZZ[\![x_{\ast\ast}]\!][\![t]\!]$, 
$\mathcal H = \prod_{i,j=1}^\infty(1-x_{ij}t^i)^{-1}$ 
and 
$\mathcal E= \prod_{i,j = 1}^\infty(1-x_{ij}t^{i})$, 
%Applying $\Omega$ to 
%the product of the power series in \eqref{eqn:recallEsHidentities} 
%Multiplying these together and applying $\Omega$ 
%(acting trivially on $t$) shows that 
and so 
$
1 = 
%    \Omega(\mathcal H \mathcal E) = 
\Omega(\mathcal H) \Omega(\mathcal E)
=
\mathcal E \Omega(\mathcal E)
$ 
which implies $\Omega(\mathcal E) = \mathcal E^{-1} = \mathcal H$. 
We conclude $\Omega(E_{d}) = H_{d}$ 
for all $d$. 
Running the same argument for $\psi_m(\mathcal E)$ and $\psi_m(\mathcal H)$ 
%(noting the $m$th Adams operation $\psi_m$ is a ring endomorphism) 
shows that $\Omega(E_{d^m}) = H_{d^m}$. 
%for all $d$ and $m$. 
%\begin{multline} %{*
%    1 = 
%    \sum_{d = 0}^\infty
%    \sum_{k = 0}^d
%        \Omega(H_{(k,m)}E^s_{(d-k,m)}) t^d
%     = 
%    \sum_{d = 0}^\infty
%    \sum_{k = 0}^d
%        E_{(k,m)}^s \Omega(E_{(d-k,m)}^s)t^d\\
%    =
%    \left(
%    \sum_{d=0}^\infty E_{(d,m)}^st^d
%    \right)
%    \left(
%    \sum_{d=0}^\infty \Omega(E_{(d,m)}^s)t^d
%    \right).
%\end{multline} %*}
%As inverses are unique, 
%This shows that $\Omega(E^s_{(d,m)}) = H_{(d,m)}$ 
%for all $d$ and $m$, 
%    and therefore $\Omega^2 = \mathrm{id}$. 
%Multiplying these together shows that  
%\begin{equation} %{*
%    \sum_{k = 0}^d
%        E^s_{(k,1)}H_{(d-k,1)} = \delta_{d,0}
%        \quad
%        \text{for all $d \in \{0,1,2,\ldots\}$.}
%\end{equation} %*}
%Applying $\Omega$ obtains 
%\begin{equation} %{*
%    \sum_{k = 0}^d
%        H_{(k,1)} \Omega(H_{(d-k,1)}) = \delta_{d,0}
%        \quad
%        \text{for all $d \in \{0,1,2,\ldots\}$.}
%\end{equation} %*}
\end{proof} %*}

\begin{corollary}
$\{E_{d^m}\}_{d,m = 1}^\infty$ is an algebraic basis of $\PS$. 
$\{E_\tau\}_{\tau\text{ type}}$ is a linear basis of $\PS$. 
\end{corollary}

\subsection{A particular specialization} %{*

%We point out a particular specialization %of $\PS$ 
%which will be used in \S\ref{sec:inversiontopstratum} to study 
%the matrix coefficients relating 
%the $H$ and $M$ polysymmetric bases. 
Let $R$ denote the ring $\Lambda_\QQ$ 
equipped with the unique $\lambda$-structure 
for which it has \emph{trivial} Adams operations 
(this differs from its standard $\lambda$-structure). 
There is a unique $\lambda$-homomorphism 
satisfying 
\begin{align} %{*
\Phi \colon \PS &\to R\\
M_d&\mapsto q_d
\end{align} %*}
where $q_1,q_2,q_3,\ldots \in \Lambda_\QQ$ 
are defined by\footnote{ 
One can show that $\{q_d\}_{d=1}^\infty$ is 
an algebraic basis of $\Lambda_\QQ$. 
%We are not aware of a name for this basis, 
These symmetric functions seem to have received little attention 
from the literature.} 
\begin{equation}\label{defn:qd} %{*
\sum_{d=0}^\infty {h_d t^d} = 
\prod_{d=1}^\infty (1-t^d)^{-q_d}.
\end{equation} %*}

\begin{proposition}\label{prop:phiimagesHP}
$\Phi(H_{k^m}) = h_k$ and 
$\Phi(P_{k^m}) = p_k$ 
for any $k,m \geq 1$. 
\end{proposition}

\begin{proof}
As $R$ has trivial Adams operations 
and $\Phi$ is a $\lambda$-homomorphism 
(thus commuting with Adams operations) 
it suffices to consider the case when $m = 1$. 
Applying $\Phi$ to 
\begin{equation} %{*
\sum_{d=0}^\infty {H_d t^d} = 
\prod_{d=1}^\infty \sigma_{t^d}(M_d) = 
\prod_{d=1}^\infty 
\exp\left( 
    \sum_{k=1}^\infty 
    \psi_k M_d \frac{t^{kd}}{k}
\right)
\end{equation} %*}
results in 
\begin{equation} %{*
\sum_{d=0}^\infty {\Phi(H_d) t^d} = 
\prod_{d=1}^\infty 
\exp\left( 
    \sum_{k=1}^\infty 
    q_d \frac{t^{kd}}{k}
\right)
=
\prod_{d=1}^\infty (1-t^d)^{-q_d}
\end{equation} %*}
which proves $\Phi(H_d) = h_d$. 
Now applying $\Phi$ to 
\begin{equation} %{*
1 + H_1 t + H_2 t^2 + H_3 t^3+ \cdots = 
\exp(P_1 + P_2\frac{t^2}{2}+P_3\frac{t^3}{3}+\cdots) 
\end{equation} %*}
results in 
\begin{equation} %{*
1 + h_1 t + h_2 t^2 + h_3 t^3+ \cdots = 
\exp(\Phi(P_1) + \Phi(P_2)\frac{t^2}{2}+\Phi(P_3)\frac{t^3}{3}+\cdots), 
\end{equation} %*}
which shows that $\Phi(P_d) = p_d$ 
as the left-hand side is 
$\exp(p_1 + p_2\frac{t^2}{2}+p_3\frac{t^3}{3}+\cdots)$. 
\end{proof}

%*}

%*}

\section{Transition matrices}\label{sec:matrixcoeffs} %{*

Now we turn to the relationships between polysymmetric bases. 
It is well-known that the matrix coefficients between 
the classical bases of symmetric functions 
admit combinatorial interpretations. 
These generalize naturally to the polysymmetric setting. 

\subsubsection{Duality} 
There is a non-classical duality on types: 
$$\vec{b}^{\vec{m}} \leftrightarrow \vec{m}^{\vec{b}}.$$ 
Although this duality has no analogue for partitions, 
%From the viewpoint of types, a partition $\lambda$ 
%    is either of the form 
%    $\lambda^{\vec{1}}$ or $\vec{1}^{\lambda}$. 
this turns out to be responsible 
for some classical facts about symmetric functions 
(see (1), (4) and (5) below and 
their generalizations 
(i), (iv), (v) proven using the duality of types). 

\subsection{Polysymmetric extensions of some classical constructions}\label{sec:polysymmetricconstructions} %{*

Let $\lambda$ and $\mu$ denote partitions of degree $d$. 
Let us recall a few classical constructions: 
\begin{itemize}
\item[$\diamond$] 
the {bilinear form} 
        $\langle \cdot, \cdot \rangle 
        \colon \Lambda \otimes \Lambda \to \ZZ$ 
defined by 
    \[\langle m_\lambda,h_\mu \rangle_\Lambda = \begin{cases}
1 &\text{ if } \lambda = \mu,\\
0 &\text{ otherwise,} 
\end{cases}\]
\item[$\diamond$] 
the numbers $M_{\lambda \mu}$ defined by the transition $e \to m$ 
\begin{equation} %{*
e_\lambda = 
\sum_{\mu \vdash d}
M_{\lambda \mu}
m_\mu, 
\end{equation} %*}
\item[$\diamond$] 
the numbers $N_{\lambda \mu}$ defined by the transition $h \to m$ 
\begin{equation} %{*
h_\lambda = 
\sum_{\mu \vdash d}
N_{\lambda \mu}
m_\mu.
\end{equation} %*}
\end{itemize}
Let $\vec{\lambda}$ denote 
    the list of parts of $\lambda$ in weakly descending order. 
These constructions satisfy 
    the following properties \cite[\S7]{Stanley}: 
\begin{enumerate} %{*
\item the bilinear form $\langle \cdot , \cdot \rangle_\Lambda$ 
    is symmetric,  
\item $M_{\lambda \mu}$ equals the number of $\{0,1\}$-matrices $A$ 
    satisfying 
    $A\vec{1} = \vec{\lambda}$ and $A^T \vec{1} = \vec{\mu}$, 
    %$\mathrm{row}(A) = \vec{\lambda}$ and $\mathrm{col}(A) = \vec{\mu}$, 
\item $N_{\lambda \mu}$ equals the number of 
    $\ZZgeq$-matrices $A$ satisfying 
    $A\vec{1} = \vec{\lambda}$ and $A^T \vec{1} = \vec{\mu}$, 
\item $M_{\lambda \mu} = M_{\mu \lambda}$, 
\item $N_{\lambda \mu} = N_{\mu \lambda}$. 
\end{enumerate} %*}
%((1) and (5) are trivially equivalent.) 

Now we turn to the polysymmetric setting. 
%Recall that 
%$\{E_\tau\}_{\tau}$, 
%$\{H_\tau\}_{\tau}$, and  
%$\{M_\tau\}_{\tau}$ 
%(indexed by types) 
%are linear bases of $\PS_\QQ$. 
Define a bilinear form by declaring 
%(note the transpose) 
\[\langle M_\lambda,H_\tau \rangle = \begin{cases}
1 &\text{if $\tau^T = \lambda$,} \\
0 &\text{otherwise}. 
\end{cases}\]
%As in the classical setting, the symmetry of this bilinear form 
%is not evident from its definition. 
For any type $\lambda \vDash d$ we have 
\begin{equation}\label{eqn:defncsf} %{*
E_\lambda^+ = \sum_{\tau \vDash d} \arrsqf(\tau,\lambda) M_\tau
\end{equation} %*}
for uniquely determined integers $\arrsqf(\tau,\lambda)$. 
Similarly, 
\begin{equation}\label{eqn:defnc} %{*
H_\lambda = \sum_{\tau \vDash d} \arr(\tau,\lambda) M_\tau
\end{equation} %*}
for uniquely determined integers $\arr(\tau,\lambda)$. 
These transition matrices 
    $\arr$ and $\arrsqf$ extend $N$ and $M$, 
    in the sense that 
    for partitions $\lambda$ and $\mu$,  
\begin{equation} %{*
\arrsqf(\vec{1}^{\vec{\lambda}}, \vec{\mu}^{\vec{1}})
= 
M_{\lambda \mu}
\quad
\text{and}
\quad
\arr(\vec{1}^{\vec{\lambda}}, \vec{\mu}^{\vec{1}})
= 
N_{\lambda \mu}. 
\end{equation} %*}
%Moreover, 
%\begin{equation} %{*
%    \langle m_\lambda,h_\mu \rangle_\Lambda
%    =
%    \sum_{\xi \vdash d}N_{\mu \xi}
%        \langle m_\lambda, m_\xi \rangle_\Lambda
%    =
%    \sum_{\xi \vdash d}N_{\mu \xi}
%        \langle M_{\vec{1}^{\vec{\lambda}}},
%        M_{\vec{\xi}^{\vec{1}}} \rangle
%    =
%    \sum_{\xi \vdash d} 
%        \arr(\vec{1}^{\vec{\mu}}, \vec{\xi}^{\vec{1}})
%        \langle M_{\vec{1}^{\vec{\lambda}}},
%        M_{\vec{\xi}^{\vec{1}}} \rangle.
%\end{equation} %*}
The bilinear form $\langle \cdot,\cdot\rangle_\Lambda$ 
is recovered by restricting 
the polysymmetric bilinear form $\langle \cdot,\cdot\rangle$ 
to the two copies $\Lambda_1$ and $\Lambda_2$ in $\PS$ 
given by 
\begin{equation} %{*
\Lambda_1 \hookrightarrow \PS: 
    m_\lambda \mapsto M_{\vec{1}^{\vec{\lambda}}}
\quad\text{and}\quad
\Lambda_2 \hookrightarrow \PS: 
    m_\lambda \mapsto M_{\vec{\lambda}^{\vec{1}}}. 
\end{equation} %*}

\begin{theorem}\label{thm:arrangementsandcoefficients}
Let $\tau=\vec{b}^{\vec{m}}$ and $\lambda=\vec{c}^{\vec{n}}$ 
be types of degree $d$. 
\begin{enumerate} %{*
\item[(i)] The bilinear form $\langle \cdot , \cdot \rangle$ 
is symmetric,  
\item[(ii)] $\arrsqf(\tau,\lambda)$ 
%of $M_\tau$ in \eqref{eqn:defncsf} 
equals the number of $\{0,1\}$-matrices $A$ satisfying 
%with entries in $\{0,1\}$ 
\begin{equation}\label{eqn:arrdefnmatrices}
A\vec{n} = \vec{m} \quad\text{and}\quad A^T \vec{b} = \vec{c}, 
\end{equation}
\item[(iii)] $\arr(\tau,\lambda)$ 
%of $M_\tau$ in \eqref{eqn:defnc} 
equals the number of $\ZZgeq$-matrices $A$ 
satisfying \eqref{eqn:arrdefnmatrices},  
\item[(iv)] $\arrsqf(\tau,\lambda) = \arrsqf(\lambda^T,\tau^T)$, and 
\item[(v)] $\arr(\tau,\lambda) = \arr(\lambda^T,\tau^T)$. 
\end{enumerate} %*}
These generalize properties (1)-(5), respectively.  
\end{theorem}

Motivated by this, we define an 
\defn{arrangement of $\vec{b}^{\vec{m}}$ into $\vec{c}^{\vec{n}}$} 
to be a $\ZZgeq$-matrix $A$ satisfying 
$A\vec{n} = \vec{m}$ and $A^T \vec{b} = \vec{c}$. 
An arrangement is \defn{squarefree} if its entries are in $\{0,1\}$. 

\begin{remark}
Transition matrices between the 
    $HE^+EP$ bases 
    and pure tensor polysymmetric bases 
    formed from symmetric bases 
    have been studied in 
    \cite{khanna2024transitionmatricespierityperules}. 
\end{remark}

%*}

\subsection{Arrangements record factorizations}\label{sec:arrsandfacts} %{*

To prove Theorem~\ref{thm:arrangementsandcoefficients} 
we show that arrangements 
record factorizations of monomials in weighted variables. 
%This interpretation will 
%help us relate polysymmetric functions 
%to integral solutions of 
%\eqref{eqn:arrdefnmatrices}. 
Let $f_\tau(x_{\ast\ast})$ be a monomial of type $\tau$. 
If $m_1,\ldots,m_r$ are positive integers and 
$f_{1}^{m_1}\cdots f_{r}^{m_r} = f_\tau$ 
then we say that $(f_{1},\ldots,f_{r};m_1,\ldots,m_r)$ 
is an \defn{factorization of $f_\tau$ of type $\lambda$} 
where $\lambda$ is the type 
$\{(\wt f_1,m_1), \ldots, (\wt f_r, m_r)\}$. 

Assertions (ii) and (iii) of 
Theorem~\ref{thm:arrangementsandcoefficients} 
follow from the next two propositions. 
%The first proposition relates the matrix coefficients 
%$\arr_{\lambda \tau}$ and $\arrsqf_{\lambda \tau}$ 
%to factorizations, 
%while the second relates factorizations to arrangements. 

\begin{proposition}\label{prop:arrangementscnumbers} %{*
Let $\tau,\lambda \vDash d$ and 
let $f_\tau$ be a monomial of type $\tau$. 
The number of factorizations 
of $f_\tau$ of type $\lambda$ is equal to $\arr(\tau,\lambda)$. 
The number of factorizations 
of $f_\tau$ of type $\lambda$ 
using only squarefree monomials $f_1,\ldots,f_r$ 
is equal to $\arrsqf(\tau,\lambda)$. 
In particular, both quantities are independent of 
the chosen monomial $f_\tau$. 
\end{proposition} %*}

\begin{proof} %{*
%Let $M_\tau$ be the sum over all monomials of type $\tau$. 
Write $\lambda = \vec{d}^{\vec{m}}$. 
%The coefficient of $M_\tau$ in the expression 
In the expansion of the product 
\begin{equation} %{*
H_\lambda = 
\left( \sum_{\substack{\text{monomials $f_1$}\\\text{of degree $d_1$}} } 
   f_1^{m_1}
\right)\cdots 
\left( \sum_{\substack{\text{monomials $f_r$}\\\text{of degree $d_r$}} } 
   f_r^{m_r}
\right),
\end{equation} %*}
a monomial $f_\tau$ of type $\tau$ will appear 
each time $(f_1,\ldots,f_r;m_1,\ldots,m_r)$ is 
a factorization of $f_\tau$ of type $\lambda$. 
Collecting monomials according to their splitting type 
expresses $H_\lambda$ in the monomial polysymmetric basis. 
This proves that 
the number of factorizations of $f_\tau$ of type $\lambda$ 
equals $\arr(\tau,\lambda)$. 
Similarly, the coefficient of $M_\tau$ in the expression 
\begin{equation} %{*
E_\lambda^+ = 
\left( \sum_{\substack{\text{sq.free monomials}\\\text{$f_1$ of degree $d_1$}} } 
   f_1^{m_1}
\right)\cdots 
\left( \sum_{\substack{\text{sq.free monomials}\\\text{$f_r$ of degree $d_r$}} } 
   f_r^{m_r}
\right)
\end{equation} %*}
is the number of factorizations of $f_\tau$ of type $\lambda$ 
using only squarefree monomials. 
\end{proof} %*}

\begin{proposition}\label{prop:arrangementsmatrices} %{*
Let $\tau = \{(d_1,m_1), \ldots, (d_r, m_r)\}$ and 
$\lambda = \{(e_1,n_1),\ldots,(e_s,n_s)\}$. 
Let $f_\tau = x_{d_1k_1}^{m_1}\cdots x_{d_rk_r}^{m_r}$ 
be a monomial of type $\tau$. 
If $(f_{1},\ldots,f_{s};n_1,\ldots,n_s)$ 
is a factorization of $f_\tau$ of type $\lambda$, 
then the matrix $A$ defined by 
\begin{equation}\label{eqn:matrixtomonomial} %{*
f_j = x_{d_1k_1}^{A_{1j}}\cdots x_{d_rk_r}^{A_{rj}}
\end{equation} %*}
for all $j \in \{1,\ldots,s\}$ 
is an arrangement of $\tau$ into $\lambda$. 
This defines a bijection between the following sets: 
\begin{enumerate} %{*
\item factorizations of $f_\tau$ of type $\lambda$, 
\item arrangements of $\tau$ into $\lambda$. 
    %the set of ${r \times s}$ matrices $A$ with 
    %non-negative integer entries satisfying 
    %\begin{equation}\label{eqn:arrangementsmatrices} %{*
    %A\vec{n} = \vec{m}
    %\quad\text{and}\quad
    %A^T\vec{d} = \vec{e}. 
    %\end{equation} %*}
\end{enumerate} %*}
This restricts to a bijection 
between the following subsets: 
\begin{enumerate} %{*
\item[(1$\,'$)] factorizations of $f_\tau$ of type $\lambda$ 
using only squarefree monomials, 
\item[(2$\,'$)] arrangements of $\tau$ into $\lambda$ 
with entries in $\{0,1\}$. 
% ${r \times s}$ matrices $A$ with 
% entries in $\{0,1\} \subset \NN$ satisfying 
%     \eqref{eqn:arrangementsmatrices}. 
    %$$m_i = \sum_{b=1}^sA_{ib}n_b
    %\quad\text{and}\quad
    %e_j := \sum_{a=1}^rA_{aj}d_a
    %$$ 
    %whose $i$th row sums to a divisor of $m_i$ 
\end{enumerate} %*}
\end{proposition} %*}

\begin{proof} %{*
%Let $f_\tau = x_{d_1,k_1}^{m_1}\cdots x_{d_r,k_r}^{m_r}$ be 
%    a monomial of type $\tau$. 
Let $(f_{1},\ldots,f_{s};n_1,\ldots,n_s)$ be 
a factorization of $f_\tau$ of type $\lambda$. 
For each $j \in \{1,\ldots,s\}$ there are uniquely determined 
non-negative integers 
$A_{1j},\ldots,A_{rj}$ satisfying
$$f_j = x_{d_1k_1}^{A_{1j}}\cdots x_{d_rk_r}^{A_{rj}}$$ 
which shows that $A$ is well-defined. 
The equality $f_1^{n_1} \cdots f_s^{n_s} = f_\tau$ implies that 
\begin{equation} %{*
\sum_{j = 1}^s A_{ij}n_j = m_i
\end{equation} %*}
for each $i \in \{1,\ldots,r\}$. 
Since $\deg f_j = e_j$, we also have  
\begin{equation} %{*
\sum_{i = 1}^r A_{ij} d_i = e_j
\end{equation} %*}
for each $j \in \{1,\ldots,s\}$. 
%The matrix $A \coloneqq (A_{ij})$ is a solution to \eqref{eqn:arrangementsmatrices}. 
Conversely, it is clear that 
any arrangement $A$ of $\tau$ into $\lambda$ 
determines monomials $f_1,\ldots,f_s$ 
according to \eqref{eqn:matrixtomonomial} 
which give a factorization 
$(f_{1},\ldots,f_{s};n_1,\ldots,n_s)$ 
of $f_\tau$ of type $\lambda$. 

Restricting the factorization
$(f_{1},\ldots,f_{s};n_1,\ldots,n_s)$ defined by $A$ 
to use squarefree monomials 
is the same as restricting each exponent 
in \eqref{eqn:matrixtomonomial} to be $0$ or $1$. 
\end{proof} %*}

Now we finish the proof of 
Theorem~\ref{thm:arrangementsandcoefficients}. 

\begin{proof}[Proof of (i), (iv), and (v)]
The equalities in (iv) and (v) are realized by 
a bijection defined by matrix transposition. 
More precisely, 
$A \mapsto A^T$ 
gives a bijection 
from the set of arrangements of $\tau$ into $\lambda$ 
to the set of arrangements of $\lambda^t$ into $\tau^t$. 
This proves (iv) and (v) in view of (ii) and (iii). 
Since the set $\{H_\lambda\}_{\lambda}$ is 
a linear basis of $\PS$  
(Theorem~\ref{thm: H, E algebraic basis}),  
the symmetry of $\langle\cdot,\cdot\rangle$ 
is equivalent to (v). 
\end{proof}

\begin{remark}\label{rmk: category structure on types}
Arrangements may be regarded as morphisms between types 
    where composition is defined by matrix composition. 
This gives a refinement of the partial order 
    as $\tau \leq \lambda \iff \arr_{\tau \lambda} >0$. 
\end{remark}

\subsubsection{Self-arrangements} 

\begin{lemma}\label{lemma:self-arrangementsterminology}
    Let $\tau = (d_1^{m_1} d_2^{m_2}\ldots,d_r^{m_r})$ be a type. 
Let $\alpha$ be a self-arrangement of $\tau$.  
%and write $\beta$ for the inverse permutation of $\alpha$. 
Let $f_\tau = x_{d_1j_1}^{m_1}\cdots x_{d_rj_r}^{m_r}$ 
be a monomial of type $\tau$. 
Then 
$$a_\alpha = (x_{d_{1}j_{\alpha(1)}}, 
x_{d_{2}j_{\alpha(2)}},\ldots,
x_{d_{r}j_{\alpha(r)}};m_1,\ldots,m_r)$$ 
is a factorization of $f_\tau$ of type $\tau$. 
Conversely, every factorization of $f_\tau$ of type $\tau$ 
is of the form $a_\alpha$ for 
a unique self-arrangement $\alpha$ of $\tau$. 
\end{lemma}

\begin{proof} %{*
It is immediate to see that $a_\alpha$ 
is a factorization of $f_\tau$ of type $\tau$. 
Conversely, let $a$ be a factorization of $f_\tau$ of type $\tau$, 
and let $A$ be the arrangement of $\tau$ into $\tau$ 
corresponding to $a$ by 
    Proposition~\ref{prop:arrangementsmatrices}. 
Then $A$ is an $r \times r$ matrix with 
non-negative integer entries, and 
    $d$ and $m$ are positive vectors 
    such that  
    $m = Am$ and $d = A^Td$. 
We claim that this forces $A$ to be a permutation matrix. 
Multiplying $(A-1)m = 0$ on the left by 
the row vector $(1,1,\ldots,1)$ shows that 
$m_1(A_{1,1}+\cdots+A_{r,1}-1) + \cdots + 
m_r(A_{1,r}+\cdots+A_{r,r}-1) = 0$. 
By positivity of $m$, each column of $A$ must sum to $1$. 
Similarly, $(A^T-1)d = 0$ shows each row of $A$ sums to $1$, 
so $A$ is a permutation matrix. 
It is easy to see that 
$A$ preserves the degrees and multiplicities of $\tau$. 
\end{proof} %*}

%*}

\subsection{The `merge and forget' partial order}\label{sec:mergingforgetting} %{*

We introduce a partial order on the set of types. 
First we define two operations on types. \par
\begin{itemize}
\item[$\diamond$]  
    Suppose that two elements of a type 
        $\tau=\{(d_1,m_1),\ldots,(d_r,m_r)\}$, 
        say $(d_1,m_1)$ and $(d_2,m_2)$, 
        have the same multiplicity, i.e. $m_1 = m_2$. 
We say that the type 
$$M(\tau)=\{(d_1+d_2,m_1),(d_3,m_3),\ldots,(d_r,m_r)\}$$ 
    is obtained from $\tau$ 
    by an \defn{elementary merge} $M$.
\vspace{.25em}
\item[$\diamond$] Let an element of $\tau$ be chosen, 
say $(d_1,m_1)$. 
Let $a$ be a positive integer which is 
    strictly less than $m_1$. 
We say that the type 
    $$F(\tau) = \{(d_1,m_1-a),(d_1,a),(d_2,m_2),\ldots,(d_r,m_r)\}$$ 
is obtained from $\tau$ 
    by an \defn{elementary forget} $F$. 
        %(One ``forgets'' that $a$ of the $d_1$s 
        %are repeated in the first part.)
\end{itemize}
A composition of elementary merges (resp. forgets) 
is called a merge (resp. forget). 
These operations are dual under $\tau \mapsto \tau^t$. 
Let $\PT{d}$ denote 
the set of types of degree $d$ equipped with 
the partial order 
generated by merges and forgets, 
and let $\IT{d}$ denote its incidence algebra over $\QQ$. 
%\begin{remark}
%There are many natural order-respecting inclusions 
%between the partially ordered sets $T_1,T_2,\ldots$. 
%The simplest is $T_d \to T_{d+1}$ 
%obtained by appending $1^1$. 
%\end{remark}

\begin{proposition}\label{prop:arrangementforgetmerge}\,
    \begin{enumerate}
        \item Elementary merges 
    $M \colon \tau \to \lambda = M(\tau)$ 
    correspond to squarefree arrangements 
    $A \colon \tau \to \lambda$ 
    with exactly one $1$ in each column. 
    \item Elementary forgets 
    $F \colon \tau \to \lambda = F(\tau)$ 
    correspond to squarefree arrangements 
    $A \colon \tau \to \lambda$ 
    with exactly one $1$ in each row. 
\item Any arrangement is equal to a forget followed by a merge. 
\item The partial order on $T_d$ is equivalent to 
    the partial order obtained 
    by declaring $\tau \leq \lambda \iff a(\tau,\lambda)>0$. 
    In particular, the transition matrices 
        $\arr$ and $\arrsqf$ are in 
        the incidence algebra $\IT{d}$. 
\item The transition matrices $a$ and $e$ are invertible in 
    the incidence algebra $I_d$. 
    \end{enumerate}
\end{proposition}

For the proof we apply 
merge and forget operations to factorizations. 
If $\alpha$ is a factorization of type $\vec{b}^{\vec{m}}$, 
then any merge or forget 
which is admissible for $\vec{b}^{\vec{m}}$ 
may be applied to $\alpha$ in the obvious manner 
(for instance, 
we may merge the first two parts of 
$\alpha=(f_1,f_2,f_3;m,m,n)$ to get 
$M(\alpha)=(f_1f_2,f_3;m,n)$. 

\begin{proof} %{*
The first two assertions are immediate from the definitions. 
The fourth is a consequence of the third assertion. 
For the third assertion, let 
    $f = x_{d_1k_1}^{m_1}\cdots x_{d_rk_r}^{m_r}$ 
    be a monomial of type $\tau$ 
and let $\beta=(f_1,\ldots,f_s;n_1,\ldots,n_s)$ 
be a factorization of $f$ of type $\lambda$. 
%Each $f_i$ is a monomial in the variables 
%    $x_{d_1,k_1},\ldots,x_{d_r,k_r}$. 
Let $\alpha_0$ denote the factorization 
$(x_{d_1k_1},\ldots,x_{d_rk_r};m_1,\ldots,{m_r})$ 
of $f$ of type $\tau$. 
If $n_1$ is less than 
the multiplicities of every part in $\alpha_0$ 
corresponding to a variable appearing in the monomial $f_1$, 
then one can apply elementary forgets to $\alpha_0$ 
until each variable in 
$f_1$ corresponds to some part in 
the new factorization $\alpha_1$ with multiplicity $n_1$. 
Replacing $\alpha_0$ with $\alpha_1$, 
we repeat the above process for each $i \in \{2,\ldots,r\}$. 
In the final factorization $\alpha_r$ 
each variable appearing in each $f_i$ 
corresponds to some part in $\alpha_r$ 
with exact multiplicity $n_i$,  
and one can now merge parts of $\alpha_r$ 
to obtain $\beta$. 

For the last assertion, we use that 
that an element $f \in I_{d}$ is invertible 
if and only if 
$f(\tau,\tau) \neq$ for all $\tau$ 
(to see this one chooses 
    an upper-triangular matrix representation of $I_{d}$). 
The number of self-arrangements of a type $\tau$ is 
$\arr(\tau,\tau) = \arrsqf(\tau,\tau) = 
\prod_{k,m}\tau[k^m]! \neq 0$. 
\end{proof} %*}

%functions $f \colon \PT{d} \times \PT{d} \to \ZZ$ 
%which satisfy $f(\tau,\lambda) = 0$ 
%if $\tau \not \leq \lambda$. 
%This set is equipped with the multiplication defined by 
%$(f\cdot g)(\tau,\lambda) = 
%\sum_{\tau \leq \kappa \leq \lambda}
%f(\tau,\kappa) g(\kappa,\lambda)$. 
%The ring $\IT{d}$ is the 
%\defn{incidence algebra on types of degree $d$}. 

\begin{figure} %{*
\begin{minipage}{0.3\textwidth}	
\[
\adjustbox{scale=.7,center}{%
\begin{tikzcd}[every arrow/.append style={dash}]
              &  & (3)                 \\
(2)               &  & (2 1) \ar[u]    \\
(1 1) \ar[u] &  & (1 1 1) \arrow[u] \\
(1^2) \ar[u]   &  & (1^2 1) \arrow[u]  \\
              &  & (1^3) \ar[u]     \\
\end{tikzcd}}\]
\end{minipage}
\begin{minipage}{0.3\textwidth}	
\[
\adjustbox{scale=.7,center}{%
\begin{tikzcd}[column sep=.25cm,row sep=.4cm,every arrow/.append style={dash}]
                & (4)                              &                                  \\
(3 1) \ar[ru]   &                                  & (2 2) \ar[lu]                \\
                & (2 1 1) \ar[ru] \ar[lu]  &                                  \\
(2 1^2) \ar[ru] & (1 1 1 1) \ar[u]           & (2^2) \ar[uu]                 \\
                & (1^2 1 1) \ar[u] \ar[lu] &                                  \\
(1^3 1) \ar[ru] &                                  & (1^2 1^2) \ar[lu] \ar[uu] \\
                & (1^4) \ar[ru] \ar[lu]      &                                  \\
\end{tikzcd}}\]
\end{minipage}
\begin{minipage}{0.3\textwidth}	
\[
\adjustbox{scale=.7,center}{%
\begin{tikzcd}[column sep=.25cm,row sep=.4cm,every arrow/.append style={dash}]
           &  (5)    &\\
(41) \ar[ur] &                               & (32) \ar[ul]\\
(221) \ar[urr] \ar[u] &                 & (311) \ar[u] \ar[ull] \\
&        (2111) \ar[ur]\ar[ul] &(31^2) \ar[u] \\
(2^21) \ar[uu]  &        (11111) \ar[u]     &(21^21) \ar[u] \ar[ul] \\
&        (1^2111)\ar[u]\ar[ur]      &(21^3) \ar[u] \\
(1^21^21) \ar[ur]\ar[uu]&                               & (1^311)\ar[u]\ar[ul]\\
(1^41) \ar[urr]\ar[u]&                               & (1^31^2)\ar[u]\ar[ull]\\
                       &  (1^5) \ar[ru]\ar[lu]   &\\
\end{tikzcd}}\]
\end{minipage}
\caption{The partially ordered sets of types of degrees $2,3,4,5$.}
\end{figure} %*}

%\begin{example} %{*
%The following arrangement gives 
%an elementary merge $M$ 
%from $(d_1^{n_1}d_2^{n_2}d_3^{n_2}d_4^{n_3})$ to $(d_1^{n_1}(d_2+d_3)^{n_2}d_4^{n_3})$:  
%\begin{equation} %{*
%\begin{bmatrix}
%    1&0&0&0 \\
%    0&1&1&0 \\
%    0&0&0&1 
%\end{bmatrix}
%\begin{bmatrix}
%    d_1\\
%    d_2\\
%    d_3\\
%    d_4
%\end{bmatrix}
%=
%\begin{bmatrix}
%    d_1\\
%    d_2+d_3\\
%    d_4
%\end{bmatrix}
%\quad
%\text{and}
%\quad
%\begin{bmatrix}
%    1&0&0 \\
%    0&1&0 \\
%    0&1&0 \\
%    0&0&1
%\end{bmatrix}
%\begin{bmatrix}
%    n_1\\
%    n_2 \\
%    n_3
%\end{bmatrix}
%=
%\begin{bmatrix}
%    n_1\\
%    n_2\\
%    n_2\\
%    n_3
%\end{bmatrix}. 
%\end{equation} %*}
%Under duality, the same matrix equations define 
%an elementary forget $F = M^T$ 
%from $(n_1^{d_1}n_2^{d_2+d_3}n_3^{d_4})$ to 
%$(n_1^{d_1}n_2^{d_2}n_2^{d_3}n_3^{d_4})$.   
%\end{example} %*}

\begin{remark} %{*
$\arr(\tau,\tau) = \arrsqf(\tau,\tau) = 
    \prod_{k,m}\tau[k^m]!$ 
implies $\arr^{-1}$ and $\arrsqf^{-1}$ 
are valued in $\ZZ[\frac{1}{d!}]$. 
\end{remark} %*}

%\begin{remark}
%Arrangements are related to 
%the bilinear form on $\PS$ as follows: 
%\begin{equation} %{*
%\langle H_\lambda, H_{\tau^t} \rangle = \arr_{\tau \lambda} 
%\quad\text{and}\quad
%\langle M_\lambda, M_{\tau^t} \rangle = \arr^{-1}_{\tau \lambda} .
%\end{equation} %*}
%\end{remark}

%*}

\subsection{Inversion on the top stratum}\label{sec:inversiontopstratum} %{* 

In this section we use higher plethysm 
to deduce an explicit formula
for inverse arrangement numbers $\arr^{-1}_{\tau d}$ 
on the top stratum. 
Let us call $\tau$ \defn{$m$-pure} 
if every part of $\tau$ has multiplicity $m$, 
and \defn{mixed} otherwise. 

\begin{theorem}\label{thm:formulatopstratum}
Let $\tau$ be a type of degree $d$. 
\begin{enumerate} %{*
\item $\arr^{-1}_{\tau d}$ vanishes if $\tau$ is mixed. 
\item If $\tau = (b_1^m \cdots b_r^m)$ is $m$-pure, then  
\begin{equation}%\label{eqn:invarrformula} %{*
\arr^{-1}_{\tau d} = 
    \frac{\mu(m)}{m} \frac{(-1)^{r-1}}{r} 
    \frac{r!}{\tau[1^m]!\cdots \tau[(d/m)^m]!} . 
\end{equation} %*}
\end{enumerate} %*}
\end{theorem}

\begin{proof}
We begin with the first assertion. 
Recall that $\PS_\QQ$ has the linear basis $\{H_\tau\}_\tau$. 
Define the linear projection operator $\Theta \colon \PS_\QQ \to \PS_\QQ$ given by 
\begin{equation} %{*
\Theta(H_\tau)  = 
\begin{cases}
    H_\tau &\text{if $\tau$ is pure of some multiplicity,}\\
    0 &\text{otherwise.}
\end{cases}
\end{equation} %*}
Recall that 
$1 + H_1 t + H_2 t^2 + H_3 t^3+ \cdots = 
\exp(P_1 + P_2\frac{t^2}{2}+P_3\frac{t^3}{3}+\cdots)$. 
Applying the $m$th Adams operation acting trivially on $t$ 
and taking a logarithm results in 
\begin{equation}\label{eqn:polysymmetrichomogpower} %{*
\log(1 + H_{1^m} t + H_{2^m} t^2 + H_{3^m} t^3+ \cdots) = 
P_{1^m} + P_{2^m}\frac{t^2}{2}+P_{3^m}\frac{t^3}{3}+\cdots. 
\end{equation} %*}
The coefficient of any power of $t$ on the left-hand side 
is of the form $H_\tau$ for a pure type $\tau$ 
so both sides are fixed by $\Theta$. 
We conclude that $\Theta(P_{k^m}) = P_{k^m}$ 
for all $k$ and $m$. 
Therefore 
\begin{equation} %{*
\Theta(M_d) = 
\Theta\left( 
\frac{1}{d} \sum_{k|d} \mu(\tfrac{d}{k}) P_{k^{d/k}}
\right)
=
\frac{1}{d} \sum_{k|d} \mu(\tfrac{d}{k}) P_{k^{d/k}}
=M_d
\end{equation} %*}
and then 
\begin{equation} %{*
\sum_{\tau \leq d} \arr^{-1}_{\tau d}
H_\tau  =  
M_d = \Theta(M_d) = 
%\sum_{\tau \leq d} \arr^{-1}_{\tau d}
%\Phi(H_\tau)
%=
\sum_{\substack{\tau \leq d\\\tau \text{ pure}}} \arr^{-1}_{\tau d}
H_\tau . 
\end{equation} %*}
The set $\{H_\tau\}_\tau$ is linearly independent, 
so this shows that $\arr^{-1}_{\tau d} = 0$ 
if $\tau$ is mixed. 

Now we prove the second assertion. 
Let $R$ denote the $\lambda$-ring $\Lambda_\QQ$, 
equipped with the unique $\lambda$-structure 
for which it has trivial Adams operations. 
Since $\PS$ is the universal graded $\lambda$-ring, 
there is a unique $\lambda$-homomorphism 
satisfying 
\begin{align} %{*
\Phi \colon \PS &\to R\\
M_d&\mapsto q_d
\end{align} %*}
where $q_1,q_2,q_3,\ldots \in \Lambda_\QQ$ 
are defined by 
\begin{equation} %{*
\sum_{d=0}^\infty {h_d t^d} = \prod_{d=1}^\infty (1-t^d)^{-q_d}.
\end{equation} %*}
We have shown earlier that 
$\Phi(H_{k^m}) = h_k$ 
and 
$\Phi(P_{k^m}) = p_k$ 
for all $k,m \geq 1$ 
(Proposition~\ref{prop:phiimagesHP}). 
Applying $\Phi$ to  
$M_d = 
\frac{1}{d} \sum_{k|d} \mu(\tfrac{d}{k}) \psi_{d/k} P_{k}$ 
results in 
\begin{equation}\label{eqn:qdintermsofhd} %{*
\Phi(M_d) = q_d = 
\frac{1}{d} \sum_{k|d} \mu(\tfrac{d}{k}) p_k
=
\sum_{\tau \leq d} \arr^{-1}_{\tau d}
\prod_{b^m \in \tau}h_b  
=
\sum_{m|d}
\sum_{\substack{\tau \leq d\\\tau \text{ $m$-pure}}} 
\arr^{-1}_{\tau d}
\prod_{b^m \in \tau}h_b .  
\end{equation} %*}
This equation is inhomogeneous; 
we take the degree $d/m$ piece to see that 
\begin{equation}\label{eqn:inhomogeneousph} %{*
\tfrac{1}{d} \mu(m) p_{d/m}
=
\sum_{\substack{\tau \leq d\\\tau \text{ $m$-pure}}} 
\arr^{-1}_{\tau d}
\prod_{b^m \in \tau}h_b. 
\end{equation} %*}
Now we compare with the expansion of power symmetric functions 
in terms of the basis $\{h_\lambda\}_{\lambda \text{ partition}}$. 
Expanding the Taylor series for the logarithm 
in the relation 
$\log(1 + h_1 t + h_2 t^2 + h_3 t^3+ \cdots) = 
p_1 + p_2\frac{t^2}{2}+p_3\frac{t^3}{3}+\cdots$ 
shows that 
\begin{equation}\label{eqn:pnintermsofh} %{*
p_m %= M_m(h_1,\ldots,h_m)
=
\sum_{\substack{n_1+2n_2+\cdots + mn_m = m\\n_1,\ldots,n_m \geq 0}}
\frac{m(-1)^{1+\sum_k n_k}}{\sum_k n_k}
\frac{(\sum_k n_k) !}{n_1! \cdots n_m!}
\prod_{k=1}^m
h_k^{n_k}. 
\end{equation} %*}
We observe two consequences of \eqref{eqn:inhomogeneousph}: 
\begin{itemize} %{*
\item[$\diamond$] if $m$ is square-free 
    and $\tau$ is $m$-pure, 
    then $\mu(m) d \arr^{-1}_{\tau d}$ 
    is the coefficient of $\prod_{b^m \in \tau}h_b$ 
    in the linear expansion of $p_{d/m}$ 
    in the basis $\{h_\lambda\}_{\lambda \text{ partition}}$, 
    and 
\item[$\diamond$] if $m$ is not square-free 
    and $\tau$ is $m$-pure, 
    then $\arr^{-1}_{\tau d}$ is zero. 
\end{itemize} %*}
In either event, 
    if $\tau$ is $m$-pure then 
    $\arr^{-1}_{\tau d}$ 
    is equal to $\mu(m)/d$ times 
    the coefficient of $\prod_{b^m \in \tau}h_b$ 
    in the linear expansion of $p_{d/m}$, i.e. 
\begin{equation} %{*
\arr^{-1}_{\tau d} = 
\frac{\mu(m)}{d} 
\frac{(d/m)(-1)^{1+\sum_k n_k}}{\sum_k n_k}
\frac{(\sum_k n_k) !}{n_1! \cdots n_{d/m}!}
=
    \frac{\mu(m)}{m} \frac{(-1)^{r-1}}{r} 
    \frac{r!}{\tau[1^m]!\cdots \tau[(d/m)^m]!} . 
\end{equation} %*}
where $n_k = \tau[k^m]$ denotes 
the number of occurrences of $k^m$ in $\tau$. 
\end{proof}

\subsubsection{An observation on the M\"obius function of types} 

Call a type \defn{unramified} 
if all of its multiplicities are one 
and \defn{ramified} otherwise. 

\begin{proposition}
    $\mu_{\tau d} = 0$ if $\tau$ is ramified. 
\end{proposition}

\begin{proof} %{*
Recall that $\mu$ is defined to satisfy 
    $\delta_{\tau \lambda}
    =
    \sum_{\tau \leq \kappa \leq \lambda}
    \mu_{\kappa \lambda}$. 
Any type $\tau$ has a unique unramified type $\tau'$ 
    defined by applying all possible maximal forget operations 
    to $\tau$ (e.g.~if $\tau = (3^2 2^3 2^1)$ 
    then $\tau' = (332222)$); 
the type $\tau'$ has the defining property that 
it is minimal among all unramified types 
    that are greater than or equal to $\tau$ 
    (this follows from 
    Proposition~\ref{prop:arrangementforgetmerge}). 

%Let $T_d$ denote the partially ordered set 
%    of types of degree $d$. 
Choose any function $j \colon T_d \to \NN$ 
    satisfying $\tau < \lambda \implies j(\tau) < j(\lambda)$. 
Define the \defn{depth} of $\tau$ to be $j(d) - j(\tau)$. 
We proceed by induction on depth. 
The claim is vacuously true if $\tau = d$ (depth zero). 
For positive depth, first observe that 
\begin{equation} %{*
    \delta_{\tau d}
    =
    \sum_{\substack{\tau \leq \lambda\\\lambda\text{ ramified}}}
    \mu_{\lambda d}
    +
    \sum_{\tau' \leq \lambda}
    \mu_{\lambda d}.
\end{equation} %*}
If $\tau$ is unramified then there is nothing to show, 
while if $\tau$ is ramified then 
by the induction hypothesis this is equal to 
    $\mu_{\tau d}
    +
    \sum_{\tau' \leq \lambda}
    \mu_{\lambda d}
    =
    \mu_{\tau d}
    +
    \delta_{\tau' d}
    =
    \mu_{\tau d}$. 
We conclude that $\mu_{\tau d} = \delta_{\tau d} = 0$. 
\end{proof} %*}

%*} 

\subsection{Grids and tilings}\label{sec:gridsAndTilings} %{* 

We sketch a pictorial interpretation for types and arrangements. 
Types will be associated with diagrams of boxes (``grids''). 
For example, 
   the types $\mu = (1^3 1^2 2)$ and $\lambda = (1^2 1^2 1^1 2^1)$ 
   are associated with the two grids, respectively: 
\begin{figure}[h!] %{*
%\fbox{
\begin{minipage}[b]{.35\textwidth} %{* 
\centering
\begin{tikzpicture}[every node/.style={minimum size=.5cm-\pgflinewidth, outer sep=0pt}]
   \draw[step=0.5cm,color=black] (0,0) grid (0.5,1.5);
   %\node[fill=cl1!40] at (0.25,0.25) {};
   %\node[fill=cl1!40] at (0.25,0.75) {};
   %\node[fill=cl1!40] at (0.25,1.25) {};
\end{tikzpicture}
\begin{tikzpicture}[every node/.style={minimum size=.5cm-\pgflinewidth, outer sep=0pt}]
   \draw[step=0.5cm,color=black] (0,0) grid (0.5,1);
   %\node[fill=cl2!40] at (0.25,0.25) {};
   %\node[fill=cl2!40] at (0.25,0.75) {};
\end{tikzpicture}
\begin{tikzpicture}[every node/.style={minimum size=.5cm-\pgflinewidth, outer sep=0pt}]
   \draw[step=0.5cm,color=black] (0,0) grid (1,.5);
   %\node[fill=cl3!40] at (0.25,0.25) {};
   %\node[fill=cl3!40] at (0.75,0.25) {};
\end{tikzpicture}
\end{minipage}%
%    \fbox{
\begin{minipage}[b]{.05\textwidth}
   \centering
   \,
\end{minipage} %*} 
%    }
   \begin{minipage}[b]{.35\textwidth} %{* 
\centering
\begin{tikzpicture}[every node/.style={minimum size=.5cm-\pgflinewidth, outer sep=0pt}]
   \draw[step=0.5cm,color=black] (0,0) grid (0.5,1);
   %\node[fill=cl5!40] at (0.25,0.25) {};
   %\node[fill=cl5!40] at (0.25,0.75) {};
\end{tikzpicture}
\begin{tikzpicture}[every node/.style={minimum size=.5cm-\pgflinewidth, outer sep=0pt}]
   \draw[step=0.5cm,color=black] (0,0) grid (0.5,1);
   %\node[fill=cl6!40] at (0.25,0.25) {};
   %\node[fill=cl6!40] at (0.25,0.75) {};
\end{tikzpicture}
\begin{tikzpicture}[every node/.style={minimum size=.5cm-\pgflinewidth, outer sep=0pt}]
   \draw[step=0.5cm,color=black] (0,0) grid (0.5,0.5);
   %\node[fill=cl7!40] at (0.25,0.25) {};
\end{tikzpicture}
\begin{tikzpicture}[every node/.style={minimum size=.5cm-\pgflinewidth, outer sep=0pt}]
   \draw[step=0.5cm,color=black] (0,0) grid (1,.5);
   %\node[fill=cl8!40] at (0.25,0.25) {};
   %\node[fill=cl8!40] at (0.75,0.25) {};
\end{tikzpicture}
   \end{minipage} %*} 
%    }
\end{figure} %*}

\begin{definition}
%A \defn{block} is a rectangle made of (unit) boxes, 
%    specified by its dimensions $(w,h)$. 
%A \defn{grid} is any finite ordered list of blocks. 
%Two grids are \defn{equivalent} if they are 
%    the same up to reordering. 
%Two boxes of $\mu$ are \defn{horizontally} (resp. \defn{vertically}) \defn{adjacent} 
%    if they are in the same row (resp. column) of the same block. 
A \defn{tiling of $\lambda$ by $\mu$} 
is a bijection between the boxes of $\mu$ and the boxes of $\lambda$ 
satisfying two properties: 
\begin{enumerate} %{*
   \item horizontally adjacent boxes in $\mu$ map to horizontally adjacent boxes in $\lambda$, 
   \item vertically adjacent boxes in $\lambda$ come from vertically adjacent boxes in $\mu$. 
\end{enumerate} %*}
Two tilings of $\lambda$ by $\mu$ are considered equivalent if 
they differ by a permutation of the boxes of $\lambda$ 
   which stabilizes every block and 
   permutes the set of columns of a given block. 
\end{definition}

For example, the following two tilings of 
$\lambda = (1^2 1^2 1^1 2^1)$ by $\mu = (1^3 1^2 2)$ 
are inequivalent: \\[-1em]
\begin{figure}[h!] %{*
%\fbox{
\begin{minipage}[c]{0.25\textwidth}
\centering
\begin{tikzpicture}[every node/.style={minimum size=.5cm-\pgflinewidth, outer sep=0pt}]
   \draw[step=0.5cm,color=black] (0,0) grid (0.5,1.5);
   \node[fill=cl1!40] at (0.25,0.25) {};
   \node[fill=cl1!40] at (0.25,0.75) {};
   \node[fill=cl1!40] at (0.25,1.25) {};
\end{tikzpicture}
\begin{tikzpicture}[every node/.style={minimum size=.5cm-\pgflinewidth, outer sep=0pt}]
   \draw[step=0.5cm,color=black] (0,0) grid (0.5,1);
   \node[fill=cl2!40] at (0.25,0.25) {};
   \node[fill=cl2!40] at (0.25,0.75) {};
\end{tikzpicture}
\begin{tikzpicture}[every node/.style={minimum size=.5cm-\pgflinewidth, outer sep=0pt}]
   \draw[step=0.5cm,color=black] (0,0) grid (1,.5);
   \node[fill=cl3!40] at (0.25,0.25) {};
   \node[fill=cl3!40] at (0.75,0.25) {};
\end{tikzpicture}
\end{minipage}%
%    }
%    \fbox{
\begin{minipage}[b]{0.05\textwidth}
   \centering
   {\Large \textrightarrow }
\end{minipage}%
%    }
%\fbox{
\begin{minipage}[t]{0.3\textwidth}
\centering
   \vspace{-.4cm}
\begin{tikzpicture}[every node/.style={minimum size=.5cm-\pgflinewidth, outer sep=0pt}]
   \draw[step=0.5cm,color=black] (0,0) grid (0.5,1);
   \node[fill=cl1!40] at (0.25,0.25) {};
   \node[fill=cl1!40] at (0.25,0.75) {};
\end{tikzpicture}
\begin{tikzpicture}[every node/.style={minimum size=.5cm-\pgflinewidth, outer sep=0pt}]
   \draw[step=0.5cm,color=black] (0,0) grid (0.5,1);
   \node[fill=cl2!40] at (0.25,0.25) {};
   \node[fill=cl2!40] at (0.25,0.75) {};
\end{tikzpicture}
\begin{tikzpicture}[every node/.style={minimum size=.5cm-\pgflinewidth, outer sep=0pt}]
   \draw[step=0.5cm,color=black] (0,0) grid (0.5,0.5);
   \node[fill=cl1!40] at (0.25,0.25) {};
\end{tikzpicture}
\begin{tikzpicture}[every node/.style={minimum size=.5cm-\pgflinewidth, outer sep=0pt}]
   \draw[step=0.5cm,color=black] (0,0) grid (1,.5);
   \node[fill=cl3!40] at (0.25,0.25) {};
   \node[fill=cl3!40] at (0.75,0.25) {};
\end{tikzpicture}
\end{minipage}
\begin{minipage}[t]{0.25\textwidth}
   \centering
   \vspace{-.4cm}
\begin{tikzpicture}[every node/.style={minimum size=.5cm-\pgflinewidth, outer sep=0pt}]
   \draw[step=0.5cm,color=black] (0,0) grid (0.5,1);
   \node[fill=cl2!40] at (0.25,0.25) {};
   \node[fill=cl2!40] at (0.25,0.75) {};
\end{tikzpicture}
\begin{tikzpicture}[every node/.style={minimum size=.5cm-\pgflinewidth, outer sep=0pt}]
   \draw[step=0.5cm,color=black] (0,0) grid (0.5,1);
   \node[fill=cl1!40] at (0.25,0.25) {};
   \node[fill=cl1!40] at (0.25,0.75) {};
\end{tikzpicture}
\begin{tikzpicture}[every node/.style={minimum size=.5cm-\pgflinewidth, outer sep=0pt}]
   \draw[step=0.5cm,color=black] (0,0) grid (0.5,0.5);
   \node[fill=cl1!40] at (0.25,0.25) {};
\end{tikzpicture}
\begin{tikzpicture}[every node/.style={minimum size=.5cm-\pgflinewidth, outer sep=0pt}]
   \draw[step=0.5cm,color=black] (0,0) grid (1,.5);
   \node[fill=cl3!40] at (0.25,0.25) {};
   \node[fill=cl3!40] at (0.75,0.25) {};
\end{tikzpicture}
\end{minipage}
%    }
\end{figure} %*}

\FloatBarrier
The following three tilings of 
$\lambda = (3)$ by $\mu = (1^2 1)$ are equivalent: 
\begin{figure}[H] %{*
\centering
\begin{minipage}[c]{0.15\textwidth}
\begin{tikzpicture}[every node/.style={minimum size=.5cm-\pgflinewidth, outer sep=0pt}]
   \draw[step=0.5cm,color=black] (0,0) grid (0.5,1);
   \node[fill=cl1!40] at (0.25,0.25) {};
   \node[fill=cl1!40] at (0.25,0.75) {};
\end{tikzpicture}
\begin{tikzpicture}[every node/.style={minimum size=.5cm-\pgflinewidth, outer sep=0pt}]
   \draw[step=0.5cm,color=black] (0,0) grid (0.5,0.5);
   \node[fill=cl2!40] at (0.25,0.25) {};
\end{tikzpicture}
\end{minipage}
\begin{minipage}[c]{0.1\textwidth}
{\Large \textrightarrow}
\end{minipage}
\begin{minipage}[c]{0.15\textwidth}
\begin{tikzpicture}[every node/.style={minimum size=.5cm-\pgflinewidth, outer sep=0pt}]
   \draw[step=0.5cm,color=black] (0,0) grid (1.5,.5);
   \node[fill=cl1!40] at (0.25,0.25) {};
   \node[fill=cl1!40] at (0.75,0.25) {};
   \node[fill=cl2!40] at (1.25,0.25) {};
\end{tikzpicture}
\end{minipage}
\begin{minipage}[c]{0.15\textwidth}
\begin{tikzpicture}[every node/.style={minimum size=.5cm-\pgflinewidth, outer sep=0pt}]
   \draw[step=0.5cm,color=black] (0,0) grid (1.5,.5);
   \node[fill=cl1!40] at (0.25,0.25) {};
   \node[fill=cl2!40] at (0.75,0.25) {};
   \node[fill=cl1!40] at (1.25,0.25) {};
\end{tikzpicture}
\end{minipage}
\begin{minipage}[c]{0.15\textwidth}
\begin{tikzpicture}[every node/.style={minimum size=.5cm-\pgflinewidth, outer sep=0pt}]
   \draw[step=0.5cm,color=black] (0,0) grid (1.5,.5);
   \node[fill=cl2!40] at (0.25,0.25) {};
   \node[fill=cl1!40] at (0.75,0.25) {};
   \node[fill=cl1!40] at (1.25,0.25) {};
\end{tikzpicture}
\end{minipage}
\end{figure} %*}

%\noindent
%as they be obtained from one another by permuting columns. 

The following assertions are easily verified: 
    types of degree $d$ are in bijection 
    with equivalence classes of grids of area $d$; 
    arrangements are in bijection 
    with equivalence classes of tilings; 
    the duality of types corresponds to rotating 
    the blocks of a grid by 90 degrees; 
    any tiling of $\lambda$ by $\mu$ is sent to 
    a tiling of $\mu^t$ by $\lambda^t$ under the rotation of grids. 

%*}

%*}

\section{Applications} %{*

In this section we survey some geometric and 
    combinatorial applications. 
We will focus on the top stratum 
    since general strata are studied in 
    the next section. 
    %the formulas extend to any stratum 
    %by inverting arrangement numbers 
    %in the incidence algebra of types. 

\subsection{The motivic class of $\Irr_{d,n}$} %{*

The Grothendieck ring of varieties $\KV{k}$ over a field $k$ 
is additively generated 
by isomorphism classes $[X]$ of varieties $X$ over $k$ 
modulo the relations $[X] - ([X\backslash Z] + [Z])$ 
for $Z \subset X$ a closed subvariety. 
Its ring multiplication is defined by $[X][Y]=[X \times Y]$. 
%Classes in $\KV{k}$ are sometimes written without brackets. 
The ring $\KV{k}$ has a $\sigma$-structure given by\footnote{Many authors equip $\KV{k}$ with the opposite $\lambda$-structure, 
however our convention is in better analogy 
with the standard identification of $\lambda$-operations 
with exterior powers on vector bundles.} 
\begin{equation} %{*
    \sigma_t(X) \coloneqq Z_X(t) 
    = \sum_{n \geq 0} [\SymProd{X}{n}] \cdot t^n 
\end{equation} %*}
where $\SymProd{X}{n}$ is 
the unordered configuration space of degree $n$ on $X$. 
Note that %the associated $\lambda$-operations do not behave as expected 
$\lambda_n(X) \neq \mathrm{Conf}_n(X)$ in general. 

\begin{remark}
It is known that $\KV{k}$ is not special. 
Indeed Larsen--Lunts~\cite{MR2098401} have shown that 
    $Z_{C_1 \times C_2}(t)$ is {irrational} 
    for curves $C_i$ of positive genus, 
while $Z_{C_1}\ast Z_{C_2}$ (Witt product) 
is rational since $Z_{C_i}$ is rational. 
    %\note{clarify this remark}
\end{remark}

The coefficients of $\Exp([\Uc_1]t)$ are known to have 
    a geometric interpretation by \cite[Theorem~1]{MR2046199}. 
The next proposition extends this to any effective element of 
    $t\KV{k}[\![t]\!]$. 

\begin{proposition}\label{prop:zetafngradedspace}
Let $\Uc_1,\Uc_2,\ldots$ be a collection of varieties. 
For each nonnegative integer $k$ let $\Sc_k$ 
    be the unordered configuration space of degree $k$ on 
    $\sqcup_d \Uc_d$ where points of $\Uc_d$ have weight $d$. 
Then in $\KV{k}[\![t]\!]$ we have the identity 
\begin{equation} %{*
    \mathrm{Exp}\left([\Uc_1]t+[\Uc_2]t^2+\cdots\right)
    =
    \sum_{k=0}^\infty [\Sc_k] t^k. 
\end{equation} %*}
\end{proposition}

\begin{proof} %{*
By the first two identities 
    of the list in \S\ref{sec:canonicalbases}, 
\[
    \sum_{n= 0}^\infty H_nt^n = 
    \prod_{d= 1}^\infty
    \sigma_{t^d}(M_d) = 
    \prod_{d= 1}^\infty 
    \sum_{k=0}^\infty
        (\sigma_k M_d) t^{kd}  
        =
    \prod_{d= 1}^\infty
    \sum_{k=0}^\infty
        h_{k,(d)}  t^{kd} . 
\]
Let $\SymProd{\Uc_d}{k}$ denote 
    the $k$th symmetric product of $\Uc_d$. 
Let $\Uc = [\Uc_1]t+[\Uc_2]t^2+\cdots$. 
Applying the higher plethysm $- \circ \Uc$ obtains 
\begin{align} %{*
    \sum_{n= 0}^\infty (H_n \circ \Uc) t^n 
    = 
    %\prod_{p\geq 1} 
    %\sum_{k=0}^\infty
    %    \big((\sigma_k M_p ) \circ \Uc\big)t^{kp}
    \prod_{d=1}^\infty 
    \sum_{k=0}^\infty
    \big( h_{k,(d)} \circ \Uc \big) t^{kd}
    =
    \prod_{d=1}^\infty 
    \sum_{k=0}^\infty
   \big( h_k\circ \Uc_d \big) t^{kd}
    =
    \prod_{d=1}^\infty 
    \sum_{k=0}^\infty
       [\SymProd{\Uc_d}{k}] t^{kd}.
    %= 
    %Z_\Uc(t) .
\end{align} %*}
    The coefficient of $t^d$ is equal to 
\begin{equation} %{*
\bigsqcup_{\substack{n_1+2n_2+\cdots+dn_d\\n_1,\ldots,n_d \geq 0}}
\prod_{k=1}^d
    [S^{n_k} \Uc_k]
    = [\Sc_d],
\end{equation} %*}
which proves the claim. 
\end{proof} %*}

Next we combine the proposition with Theorem~\ref{thm:maintheorem}. 
Let $$\mm \colon \KV{k} \to R$$ 
be a $\lambda$-ring homomorphism 
to a $\lambda$-ring $R$ without additive torsion 
(``motivic measure''). 
%Let $\lambda \vdash m$ denote a partition 
%of degree $m$ and length $\ell = \ell(\lambda)$, 
%i.e.~a collection of integers 
%$n_1,\ldots,n_m \geq 0$ 
%satisfying $n_1+2n_2+\cdots + mn_m = m$ 
%and $n_1+n_2+\cdots + n_m = \ell$. 

\begin{corollary}\label{thm:irrdnmotivicformulageneral2} %{*
If $\mm(\Sc_1), \ldots,\mm(\Sc_d)$ are contained 
   in a $\lambda$-subring of $R$ 
   on which Adams operations separate 
    (e.g. $R$ special), then 
\begin{equation}\label{eqn:irrdnmotivicformulatoptype2} %{*
\mm(\Uc_d)=
\sum_{d=km}
\sum_{\lambda \vdash m}
    (-1)^{\ell-1}\frac{\mu(k)}{k\ell}
    \binom{\ell}{n_1,\ldots,n_m}
    \psi_{k}
    \prod_{j=1}^m
        \mm(\Sc_j)^{n_j}. 
\end{equation} %*}
\end{corollary} %*}

\begin{remark}
Every denominator in 
    \eqref{eqn:irrdnmotivicformulatoptype2} 
    is a divisor of $d$ 
    (this follows from \eqref{eqn:pnintermsofh} 
    and the fact that the $h_k$ 
    are an integral basis of $\Lambda$). 
\end{remark}

\begin{proof}[Proof of Theorem~\ref{thm:GeomIrred}, top stratum] %{*
The special case of Theorem~\ref{thm:GeomIrred} is 
\begin{align} %{*
    \Sc_d &= \PP^{\binom{n+d}{d}-1}_k
     = \{\text{degree $d$ hypersurfaces in $\PP^n_k$}\},\\
    \rotatebox[origin=c]{90}{$\subset$} \,\,& \\
    \Uc_d &= 
     \{\text{geometrically irreducible hypersurfaces}\} .
\end{align} %*}
The motivic measure $\xi$ 
    is the canonical map from $\KV{k}$ 
    to its quotient modulo additive torsion. 
The Adams operations separate over 
    the polynomial subring 
    generated by $[\AA^1_k]$, 
    and this subring contains $\xi([\Sc_1]),\ldots,\xi([\Sc_d])$. 
This proves the formula for $[\Irr_{d,n}]$ in $\KV{k}$ 
    up to additive torsion. 
Let $\Sc_\tau = \overline{\Uc_\tau}$ 
    be the space of hypersurfaces 
    of splitting type $\leq \tau$. 
We have that $[\Sc_\tau]=\sum_{\lambda \leq \tau}[\Uc_\lambda]$ 
    so by M\"obius inversion 
    $[\Irr_{d,n}] =[\Uc_d]
    = \sum_{\tau \leq d} \mu_{\tau d}[\Sc_\tau]$. 
Since each $\Sc_\tau$ is a projective space, 
    this shows that $[\Irr_{d,n}]$ is contained 
    in the polynomial subring of $\KV{k}$ generated by $[\AA^1_k]$. 
This subring contains no additive torsion, 
    and the formula for $[\Irr_{d,n}]$ is proven. 
\end{proof}

We give two examples of motivic measures $\xi$ to which 
    the theorem may be applied. 

%*}

\subsubsection{Example 1: $E$-polynomials} %{*

The polynomial ring $\mathbb Z[u,v]$ 
has a $\sigma$-structure 
defined by $\sigma_n(u^pv^q) = u^{np}v^{nq}$ 
which is special. 
It is well-known that there is a unique $\lambda$-ring homomorphism 
$\Ep{} \colon \KV{\CC} \to \ZZ[u,v] \colon x \mapsto \Ep{x}$ 
such that for any variety $X$ one has 
$$\Ep{X} = 
\sum_{p,q,k \geq 0}
%\sum_{k \geq 0}
    (-1)^k 
    h^{p,q}(H^k_c(X,\QQ)) 
    u^p v^q$$ 
    where 
    $h^{p,q}(H) = \dim \mathrm{Gr}_F^p \mathrm{Gr}^{p+q}_W H_\CC$ 
    for a mixed Hodge structure $(H,F,W)$. 

%*}

\subsubsection{Example 2: Zeta function over a finite field} %{*

Let $k = \FF_q$. 
%Although point counting defines a ring homomorphism 
%$\KV{\FF_q} \to \ZZ$ 
%it is not a $\lambda$-ring homomorphism, e.g. 
%$\sigma_2(\mathbb A^1) = \mathbb A^2$ but 
%$\mm_q(\sigma_2(\mathbb A^1)) 
%= q^2 \neq \frac{q(q+1)}{2} = \sigma_2(\mm_q(\mathbb A^1))$. 
%However, the zeta function, 
%which records point counts over all finite extensions of $\FF_q$, 
%does extend to a motivic measure 
%and can be used in place of $\mm_q$ 
%    to obtain point counting formulas. 
Recall (Example~\ref{ex:wittvectors}) that $\Lambda(\QQ)$ 
is a special $\lambda$-ring. 
Its Adams operations satisfy 
\[\psi_m\exp\left(\sum_{d=1}^\infty w_d\frac{t^d}{d}\right) 
= \exp\left(\sum_{d=1}^\infty w_{dm}\frac{t^d}{d}\right).\]
The function sending a variety $X/k$ to 
its zeta function $\zeta_X(t)$ 
uniquely extends to a ring homomorphism 
$\zeta \colon \KV{\FF_q} \to \Lambda(\QQ) = 1 + t\QQ[\![t]\!]$ 
\cite[Theorem~2.1]{ramachandran2015zeta}. 

\begin{proposition}\label{prop:zetainvariant} 
The homomorphism 
$\zeta \colon \KV{\FF_q} \to \Lambda(\QQ)$ 
commutes with $\lambda$-operations. 
\end{proposition}

\begin{proof} %{*
It is well-known that 
    the motivic zeta function $Z_X(t)$ of $X$ equals 
$\exp\left(\sum_{r}\psi_r(X)\frac{t^r}{r}\right)$  
and specializes to 
$\zeta_X=
\exp\left(\sum_{r}|X(\FF_{q^r})|\frac{t^r}{r}\right)$ 
under the point-counting ring homomorphism 
$\mm_q \colon \KV{\FF_q} \to \ZZ \colon X \mapsto |X(\FF_q)|$. 
Let 
$\mm_{q,q^d} \colon \KV{\FF_q} \to \ZZ \colon 
X \mapsto |X(\FF_{q^d})|$. 
More generally, since the formation of symmetric products 
    commutes with flat base change, 
    $\mm_{q^d}Z_{X_{\FF_{q^d}}}(t) 
    = \sum_{n}|\SymProd{(X_{\FF_{q^d}})}{n}(\FF_{q^d})|t^n
    = \sum_{n}|(\SymProd{X}{n})_{\FF_{q^d}}(\FF_{q^d})|t^n
    = \sum_{n}|(\SymProd{X}{n})(\FF_{q^d})|t^n$, 
    so 
    $\zeta_{X_{\mathbb F_{q^d}}}$ 
    %= \exp\left(\sum_{r}
    %|X_{\mathbb F_{q^d}}(\FF_{q^{rd}})|\frac{t^r}{r}\right)$ 
    equals the image of $Z_X(t)$ 
    under $\mm_{q,q^d}$.  
This shows that 
$\mm_{q,q^d}(\psi_r(X)) 
= |\psi_r(X)(\mathbb F_{q^d})|
= |X(\mathbb F_{q^{rd}})|$ 
and so 
\begin{equation}\label{eqn:adamstozeta}
\zeta_{\psi_r(X)}
= \exp\left(\sum_{d=1}^\infty
|\psi_r(X)(\mathbb F_{q^d})|\frac{t^d}{d}\right) 
= \exp\left(\sum_{d=1}^\infty
    |X(\mathbb F_{q^{rd}})|
    \frac{t^d}{d}\right) 
= \psi_r(\zeta_X) .  
\end{equation}
Since Adams operations are additive 
this shows that $\zeta$ commutes with Adams operations 
on all of $\KV{\FF_q}$. 
As $\Lambda(\QQ)$ is free of additive torsion, 
commuting with Adams operations implies 
commuting with $\lambda$-operations 
    (Lemma~\ref{lemma:adamstolambda}). 
\end{proof} %*}

\begin{corollary}\label{cor:pointcountingmainformula}
\begin{equation}\label{eqn:pointcountingmainformula} %{*
|\Uc_d(\mathbb F_q)|=
\sum_{d=km}
\sum_{\lambda \vdash m}
    (-1)^{\ell-1}\frac{\mu(k)}{k\ell}
    \binom{\ell}{n_1,\ldots,n_m}
\prod_{j=1}^m
    |\Sc_j(\mathbb F_{q^{k}})|^{n_j}.
\end{equation} %*}
\end{corollary}

\begin{proof} %{*
Apply \eqref{eqn:irrdnmotivicformulatoptype2} 
    to the motivic measure $\zeta$ to obtain 
\begin{equation}%\label{eqn:pointcountingmainformula} %{*
    \zeta_{\Uc_d}=
\sum_{d=km}
\sum_{\lambda \vdash m}
    (-1)^{\ell-1}\frac{\mu(k)}{k\ell}
    \binom{\ell}{n_1,\ldots,n_m}
\prod_{j=1}^m
    (\psi_{k}
    \zeta_{\Sc_j})^{n_j}.
\end{equation} %*}
Identify $\Lambda(\QQ)$ with 
    the ring of big Witt vectors valued in $\QQ$. 
The corollary follows from applying the 
    projection to the first ghost coordinate: 
\begin{equation} %{*
    \frac{d}{dt}\circ\log\bigg\rvert_{t=0}\colon 
    \Lambda(\QQ) \to \QQ\colon 
    f(t)\mapsto \frac{f'(0)}{f(0)}.
\end{equation} %*}
We compute Adams operations 
    using \eqref{eqn:adamstozeta} which shows 
    this projection sends $\psi_r(\zeta_X)$ to $X(\mathbb F_{q^r})$. 
\end{proof} %*}

%*}

    %*}

\subsection{Inverse P\'olya enumeration}\label{sec:invpolyaenumeration} %{*

Consider a collection of objects whose 
    weights are positive integers 
    and assume that 
    the number $u_d$ of objects of any given weight $d$ is finite. 
Consider the problem of determining 
    the integers $u_1,u_2,\ldots$ given 
    the integers $x_1,x_2,\ldots$ where $x_d$ 
    is the total number of unordered 
    combinations of objects of weight $d$. 
This setting corresponds to 
    the $\lambda$-ring $R = \ZZ$ with trivial Adams operations, 
and the coefficients of the plethystic logarithm for $\ZZ$ are 
    the following universal polynomials: 
\begin{equation}\label{eqn:fdck2} %{*
    u_d = 
\sum_{d=km}
\sum_{\lambda \vdash m}
    (-1)^{\ell-1}\frac{\mu(k)}{k\ell}
    \binom{\ell}{n_1,\ldots,n_m}
    x^\lambda \in \QQ[x_1,\ldots,x_d].
\end{equation} %*}
For example, 
    $u_2=x_2 - \tfrac12(x_1^2+x_1)$ 
    and 
    $u_3=x_3 - x_2x_1 + 
    \tfrac{1}{3}(x_1^3-x_1).$ 
The problem of determining $x_1,x_2,\ldots$ 
    given $u_1,u_2,\ldots$ is a form of P\'olya enumeration, 
    whence the name.

\section{Motivic measures of general strata}\label{sec:genconfigspaces} %{*

In this section we investigate the motivic measure 
    of a general stratum $\Uc_\tau$ in $\Sc_d$ 
    of points with exact splitting type $\tau$. 
We work in the same abstract setting as in 
Proposition~\ref{prop:zetafngradedspace}. 
That is to say, $\Uc_1,\Uc_2,\ldots$ denotes 
    an arbitrary collection of varieties and $\Sc_k$ is 
    the unordered configuration space of degree $k$ on 
    $\sqcup_d \Uc_d$ where points of $\Uc_d$ have weight $d$. 

\subsubsection{The type stratification} %{* 

For any integers $d,e$ we have the morphism 
    $\Sc_d \times \Sc_e \to \Sc_{d+e}$ 
induced by multiplication. 
Any arrangement $A$ of $\lambda$ into $\tau$ determines a morphism 
\begin{align} %{*
    \pi_A \colon 
    \prod_{b^m \in \lambda} \Sc_b &\to 
    \prod_{c^n \in \tau} \Sc_c  \\
    (c_1,\ldots,c_r)&\mapsto 
    (c_1^{A_{11}}\cdots c_r^{A_{r1}},\ldots,c_1^{A_{1s}}\cdots c_r^{A_{rs}}).
\end{align} %*}
We write $\pi$ for the distinguished morphism 
$\prod_{b^m \in \lambda} \Sc_b \to \Sc_d$ 
corresponding to the unique arrangement of $\lambda$ into the maximal type $d$. 
    The \defn{$\tau$-stratum $\Sc_\tau \subset \Sc_d$} is 
    the Zariski closure of 
    $\pi(\prod_{c^n \in \tau} \Sc_c)$. 
The \defn{subspace $\Uc_\tau \subset \Sc_d$ of elements of (exact) type $\tau$} 
    is the locally closed subscheme 
    $\Sc_\tau \backslash \cup_{\lambda < \tau} \Sc_\lambda$. 

%*} 

\subsection{Generalized configuration spaces} %{*

Let $X$ be a variety over a field $k$. 
Let $\vec{n} = (n_1,\dots,n_d) \in \ZZgeq^d$ 
    and let $\Delta \subset X^{n_1+\dots+n_d}$ 
    be the ``big diagonal'' consisting of points 
    where at least two of the coordinates are equal. 
Let 
    \[\Conf_{\vec{n}}(X) = 
    \big(X^{n_1+\dots+n_d}-\Delta\big)\,\,/\,\,
    \Sigma_{n_1}\times\dots\times\Sigma_{n_d}.\]
For a type $\tau \vDash d$ 
with associated partitions $\tau_p$ corresponding to degree $p$, 
let $\vec{n}_p = (\tau[p^1],\ldots,\tau[p^d])  \in \ZZgeq^d$. 
We leave the proof of the following easy lemma to the reader. 

\begin{lemma}\label{lemma:strataasgenconfigspaces}
Combining configurations gives a canonical isomorphism 
    between any open stratum 
    $\Uc_\tau$ with a product of generalized configuration spaces: 
    \begin{equation}
        \Uc_\tau = \prod_{p= 1}^d\Conf_{\vec{n}_p}(\Uc_p). 
    \end{equation}
\end{lemma}

%\begin{proof} %{*
%For a partition $\tau_p$ as above, 
%    we let $p^{\tau_p}$ denote the type with 
%    all degrees $p$ and multiplicities given by $\tau_p$. 
%The multiplication map (induced by addition of cycles) 
%    \begin{equation}
%    \Uc_{1^{\tau_1}}\times\Uc_{2^{\tau_2}}\times\dots \to \Uc_{\tau}
%\end{equation}
%is an isomorphism. 
%Similarly, the map
%\begin{align}
%    \Conf_{\vec{n}_p}(\Uc_p) &\to \Uc_{p^{\tau_p}}\\
%    (D_1,D_2,\ldots) &\mapsto \sum_{k}k D_k
%\end{align}
%is also an isomorphism which completes the proof.
%\end{proof} %*}

Generalized configuration spaces 
satisfy the following recurrence relation in $\KV{k}$. 
Let $m \geq 1$ be an integer, $\vec{w} = (n_1,\dots,n_d) \in \ZZgeq^d$ 
and $\vec{v} = (n_1,\dots,n_{d},m)$. 
We denote $|\vec{w}| = \sum_{i}n_i$. 
For vectors $\vec{u}_1,\vec{u}_2$ 
we write $\vec{u}_1 \leq \vec{u}_2$ 
if each coordinate satisfies the inequality. 

\begin{proposition}\label{prop:generalizedconfspacerecurrence}
For any $\vec{v},\vec{w}$ as above, 
    we have the following recurrence in $\KV{k}$, 
\begin{equation}\label{eqn:generalizedconfspacerecurrence}
    [\Conf_{\vec{v}}(X)] 
    = [\Conf_{\vec{w}}(X)\times \Conf_{m}(X)] 
    - \sum_{\substack{\vec{u} \leq \vec{w} \in \ZZgeq^d\\ 1 \leq |\vec{u}|\leq m }}
    [\Conf_{(\vec{w}-\vec{u},m-|\vec{u}|,\vec{u})}(X)]
\end{equation}
where $(\vec{w}-\vec{u},m-|\vec{u}|,\vec{u})$ denotes 
    the vector obtained by concatenation.
\end{proposition}

\begin{proof} %{*
The open immersion 
\begin{align}
    \Conf_{\vec{v}}(X) &\to \Conf_{\vec{w}}(X)\times\Conf_m(X)\\
    (D_1,\dots,D_d,E) &\mapsto ((D_1,\dots,D_d),E)  
\end{align}
identifies $\Conf_{\vec{v}}(X)$ 
    with an open subspace of $\Conf_{\vec{w}}(X)\times\Conf_m(X)$. 
The complement consists of points 
    $(D_1,\dots,D_d,E)$ where the support of $E$ 
    is not disjoint from the supports of the $D_i$. 
Such points are stratified by vectors 
    $\vec{u} = (u_1,\dots,u_d)$ where $u_i$ equals 
    the degree of the intersection of $E$ with $D_i$. 
The morphism 
\begin{align}
    \Conf_{(\vec{w}-\vec{u},m-|\vec{u}|,\vec{u})}(X) &\to \Conf_{\vec{w}}(X)\times \Conf_m(X)\\
    (D_1,\dots,D_d,E,F_1,\dots,F_d) &\mapsto (D_1+F_1,\dots,D_d+F_d,E+\sum_i F_i) 
\end{align}
identifies the open stratum corresponding to 
    $\vec{u}$ with 
    $\Conf_{(\vec{w}-\vec{u},m-|\vec{u}|,\vec{u})}(X)$. 
\end{proof} %*}

\begin{remark}\label{rmk:compatibilityofrecurrencewithbinomialcoefficients}
Define the generating series 
\[C_d(t_1,\dots,t_d)(X) = 
    \sum_{\vec{w} = (n_1,\dots,n_d)\geq \vec{0}}
    [\Conf_{\vec{w}}(X)]t_1^{n_1}\dots t_d^{n_d} \in \KV{k}[\![t_1,\ldots,t_d]\!].\]
Then Proposition~\ref{prop:generalizedconfspacerecurrence} 
    is equivalent to the identity  
    $$C_{2d+1}(t_1,\dots,t_d,s,t_1s,\dots,t_ds)(X) 
    = C_d(t_1,\dots,t_d)(X)C_1(s)(X).$$ 
\end{remark}

\begin{remark}
Configuration spaces with restricted multiplicities 
    have been studied in \cite{farb2019coincidences}, 
    \cite{biludashowe}. 
The polysymmetric analogue of these spaces is 
    $\Epf{d}{n} \in \PS_\ZZ$, 
    the sum of all products $f_1 \cdots f_r$ 
    where $f_1,\ldots,f_r$ are monomials in 
    $\ZZ[\![x_{\ast\ast}]\!]$ 
    with respective degrees $d_1,\ldots,d_r$ whose 
    greatest common divisor is $n$-powerfree. 
Let $\vec{t} = (t_1,\ldots,t_r)$ be a vector of indeterminates, 
and for $\vec{d} = (d_1,\ldots,d_r)$ 
write $\vec{t}^{\,\vec{d}} = t_1^{d_1}\cdots t_r^{d_r}$ 
and $|\vec{d}| = d_1 + \cdots + d_r$. 
One can prove the following identity in $\PS_\ZZ[\![\vec{t}]\!]$, 
\begin{equation} %{*
    \left( \sum_{\vec{d} \geq 0} \Epf{d}{n}\, \vec{t}^{\,\vec{d}} \right)
    \left( \sum_{{d}= 0}^\infty (\psi_n H_{{d}}) 
    (t_1\cdots t_r)^{n{d}} \right)
     = \prod_{i=1}^r \sum_{d = 0}^\infty H_d t_i^d.
\end{equation} %*}
This is a polysymmetric analogue of 
    \cite[Example 4.1.1-(3)]{biludashowe}. 
\end{remark}

%*} 

\subsection{The ring of characteristic cycles}\label{sec:binomialquotient} %{* 

The \defn{ring $\RC{k}$ of characteristic cycles over $k$} 
    is defined to be $\KV{k}_B$, 
    the quotient binomial ring of $\KV{k}$ 
    (cf. Remark~\ref{rmk:quotientBinomialRing}). 
    Let $\Rcc{\Uc}$ denote the class of a variety in $\RC{k}$. 
The ring $\RC{k}$ avoids certain pathologies exhibited by $\KV{k}$. 
One such pathology is that in general 
    $M_\tau \circ (\Uc_1 t + \Uc_2 t^2+\cdots)$ 
    cannot be identified with the stratum $\Uc_\tau$ 
    except when $\tau = d$. 
Even if a motivic measure $\mm$ is valued in 
    a special $\lambda$-ring 
    we may not have 
    $\mm(M_\tau \circ (\Uc_1 t + \Uc_2 t^2+\cdots)) 
    = \mm(\Uc_\tau)$. 
Nonetheless we will show that 
    $$\Rcc{M_\tau \circ (\Uc_1 t + \Uc_2 t^2+\cdots)} 
    = \Rcc{\Uc_\tau}.$$ 
Another such pathology is that 
    in view of the formula $[\SymProd{X}{n}] = \sigma_n[X]$ 
    and the analogy with the standard $\lambda$-operations 
    on vector bundles, 
    one might expect that $[\Conf_n(X)] = \lambda_n[X]$ 
    however this generally fails. 
Nonetheless we have 
the following result of Macdonald~\cite{MR143204} 
\begin{equation}\label{eqn:chiofconf} %{*
    \chi(\Conf_n(X)) = \binom{\chi(X)}{n} = \lambda_n(\chi(X)) 
\end{equation} %*}
so even if $[\Conf_n(X)]-\lambda_n([X])$ fails to vanish 
    it is at least in the kernel of $\chi$. 
In fact, 
Macdonald's identity holds for any binomial measure; equivalently, 
the expected formula 
$$\Rcc{\Conf_n(X)} = \binom{\Rcc{X}}{n} = \lambda_n \Rcc{X}$$ 
holds in $\RC{k}$. 
The next result generalizes this to a wider class of varieties, 
    namely generalized configuration spaces. 

\begin{theorem}\label{thm:charclassgenconfigsp}
Let $\vec{n} = (n_1,\ldots,n_d) \in \ZZgeq^d$ and let $X$ be a variety over $k$. 
Then 
\begin{equation}\label{eqn:charcycleofgenconf} 
    \Rcc{\Conf_{\vec{n}}(X)} = \binom{\Rcc{X}}{n_1,\ldots,n_d}
    = \frac{\Rcc{X}(\Rcc{X}-1) \cdots (\Rcc{X}-N+1)}{n_1!\cdots n_d!}
\end{equation}
where $N = \sum_i n_i$. 
\end{theorem}

\begin{proof} %{*
First suppose $d = 1$. 
Then %$\Conf_{\vec{n}}(X) = \Zpf{n}{2}(X)$ 
by \cite[Example 4.1.1-(3)]{biludashowe} we have that 
\begin{equation}
    \sum_{m= 0}^\infty [\Conf_m(X)]t^m = \frac{Z_X(t)}{Z_X(t^2)}.
\end{equation}
    Let $\Rcc{Z_{X}(t)}$ denote the image of $Z_X(t)$ in 
    $\RC{k}[\![t]\!]$. 
Then 
\begin{equation}
    \Rcc{Z_{X}(t)} = \exp\left(\sum_{r= 1}^\infty\psi_r(\Rcc{X})\frac{t^r}{r}\right) = (1-t)^{-\Rcc{X}}
\end{equation}
so 
\begin{equation}
    \sum_{m= 0}^\infty
    \Rcc{\Conf_m(X)}t^m 
    = \left(\frac{1-t^2}{1-t}\right)^{\Rcc{X}} 
    = \sum_{m=0}^\infty\binom{\Rcc{X}}{m}t^m.
\end{equation}

For the general case 
we proceed by induction on $|\vec{n}|$. 
Suppose $d \geq 2$ and set $\vec{n}' = (n_1,\ldots,n_{d-1})$. 
Applying $\Rcc{-}$ to \eqref{eqn:generalizedconfspacerecurrence} 
and using the induction hypothesis\footnote
    {Note that $|\vec{n}'-\vec{u}| + (n_d-|\vec{u}|) + |\vec{u}| < |\vec{n}|$.} 
obtains 
\begin{multline}
    \Rcc{\Conf_{\vec{n}}(X)} 
    = \binom{\Rcc{X}}{n_1,\dots,n_{d-1}} \binom{\Rcc{X}}{n_d} \\ 
    -\sum_{\substack{\vec{u} \leq \vec{n}' \in \ZZgeq^{d-1}\\ 1 \leq |\vec{u}|\leq n_d }}
        \binom{\Rcc{X}}{n_1-u_1,\dots,n_{d-1}-u_{d-1},n_d-|\vec{u}|,u_1,\dots,u_{d-1}} . 
\end{multline}
%By Remark~\ref{rmk:compatibilityofrecurrencewithbinomialcoefficients}, we see that

For any binomial ring $R$ and element $x \in R$, the formal series 
\[B_{d}(x;t_1,\dots,t_d) \coloneqq (1+t_1+\cdots+t_d)^{x} 
    = \sum_{(n_1,\dots,n_d)\geq \vec{0}}
    \binom{x}{n_1,\dots,n_d}t_1^{n_1}\cdots t_d^{n_d} 
    \in R[\![t_1,\ldots,t_d]\!]\]
satisfies the identity from 
    Remark~\ref{rmk:compatibilityofrecurrencewithbinomialcoefficients}, 
    i.e.  
\begin{align}
    B_{2d+1}(x;t_1,\ldots,t_d,s,t_1s,\ldots,t_ds) 
    &= (1+t_1+\cdots+t_d + s + t_1s + \cdots + t_ds)^{x} \\
    &= (1+t_1+\cdots+t_d)^x(1+s)^x = B_d(x;t_1,\ldots,t_d)B_1(x;s).
\end{align}
Therefore the following identity holds (as polynomials in $x$), 
\begin{multline}
    \binom{x}{n_1,\dots,n_d} = 
\binom{x}{n_1,\dots,n_{d-1}} \binom{x}{n_d} \\
    -\sum_{\substack{\vec{u} \leq \vec{n}' \in \ZZgeq^{d-1}\\ 1 \leq |\vec{u}|\leq n_d }}
        \binom{x}{n_1-u_1,\dots,n_{d-1}-u_{d-1},n_d-|\vec{u}|,u_1,\dots,u_{d-1}},  
\end{multline}
    Taking $x = \Rcc{X}$ proves the theorem. 
\end{proof} %*}

\begin{corollary}\label{cor:arbitrarystratabinomial}
For any type $\tau\vDash d$ we have 
\begin{equation}
    \Rcc{\Uc_\tau} = 
    \prod_{p = 1}^d 
    \binom{\Rcc{\Uc_p}}{\tau[p^1],\ldots,\tau[p^d]}. 
\end{equation}
\end{corollary}

\begin{remark}[rationality of the motivic zeta function]\label{rmk:rationality of the motivic zeta function} 
The ring of characteristic cycles over $k$ is 
    a special $\lambda$-ring, and therefore receives 
    a homomorphism from the specialization 
    $\KV{k}_{\mathrm{sp}}$ 
    of the Grothendieck ring of varieties 
    (cf.~\cite[Lemma~4.12]{MR2098401}). 
Larsen and Lunts have asked whether 
    the motivic zeta function of a variety in 
    $\KV{k}_{\mathrm{sp}}[\![t]\!]$ is rational 
    \cite[Question~8.8]{MR2098401}. 
Is the image of the motivic zeta function of a variety 
    in $\RC{k}[\![t]\!]$ rational? 
\end{remark}

\begin{theorem}\label{thm:trivialadamsoperationsstrata}
For any type $\tau\vDash d$ we have
\begin{equation}
    \Rcc{\Uc_\tau} = M_\tau\circ\Rcc{\Uc}
    \quad\text{and}\quad
    %\prod_{p= 1}^dm_{\tau_p}\circ \Rcc{\Uc_p}.
    \Rcc{\Sc_\tau} = H_\tau\circ\Rcc{\Uc}. %= \prod_{p= 1}^dh_{\tau_p}\circ \Rcc{\Sc_p}.
\end{equation}
In particular, if $\mm$ is a motivic measure valued 
    in a binomial ring then for any type $\lambda$, 
\begin{equation} %{*
\mm(\Uc_\lambda) = 
\sum_{\tau \leq \lambda}
\arr^{-1}_{\tau \lambda}
\prod_{b^m \in \tau}
    \mm(\Sc_b) 
    \quad\text{and}\quad
\mm(\Sc_\lambda) = 
\sum_{\tau \leq \lambda}
\arr_{\tau \lambda}
\prod_{b^m \in \tau}
    \mm(\Uc_b) .
\end{equation} %*}
\end{theorem}

%This proves 
%    Theorem~\ref{thm:introtrivialadamsoperationsstrata} 
%    in view of the universal property of $\RC{k}$, 
%    \eqref{eqn:defnc}, and the definition of higher plethysm. 

For a type $\tau$ 
let $\tau_p$ denote the multiset of $k$ 
satisfying $p^k \in \tau$, 
regarded as a partition. 
We write each partition $\tau_p$ as 
$(1^{{n}_1} \cdots d^{{n}_d})$ 
for non-negative integers 
$\vec{n}_p = (n_1,\ldots,n_d) \in \ZZgeq^d$ 
(so $\vec{n}_p = (\tau[p^1],\ldots,\tau[p^d])$). 

\begin{proof} %{*
Recall that 
    $M_\tau = m_{\tau_1}(x_{1\ast})\cdots m_{\tau_d}(x_{d\ast})$ 
    (Remark~\ref{rmk:Mtautosymmetricmlambda}). 
Let $\Rcc{\Uc}$ denote the element 
    $\Rcc{\Uc_1}t+\Rcc{\Uc_2}t^2+\cdots$ of $t\RC{k}[\![t]\!]$. 
By definition of higher plethysm, 
\begin{equation} %{*
    M_\tau \circ \Rcc{\Uc} 
    = \prod_{p=1}^d 
    m_{\tau_p} \circ \Rcc{\Uc_p}. 
\end{equation} %*}
On the other hand by Lemma~\ref{lemma:strataasgenconfigspaces}, 
    \[\Rcc{\Uc_\tau} 
    = \prod_{p= 1}^d\Rcc{\Conf_{\vec{n}_p}(\Uc_p)}. \]
We conclude by Theorem~\ref{thm:charclassgenconfigsp} 
    and Proposition~\ref{prop:plethysmsasbinomialcoefficients} 
    that $\Rcc{\Uc_\tau} = M_\tau\circ\Rcc{\Uc}$. 
Now observe that %$\RC{k}$ is a special $\lambda$-ring, 
    $\Rcc{\Sc_\tau} = 
    \sum_{\lambda \leq \tau} \Rcc{\Uc_\lambda} 
    = \sum_{\lambda \leq \tau} (M_\lambda \circ \Rcc{\Uc}) 
    = (\sum_{\lambda \leq \tau} M_\lambda) \circ \Rcc{\Uc}
    = H_\tau \circ \Rcc{\Uc}$. 
\end{proof} %*}

\begin{lemma}\label{lem:classoflefschetzmotivesinbinomialquotient}
$\Rcc{\AA^1_k}$ is an idempotent in $\RC{k}$ 
    and it is not equal to $0$ or $1$. 
\end{lemma}

\begin{proof} %{*
    $\mathrm{Sym}^n(\mathbb A^1) = \mathbb A^n$ 
    by the fundamental theorem of symmetric functions. 
Thus 
    \[\sum_{r\geq 1}\psi_{r}([\AA^1_k])\frac{t^r}{r} = \log\sum_{n\geq 0}\sigma_n([\AA^1_k])t^n = \log\left(\frac{1}{1-[\AA^1_k ]t}\right)\]
    implies that $\psi_r([\AA^1_k]) = [\AA^1_k]^r \in \KV{k}$. 
    In the binomial quotient, 
$\Rcc{\AA^1_k}^2 = \psi_2(\Rcc{\AA^1_k}) = \Rcc{\AA^1_k}$. 
Moreover, one easily finds ring homomorphisms mapping $\Rcc{\AA^1_k}$ 
    to either $0$ or $1$. 
%Indeed
%$\chi_c^{\mathrm{hol}}(\LL) = \chi_c^{\mathrm{hol}}(\mathbb P^1_k) 
%    - \chi_c^{\mathrm{hol}}(\operatorname{Spec} k) = 1-1 = 0$ 
%while
%$\chi_c(\LL) = \chi_c(\mathbb P^1_k) - \chi_c(\operatorname{Spec} k) 
%    = 2-1=1$. 
\end{proof} %*}

%We can determine the class of $\Uc_d$ in $\RC{k}$ completely 
%(and as a consequence of Theorem \ref{thm:trivialadamsoperationsstrata}, 
%the classes corresponding to $\Uc_\tau,\Sc_\tau$ for any type $\tau$). 

The next theorem contains the rest of 
    Theorem~\ref{thm:GeomIrred}. 

\begin{theorem}\label{thm:charcycleofirrd}
The binomial class $\Rcc{\Irr_\tau}$ is zero 
    unless $d^m \in \tau \implies d= 1$ in which case 
\begin{equation} %{*
    \Rcc{\Irr_\tau} = 
        \binom{\Rcc{\AA^1_k}(n-1) + 1}{\tau[1^1],\ldots,\tau[1^d]} .
\end{equation} %*}
\end{theorem}

\begin{proof} %{*
By Corollary \ref{cor:arbitrarystratabinomial} 
    it suffices to prove the case $\tau = d$. 
As $\Rcc{\AA^1_k}$ is an idempotent 
    (Lemma~\ref{lem:classoflefschetzmotivesinbinomialquotient}) 
    we see that 
    $\Rcc{\mathbb P^N_k} = \sum_{i=0}^N\Rcc{\AA^1_k}^N = (N+1)\Rcc{\AA^1_k} + (1-\Rcc{\AA^1_k})$. 
Therefore, by Proposition \ref{prop:zetafngradedspace} we have
\begin{align}
    \Rcc{Z_{\Irr_d}(t)} = \prod_{d=1}^\infty(1-t^d)^{-\Rcc{\Irr_d}} 
    &= \sum_{d= 0}^\infty\Rcc{\Sc_d}t^d\\
    &= \Rcc{\AA^1_k}\sum_{d= 0}^\infty\binom{n+d}{d}t^d + (1-\Rcc{\AA^1_k})\sum_{d= 0}^\infty t^d\\
    &= \frac{\Rcc{\AA^1_k}}{(1-t)^n} + \frac{1-\Rcc{\AA^1_k}}{1-t}.
\end{align}
Since $\Rcc{\AA^1_k}$ is an idempotent, we can prove an identity in $\RC{k}$ by verifying it after applying the two projections $q_1,q_0$ specializing $\Rcc{\AA^1_k} = 1$ or $\Rcc{\AA^1_k} = 0$. In the first case, we have
\[\prod_{d= 1}^\infty(1-t^d)^{-q_1(\Rcc{\Irr_d})} = (1-t)^{-n}\]
while in the second case, we have
\[\prod_{d= 1}^\infty(1-t^d)^{-q_0(\Rcc{\Irr_d})} = (1-t)^{-1}.\]
Comparing the exponents on both sides 
    and appealing to the uniqueness of such a factorization 
    obtains the claimed identity. 
\end{proof} %*}

\newpage
\appendix

\newgeometry{left=1.4cm,right=1.4cm}

\section{Arrangement numbers \texorpdfstring{$\arr_{\lambda\tau}$ in degrees $\leq 5$}{in low degrees}} %{* 

\renewcommand{\vertrowspace}{.3em}
\renewcommand{\arraystretch}{1.15}

\vspace{-.2cm}
%\fbox{
\begin{minipage}[t][8cm][t]{0.3\textwidth}	
    \vspace{.68cm}
% Degree Two
\[ %{*
\hspace{10000pt minus 1fil}
\begin{array}{rccc}
& \rotatebox{90}{$ (1^{2}) $} & \rotatebox{90}{$ (1 1) $} & \rotatebox{90}{$ (2) $} \\[\vertrowspace] 
(1^{2}) & $1$ & $1$ & $1$ \\
(1 1) & $0$ & $2$ & $1$ \\
(2) & $0$ & $0$ & $1$ \\
\end{array}
\] %*}
    \vspace{.28cm}
% Degree Three
\[ %{*
\hspace{10000pt minus 1fil}
\begin{array}{rccccc}
 & \rotatebox{90}{$ (1^{3}) $} & \rotatebox{90}{$ (1^{2} 1) $} & \rotatebox{90}{$ (1 1 1) $} & \rotatebox{90}{$ (2 1) $} & \rotatebox{90}{$ (3) $} \\[\vertrowspace]
(1^{3}) & $1$ & $1$ & $1$ & $1$ & $1$ \\
(1^{2} 1) & $0$ & $1$ & $3$ & $2$ & $1$ \\
(1 1 1) & $0$ & $0$ & $6$ & $3$ & $1$ \\
(2 1) & $0$ & $0$ & $0$ & $1$ & $1$ \\
(3) & $0$ & $0$ & $0$ & $0$ & $1$ \\
\end{array}
\] %*}
\end{minipage}%
\begin{minipage}[t][8cm][t]{0.6\textwidth}	
% Degree Four
    \vspace{.65em}
\[ %{*
\begin{array}{rccccccccccc}
 & \rotatebox{90}{$ (1^{4}) $} & \rotatebox{90}{$ (1^{3} 1) $} & \rotatebox{90}{$ (1^{2} 1^{2}) $} & \rotatebox{90}{$ (1^{2} 1 1) $} & \rotatebox{90}{$ (1 1 1 1) $} & \rotatebox{90}{$ (2 1^{2}) $} & \rotatebox{90}{$ (2 1 1) $} & \rotatebox{90}{$ (3 1) $} & \rotatebox{90}{$ (2^{2}) $} & \rotatebox{90}{$ (2 2) $} & \rotatebox{90}{$ (4) $} \\[\vertrowspace]  
(1^{4}) & $1$ & $1$ & $1$ & $1$ & $1$ & $1$ & $1$ & $1$ & $1$ & $1$ & $1$ \\
(1^{3} 1) & $0$ & $1$ & $0$ & $2$ & $4$ & $1$ & $3$ & $2$ & $0$ & $2$ & $1$ \\
(1^{2} 1^{2}) & $0$ & $0$ & $2$ & $2$ & $6$ & $2$ & $4$ & $2$ & $1$ & $3$ & $1$ \\
(1^{2} 1 1) & $0$ & $0$ & $0$ & $2$ & $12$ & $1$ & $7$ & $3$ & $0$ & $4$ & $1$ \\
(1 1 1 1) & $0$ & $0$ & $0$ & $0$ & $24$ & $0$ & $12$ & $4$ & $0$ & $6$ & $1$ \\
(2 1^{2}) & $0$ & $0$ & $0$ & $0$ & $0$ & $1$ & $1$ & $1$ & $0$ & $2$ & $1$ \\
(2 1 1) & $0$ & $0$ & $0$ & $0$ & $0$ & $0$ & $2$ & $2$ & $0$ & $2$ & $1$ \\
(3 1) & $0$ & $0$ & $0$ & $0$ & $0$ & $0$ & $0$ & $1$ & $0$ & $0$ & $1$ \\
(2^{2}) & $0$ & $0$ & $0$ & $0$ & $0$ & $0$ & $0$ & $0$ & $1$ & $1$ & $1$ \\
(2 2) & $0$ & $0$ & $0$ & $0$ & $0$ & $0$ & $0$ & $0$ & $0$ & $2$ & $1$ \\
(4) & $0$ & $0$ & $0$ & $0$ & $0$ & $0$ & $0$ & $0$ & $0$ & $0$ & $1$ \\
\end{array}
    \] %*}
\end{minipage}
%}

\vspace{-\parskip}
% Degree Five
%\fbox{
    \begin{minipage}[t][10cm][t]{0.9\textwidth}
        \centering
    \vspace{1.2cm}
\[ %{*
\begin{array}{rccccccccccccccccc}
 & \rotatebox{90}{$ (1^{5}) $} & \rotatebox{90}{$ (1^{3} 1^{2}) $} & \rotatebox{90}{$ (1^{4} 1) $} & \rotatebox{90}{$ (1^{3} 1 1) $} & \rotatebox{90}{$ (2 1^{3}) $} & \rotatebox{90}{$ (1^{2} 1^{2} 1) $} & \rotatebox{90}{$ (1^{2} 1 1 1) $} & \rotatebox{90}{$ (1 1 1 1 1) $} & \rotatebox{90}{$ (2 1^{2} 1) $} & \rotatebox{90}{$ (2 1 1 1) $} & \rotatebox{90}{$ (3 1^{2}) $} & \rotatebox{90}{$ (3 1 1) $} & \rotatebox{90}{$ (2^{2} 1) $} & \rotatebox{90}{$ (2 2 1) $} & \rotatebox{90}{$ (3 2) $} & \rotatebox{90}{$ (4 1) $} & \rotatebox{90}{$ (5) $} \\[\vertrowspace] 
(1^{5}) & $1$ & $1$ & $1$ & $1$ & $1$ & $1$ & $1$ & $1$ & $1$ & $1$ & $1$ & $1$ & $1$ & $1$ & $1$ & $1$ & $1$ \\
(1^{3} 1^{2}) & $0$ & $1$ & $0$ & $1$ & $1$ & $2$ & $4$ & $10$ & $3$ & $7$ & $2$ & $4$ & $1$ & $5$ & $3$ & $2$ & $1$ \\
(1^{4} 1) & $0$ & $0$ & $1$ & $2$ & $1$ & $1$ & $3$ & $5$ & $2$ & $4$ & $1$ & $3$ & $1$ & $3$ & $2$ & $2$ & $1$ \\
(1^{3} 1 1) & $0$ & $0$ & $0$ & $2$ & $1$ & $0$ & $6$ & $20$ & $3$ & $13$ & $1$ & $7$ & $0$ & $8$ & $4$ & $3$ & $1$ \\
(2 1^{3}) & $0$ & $0$ & $0$ & $0$ & $1$ & $0$ & $0$ & $0$ & $1$ & $1$ & $1$ & $1$ & $0$ & $2$ & $2$ & $1$ & $1$ \\
(1^{2} 1^{2} 1) & $0$ & $0$ & $0$ & $0$ & $0$ & $2$ & $6$ & $30$ & $4$ & $18$ & $2$ & $8$ & $1$ & $11$ & $5$ & $3$ & $1$ \\
(1^{2} 1 1 1) & $0$ & $0$ & $0$ & $0$ & $0$ & $0$ & $6$ & $60$ & $3$ & $33$ & $1$ & $13$ & $0$ & $18$ & $7$ & $4$ & $1$ \\
(1 1 1 1 1) & $0$ & $0$ & $0$ & $0$ & $0$ & $0$ & $0$ & $120$ & $0$ & $60$ & $0$ & $20$ & $0$ & $30$ & $10$ & $5$ & $1$ \\
(2 1^{2} 1) & $0$ & $0$ & $0$ & $0$ & $0$ & $0$ & $0$ & $0$ & $1$ & $3$ & $1$ & $3$ & $0$ & $4$ & $3$ & $2$ & $1$ \\
(2 1 1 1) & $0$ & $0$ & $0$ & $0$ & $0$ & $0$ & $0$ & $0$ & $0$ & $6$ & $0$ & $6$ & $0$ & $6$ & $4$ & $3$ & $1$ \\
(3 1^{2}) & $0$ & $0$ & $0$ & $0$ & $0$ & $0$ & $0$ & $0$ & $0$ & $0$ & $1$ & $1$ & $0$ & $0$ & $1$ & $1$ & $1$ \\
(3 1 1) & $0$ & $0$ & $0$ & $0$ & $0$ & $0$ & $0$ & $0$ & $0$ & $0$ & $0$ & $2$ & $0$ & $0$ & $1$ & $2$ & $1$ \\
(2^{2} 1) & $0$ & $0$ & $0$ & $0$ & $0$ & $0$ & $0$ & $0$ & $0$ & $0$ & $0$ & $0$ & $1$ & $1$ & $1$ & $1$ & $1$ \\
(2 2 1) & $0$ & $0$ & $0$ & $0$ & $0$ & $0$ & $0$ & $0$ & $0$ & $0$ & $0$ & $0$ & $0$ & $2$ & $2$ & $1$ & $1$ \\
(3 2) & $0$ & $0$ & $0$ & $0$ & $0$ & $0$ & $0$ & $0$ & $0$ & $0$ & $0$ & $0$ & $0$ & $0$ & $1$ & $0$ & $1$ \\
(4 1) & $0$ & $0$ & $0$ & $0$ & $0$ & $0$ & $0$ & $0$ & $0$ & $0$ & $0$ & $0$ & $0$ & $0$ & $0$ & $1$ & $1$ \\
(5) & $0$ & $0$ & $0$ & $0$ & $0$ & $0$ & $0$ & $0$ & $0$ & $0$ & $0$ & $0$ & $0$ & $0$ & $0$ & $0$ & $1$ \\
\end{array}
\] %*}
    \end{minipage}
%}

%*} 

\section{Inverse arrangement numbers \texorpdfstring{$\arr^{-1}_{\lambda\tau}$ in degrees $\leq 5$}{in low degrees}}\label{app:invarrnumbers} %{* 

\vspace{-.2cm}
%\fbox{
\begin{minipage}[t][8cm][t]{0.3\textwidth}	
    \vspace{.68cm}
% Degree Two
    \[ %{*
\hspace{10000pt minus 1fil}
\begin{array}{rrrr}
& \rotatebox{90}{ $(1^{2})$ } & \rotatebox{90}{$ (1 1)$ } & \rotatebox{90}{$ (2)$ } \\[\vertrowspace] 
(1^{2}) & 1 & -\frac{1}{2} & -\frac{1}{2} \\
(1 1) & 0 & \frac{1}{2} & -\frac{1}{2} \\
(2) & 0 & 0 & 1 \\
\end{array}
    \] %*}
    \vspace{.28cm}
% Degree Three
    \[ %{*
\hspace{10000pt minus 1fil}
\begin{array}{rrrrrr}
 & \rotatebox{90}{$ (1^{3}) $} & \rotatebox{90}{$ (1^{2} 1) $} & \rotatebox{90}{$ (1 1 1) $} & \rotatebox{90}{$ (2 1) $} & \rotatebox{90}{$ (3) $} \\[\vertrowspace] 
(1^{3}) & 1 & -1 & \frac{1}{3} & 0 & -\frac{1}{3} \\
(1^{2} 1) & 0 & 1 & -\frac{1}{2} & -\frac{1}{2} & 0 \\
(1 1 1) & 0 & 0 & \frac{1}{6} & -\frac{1}{2} & \frac{1}{3} \\
(2 1) & 0 & 0 & 0 & 1 & -1 \\
(3) & 0 & 0 & 0 & 0 & 1 \\
\end{array}
    \] %*}
\end{minipage}%
\begin{minipage}[t][8cm][t]{0.63\textwidth}	
% Degree Four
    \vspace{1.23em}
    \[ %{*
\hspace{10000pt minus 1fil}
\begin{array}{rrrrrrrrrrrr}
 & \rotatebox{90}{$ (1^{4}) $} & \rotatebox{90}{$ (1^{3} 1) $} & \rotatebox{90}{$ (1^{2} 1^{2}) $} & \rotatebox{90}{$ (1^{2} 1 1) $} & \rotatebox{90}{$ (1 1 1 1) $} & \rotatebox{90}{$ (2 1^{2}) $} & \rotatebox{90}{$ (2 1 1) $} & \rotatebox{90}{$ (3 1) $} & \rotatebox{90}{$ (2^{2}) $} & \rotatebox{90}{$ (2 2) $} & \rotatebox{90}{$ (4) $} \\[\vertrowspace] 
(1^{4}) & 1 & -1 & -\frac{1}{2} & 1 & -\frac{1}{4} & 0 & 0 & 0 & -\frac{1}{2} & \frac{1}{4} & 0 \\
(1^{3} 1) & 0 & 1 & 0 & -1 & \frac{1}{3} & 0 & 0 & -\frac{1}{3} & 0 & 0 & 0 \\
(1^{2} 1^{2}) & 0 & 0 & \frac{1}{2} & -\frac{1}{2} & \frac{1}{8} & -\frac{1}{2} & \frac{1}{4} & 0 & -\frac{1}{2} & \frac{3}{8} & \frac{1}{4} \\
(1^{2} 1 1) & 0 & 0 & 0 & \frac{1}{2} & -\frac{1}{4} & -\frac{1}{2} & 0 & 0 & 0 & \frac{1}{4} & 0 \\
(1 1 1 1) & 0 & 0 & 0 & 0 & \frac{1}{24} & 0 & -\frac{1}{4} & \frac{1}{3} & 0 & \frac{1}{8} & -\frac{1}{4} \\
(2 1^{2}) & 0 & 0 & 0 & 0 & 0 & 1 & -\frac{1}{2} & 0 & 0 & -\frac{1}{2} & 0 \\
(2 1 1) & 0 & 0 & 0 & 0 & 0 & 0 & \frac{1}{2} & -1 & 0 & -\frac{1}{2} & 1 \\
(3 1) & 0 & 0 & 0 & 0 & 0 & 0 & 0 & 1 & 0 & 0 & -1 \\
(2^{2}) & 0 & 0 & 0 & 0 & 0 & 0 & 0 & 0 & 1 & -\frac{1}{2} & -\frac{1}{2} \\
(2 2) & 0 & 0 & 0 & 0 & 0 & 0 & 0 & 0 & 0 & \frac{1}{2} & -\frac{1}{2} \\
(4) & 0 & 0 & 0 & 0 & 0 & 0 & 0 & 0 & 0 & 0 & 1 \\
\end{array}
    \] %*}
\end{minipage}

\vspace{-\parskip}
% Degree Five
    \begin{minipage}[t][10cm][t]{0.9\textwidth}
        \centering
    \vspace{1.5cm}
\[ %{*
\begin{array}{rrrrrrrrrrrrrrrrrr}
 & \rotatebox{90}{$ (1^{5}) $} & \rotatebox{90}{$ (1^{3} 1^{2}) $} & \rotatebox{90}{$ (1^{4} 1) $} & \rotatebox{90}{$ (1^{3} 1 1) $} & \rotatebox{90}{$ (2 1^{3}) $} & \rotatebox{90}{$ (1^{2} 1^{2} 1) $} & \rotatebox{90}{$ (1^{2} 1 1 1) $} & \rotatebox{90}{$ (1 1 1 1 1) $} & \rotatebox{90}{$ (2 1^{2} 1) $} & \rotatebox{90}{$ (2 1 1 1) $} & \rotatebox{90}{$ (3 1^{2}) $} & \rotatebox{90}{$ (3 1 1) $} & \rotatebox{90}{$ (2^{2} 1) $} & \rotatebox{90}{$ (2 2 1) $} & \rotatebox{90}{$ (3 2) $} & \rotatebox{90}{$ (4 1) $} & \rotatebox{90}{$ (5) $} \\[\vertrowspace] 
(1^{5}) & 1 & -1 & -1 & 1 & 0 & 1 & -1 & \frac{1}{5} & 0 & 0 & 0 & 0 & 0 & 0 & 0 & 0 & -\frac{1}{5} \\
(1^{3} 1^{2}) & 0 & 1 & 0 & -\frac{1}{2} & -\frac{1}{2} & -1 & \frac{5}{6} & -\frac{1}{6} & \frac{1}{2} & -\frac{1}{6} & -\frac{1}{3} & \frac{1}{6} & 0 & 0 & \frac{1}{6} & 0 & 0 \\
(1^{4} 1) & 0 & 0 & 1 & -1 & 0 & -\frac{1}{2} & 1 & -\frac{1}{4} & 0 & 0 & 0 & 0 & -\frac{1}{2} & \frac{1}{4} & 0 & 0 & 0 \\
(1^{3} 1 1) & 0 & 0 & 0 & \frac{1}{2} & -\frac{1}{2} & 0 & -\frac{1}{2} & \frac{1}{6} & \frac{1}{2} & -\frac{1}{6} & 0 & -\frac{1}{6} & 0 & 0 & \frac{1}{6} & 0 & 0 \\
(2 1^{3}) & 0 & 0 & 0 & 0 & 1 & 0 & 0 & 0 & -1 & \frac{1}{3} & 0 & 0 & 0 & 0 & -\frac{1}{3} & 0 & 0 \\
(1^{2} 1^{2} 1) & 0 & 0 & 0 & 0 & 0 & \frac{1}{2} & -\frac{1}{2} & \frac{1}{8} & -\frac{1}{2} & \frac{1}{4} & 0 & 0 & -\frac{1}{2} & \frac{3}{8} & 0 & \frac{1}{4} & 0 \\
(1^{2} 1 1 1) & 0 & 0 & 0 & 0 & 0 & 0 & \frac{1}{6} & -\frac{1}{12} & -\frac{1}{2} & \frac{1}{6} & \frac{1}{3} & -\frac{1}{6} & 0 & \frac{1}{4} & -\frac{1}{6} & 0 & 0 \\
(1 1 1 1 1) & 0 & 0 & 0 & 0 & 0 & 0 & 0 & \frac{1}{120} & 0 & -\frac{1}{12} & 0 & \frac{1}{6} & 0 & \frac{1}{8} & -\frac{1}{6} & -\frac{1}{4} & \frac{1}{5} \\
(2 1^{2} 1) & 0 & 0 & 0 & 0 & 0 & 0 & 0 & 0 & 1 & -\frac{1}{2} & -1 & \frac{1}{2} & 0 & -\frac{1}{2} & \frac{1}{2} & 0 & 0 \\
(2 1 1 1) & 0 & 0 & 0 & 0 & 0 & 0 & 0 & 0 & 0 & \frac{1}{6} & 0 & -\frac{1}{2} & 0 & -\frac{1}{2} & \frac{5}{6} & 1 & -1 \\
(3 1^{2}) & 0 & 0 & 0 & 0 & 0 & 0 & 0 & 0 & 0 & 0 & 1 & -\frac{1}{2} & 0 & 0 & -\frac{1}{2} & 0 & 0 \\
(3 1 1) & 0 & 0 & 0 & 0 & 0 & 0 & 0 & 0 & 0 & 0 & 0 & \frac{1}{2} & 0 & 0 & -\frac{1}{2} & -1 & 1 \\
(2^{2} 1) & 0 & 0 & 0 & 0 & 0 & 0 & 0 & 0 & 0 & 0 & 0 & 0 & 1 & -\frac{1}{2} & 0 & -\frac{1}{2} & 0 \\
(2 2 1) & 0 & 0 & 0 & 0 & 0 & 0 & 0 & 0 & 0 & 0 & 0 & 0 & 0 & \frac{1}{2} & -1 & -\frac{1}{2} & 1 \\
(3 2) & 0 & 0 & 0 & 0 & 0 & 0 & 0 & 0 & 0 & 0 & 0 & 0 & 0 & 0 & 1 & 0 & -1 \\
(4 1) & 0 & 0 & 0 & 0 & 0 & 0 & 0 & 0 & 0 & 0 & 0 & 0 & 0 & 0 & 0 & 1 & -1 \\
(5) & 0 & 0 & 0 & 0 & 0 & 0 & 0 & 0 & 0 & 0 & 0 & 0 & 0 & 0 & 0 & 0 & 1 \\
\end{array}
\] %*}
    \end{minipage}

\newpage

\section{M\"obius function of types in \texorpdfstring{degrees $\leq 5$}{low degrees}} %{* 

\vspace{-.2cm}
%\fbox{
\begin{minipage}[t][8cm][t]{0.3\textwidth}	
    \vspace{.68cm}
% Degree Two
    \[ %{*
\hspace{10000pt minus 1fil}
\begin{array}{rrrr}
& \rotatebox{90}{ $(1^{2})$ } & \rotatebox{90}{$ (1 1)$ } & \rotatebox{90}{$ (2)$ } \\[\vertrowspace] 
(1^{2}) & 1 & -1 & 0 \\
(11) & 0 & 1 & -1 \\
(2) & 0 & 0 & 1 \\
\end{array}
    \] %*}
    \vspace{.28cm}
% Degree Three
    \[ %{*
\hspace{10000pt minus 1fil}
\begin{array}{rrrrrr}
 & \rotatebox{90}{$ (1^{3}) $} & \rotatebox{90}{$ (1^{2} 1) $} & \rotatebox{90}{$ (1 1 1) $} & \rotatebox{90}{$ (2 1) $} & \rotatebox{90}{$ (3) $} \\[\vertrowspace] 
(1^{3}) & 1 & -1 & 0 & 0 & 0 \\
(1^{2}1) & 0 & 1 & -1 & 0 & 0 \\
(111) & 0 & 0 & 1 & -1 & 0 \\
(21) & 0 & 0 & 0 & 1 & -1 \\
(3) & 0 & 0 & 0 & 0 & 1 \\
\end{array}
    \] %*}
\end{minipage}%
\begin{minipage}[t][8cm][t]{0.63\textwidth}	
% Degree Four
    \vspace{1.23em}
    \[ %{*
\hspace{10000pt minus 1fil}
\begin{array}{rrrrrrrrrrrr}
 & \rotatebox{90}{$ (1^{4}) $} & \rotatebox{90}{$ (1^{3} 1) $} & \rotatebox{90}{$ (1^{2} 1^{2}) $} & \rotatebox{90}{$ (1^{2} 1 1) $} & \rotatebox{90}{$ (1 1 1 1) $} & \rotatebox{90}{$ (2 1^{2}) $} & \rotatebox{90}{$ (2 1 1) $} & \rotatebox{90}{$ (3 1) $} & \rotatebox{90}{$ (2^{2}) $} & \rotatebox{90}{$ (2 2) $} & \rotatebox{90}{$ (4) $} \\[\vertrowspace] 
(1^{4}) & 1 & -1 & -1 & 1 & 0 & 0 & 0 & 0 & 0 & 0 & 0 \\
(1^{3}1) & 0 & 1 & 0 & -1 & 0 & 0 & 0 & 0 & 0 & 0 & 0 \\
(1^{2}1^{2}) & 0 & 0 & 1 & -1 & 0 & 0 & 0 & 0 & -1 & 1 & 0 \\
(1^{2}11) & 0 & 0 & 0 & 1 & -1 & -1 & 1 & 0 & 0 & 0 & 0 \\
(1111) & 0 & 0 & 0 & 0 & 1 & 0 & -1 & 0 & 0 & 0 & 0 \\
(21^{2}) & 0 & 0 & 0 & 0 & 0 & 1 & -1 & 0 & 0 & 0 & 0 \\
(211) & 0 & 0 & 0 & 0 & 0 & 0 & 1 & -1 & 0 & -1 & 1 \\
(31) & 0 & 0 & 0 & 0 & 0 & 0 & 0 & 1 & 0 & 0 & -1 \\
(2^{2}) & 0 & 0 & 0 & 0 & 0 & 0 & 0 & 0 & 1 & -1 & 0 \\
(22) & 0 & 0 & 0 & 0 & 0 & 0 & 0 & 0 & 0 & 1 & -1 \\
(4) & 0 & 0 & 0 & 0 & 0 & 0 & 0 & 0 & 0 & 0 & 1 \\
\end{array}
    \] %*}
\end{minipage}

\vspace{-\parskip}
% Degree Five
    \begin{minipage}[t][10cm][t]{0.9\textwidth}
        \centering
    \vspace{1.5cm}
\[ %{*
\begin{array}{rrrrrrrrrrrrrrrrrr}
 & \rotatebox{90}{$ (1^{5}) $} & \rotatebox{90}{$ (1^{3} 1^{2}) $} & \rotatebox{90}{$ (1^{4} 1) $} & \rotatebox{90}{$ (1^{3} 1 1) $} & \rotatebox{90}{$ (2 1^{3}) $} & \rotatebox{90}{$ (1^{2} 1^{2} 1) $} & \rotatebox{90}{$ (1^{2} 1 1 1) $} & \rotatebox{90}{$ (1 1 1 1 1) $} & \rotatebox{90}{$ (2 1^{2} 1) $} & \rotatebox{90}{$ (2 1 1 1) $} & \rotatebox{90}{$ (3 1^{2}) $} & \rotatebox{90}{$ (3 1 1) $} & \rotatebox{90}{$ (2^{2} 1) $} & \rotatebox{90}{$ (2 2 1) $} & \rotatebox{90}{$ (3 2) $} & \rotatebox{90}{$ (4 1) $} & \rotatebox{90}{$ (5) $} \\[\vertrowspace] 
(1^{5}) & 1 & -1 & -1 & 1 & 0 & 1 & -1 & 0 & 0 & 0 & 0 & 0 & 0 & 0 & 0 & 0 & 0 \\
(1^{3}1^{2}) & 0 & 1 & 0 & -1 & 0 & -1 & 1 & 0 & 0 & 0 & 0 & 0 & 0 & 0 & 0 & 0 & 0 \\
(1^{4}1) & 0 & 0 & 1 & -1 & 0 & -1 & 1 & 0 & 0 & 0 & 0 & 0 & 0 & 0 & 0 & 0 & 0 \\
(1^{3}11) & 0 & 0 & 0 & 1 & -1 & 0 & -1 & 0 & 1 & 0 & 0 & 0 & 0 & 0 & 0 & 0 & 0 \\
(21^{3}) & 0 & 0 & 0 & 0 & 1 & 0 & 0 & 0 & -1 & 0 & 0 & 0 & 0 & 0 & 0 & 0 & 0 \\
(1^{2}1^{2}1) & 0 & 0 & 0 & 0 & 0 & 1 & -1 & 0 & 0 & 0 & 0 & 0 & -1 & 1 & 0 & 0 & 0 \\
(1^{2}111) & 0 & 0 & 0 & 0 & 0 & 0 & 1 & -1 & -1 & 1 & 0 & 0 & 0 & 0 & 0 & 0 & 0 \\
(11111) & 0 & 0 & 0 & 0 & 0 & 0 & 0 & 1 & 0 & -1 & 0 & 0 & 0 & 0 & 0 & 0 & 0 \\
(21^{2}1) & 0 & 0 & 0 & 0 & 0 & 0 & 0 & 0 & 1 & -1 & -1 & 1 & 0 & 0 & 0 & 0 & 0 \\
(2111) & 0 & 0 & 0 & 0 & 0 & 0 & 0 & 0 & 0 & 1 & 0 & -1 & 0 & -1 & 1 & 1 & -1 \\
(31^{2}) & 0 & 0 & 0 & 0 & 0 & 0 & 0 & 0 & 0 & 0 & 1 & -1 & 0 & 0 & 0 & 0 & 0 \\
(311) & 0 & 0 & 0 & 0 & 0 & 0 & 0 & 0 & 0 & 0 & 0 & 1 & 0 & 0 & -1 & -1 & 1 \\
(2^{2}1) & 0 & 0 & 0 & 0 & 0 & 0 & 0 & 0 & 0 & 0 & 0 & 0 & 1 & -1 & 0 & 0 & 0 \\
(221) & 0 & 0 & 0 & 0 & 0 & 0 & 0 & 0 & 0 & 0 & 0 & 0 & 0 & 1 & -1 & -1 & 1 \\
(32) & 0 & 0 & 0 & 0 & 0 & 0 & 0 & 0 & 0 & 0 & 0 & 0 & 0 & 0 & 1 & 0 & -1 \\
(41) & 0 & 0 & 0 & 0 & 0 & 0 & 0 & 0 & 0 & 0 & 0 & 0 & 0 & 0 & 0 & 1 & -1 \\
(5) & 0 & 0 & 0 & 0 & 0 & 0 & 0 & 0 & 0 & 0 & 0 & 0 & 0 & 0 & 0 & 0 & 1 \\
\end{array}
\] %*}
    \end{minipage}

%*} 

\restoregeometry

\bibliography{ipf}

\begin{thebibliography}{10}

\bibitem{adiprasito2018hodge}
K.~Adiprasito, J.~Huh, and E.~Katz.
\newblock {H}odge theory for combinatorial geometries.
\newblock {\em Annals of Mathematics}, 188(2):381--452, 2018.

\bibitem{MR244387}
M.~F. Atiyah and D.~O. Tall.
\newblock Group representations, {$\lambda $}-rings and the {$J$}-homomorphism.
\newblock {\em Topology}, 8:253--297, 1969.

\bibitem{biludashowe}
M.~Bilu, R.~Das, and S.~Howe.
\newblock Zeta statistics and {Hadamard} functions.
\newblock {\em Adv. Math.}, 407:68, 2022.
\newblock Id/No 108556.

\bibitem{bodin}
A.~Bodin.
\newblock Number of irreducible polynomials in several variables over finite
  fields.
\newblock {\em The American Mathematical Monthly}, 115(7):653--660, 2008.

\bibitem{MR2594509}
A.~Bodin.
\newblock Generating series for irreducible polynomials over finite fields.
\newblock {\em Finite Fields Appl.}, 16(2):116--125, 2010.

\bibitem{braden2020singular}
T.~Braden, J.~Huh, J.~P. Matherne, N.~Proudfoot, and B.~Wang.
\newblock Singular {H}odge theory for combinatorial geometries.
\newblock {\em arXiv preprint arXiv:2010.06088}, 2020.

\bibitem{cadogan}
C.~C. Cadogan.
\newblock The {M{\"o}bius} function and connected graphs.
\newblock {\em J. Comb. Theory, Ser. B}, 11:193--200, 1971.

\bibitem{carlitz63}
L.~Carlitz.
\newblock The distribution of irreducible polynomials in several
  indeterminates.
\newblock {\em Illinois J. Math.}, 7:371--375, 1963.

\bibitem{MR0172872}
L.~Carlitz.
\newblock The distribution of irreducible polynomials in several
  indeterminates. {II}.
\newblock {\em Canadian J. Math.}, 17:261--266, 1965.

\bibitem{MR3004008}
E.~Cesaratto, J.~von~zur Gathen, and G.~Matera.
\newblock The number of reducible space curves over a finite field.
\newblock {\em J. Number Theory}, 133(4):1409--1434, 2013.

\bibitem{chen}
W.~Chen.
\newblock Stability of the cohomology of the space of complex irreducible
  polynomials in several variables.
\newblock {\em Int. Math. Res. Not. IMRN}, (22):17256--17276, 2021.

\bibitem{elias2020hodge}
B.~Elias, S.~Makisumi, U.~Thiel, G.~Williamson, B.~Elias, S.~Makisumi,
  U.~Thiel, and G.~Williamson.
\newblock The {H}odge theory of {S}oergel bimodules.
\newblock {\em Introduction to Soergel Bimodules}, pages 347--367, 2020.

\bibitem{farb2019coincidences}
B.~Farb, J.~Wolfson, and M.~M. Wood.
\newblock Coincidences between homological densities, predicted by arithmetic.
\newblock {\em Advances in Mathematics}, 352:670--716, 2019.

\bibitem{FLORENTINO2021104008}
C.~Florentino, A.~Nozad, and A.~Zamora.
\newblock Serre polynomials of $\mathrm{SL}_n$- and $\mathrm{PGL}_n$-character
  varieties of free groups.
\newblock {\em Journal of Geometry and Physics}, 161:104008, 2021.

\bibitem{modularoperads}
E.~Getzler and M.~M. Kapranov.
\newblock Modular operads.
\newblock {\em Compos. Math.}, 110(1):65--126, 1998.

\bibitem{gorsky}
E.~Gorsky.
\newblock Adams operations and power structures.
\newblock {\em Mosc. Math. J.}, 9(2):305--323, 2009.

\bibitem{MR2046199}
S.~M. Gusein-Zade, I.~Luengo, and A.~Melle-Hern\'{a}ndez.
\newblock A power structure over the {G}rothendieck ring of varieties.
\newblock {\em Math. Res. Lett.}, 11(1):49--57, 2004.

\bibitem{heinloth}
F.~Heinloth.
\newblock A note on functional equations for zeta functions with values in
  {C}how motives.
\newblock {\em Ann. Inst. Fourier (Grenoble)}, 57(6):1927--1945, 2007.

\bibitem{MR2516427}
X.-d. Hou and G.~L. Mullen.
\newblock Number of irreducible polynomials and pairs of relatively prime
  polynomials in several variables over finite fields.
\newblock {\em Finite Fields Appl.}, 15(3):304--331, 2009.

\bibitem{hyde2020euler}
T.~Hyde.
\newblock Euler characteristic of the space of real multivariate irreducible
  polynomials.
\newblock {\em Proc. Am. Math. Soc.}, 150(6):2331--2343, 2022.

\bibitem{khanna2024transitionmatricespierityperules}
A.~Khanna and N.~A. Loehr.
\newblock Transition matrices and {P}ieri-type rules for polysymmetric
  functions, 2024.

\bibitem{MR2098401}
M.~Larsen and V.~A. Lunts.
\newblock Rationality criteria for motivic zeta functions.
\newblock {\em Compos. Math.}, 140(6):1537--1560, 2004.

\bibitem{MR4015455}
H.~W. Lenstra.
\newblock Construction of the ring of {W}itt vectors.
\newblock {\em Eur. J. Math.}, 5(4):1234--1241, 2019.

\bibitem{MR143204}
I.~G. Macdonald.
\newblock The {P}oincar\'{e} polynomial of a symmetric product.
\newblock {\em Proc. Cambridge Philos. Soc.}, 58:563--568, 1962.

\bibitem{MR3443860}
I.~G. Macdonald.
\newblock {\em Symmetric functions and {H}all polynomials}.
\newblock Oxford Classic Texts in the Physical Sciences. The Clarendon Press,
  Oxford University Press, New York, second edition, 2015.

\bibitem{mozgovoy2019}
S.~Mozgovoy.
\newblock Motivic classes of quot-schemes on surfaces, 2019.

\bibitem{ramachandran2015zeta}
N.~Ramachandran.
\newblock Zeta functions, {G}rothendieck groups, and the {W}itt ring.
\newblock {\em Bulletin des Sciences Math{\'e}matiques}, 139(6):599--627, 2015.

\bibitem{Stanley}
R.~P. Stanley and S.~Fomin.
\newblock {\em Enumerative Combinatorics}, volume~2 of {\em Cambridge Studies
  in Advanced Mathematics}.
\newblock Cambridge University Press, 1999.

\bibitem{MR2649360}
D.~Yau.
\newblock {\em Lambda-rings}.
\newblock World Scientific Publishing Co. Pte. Ltd., Hackensack, NJ, 2010.

\end{thebibliography}
\bibliographystyle{abbrv}

\end{document}